\RequirePackage{fix-cm}
\documentclass[preprint, envcountsame,envcountsect]{svjour3}                     
\smartqed  
\usepackage{amsmath}
\usepackage{amscd}
\usepackage{graphicx}
\usepackage{comment}

\usepackage{ulem}
\usepackage{url}

\usepackage{color}

\newcommand{\Z}{{\mathbf Z}}
\newcommand{\R}{{\mathbf R}}
\newcommand{\C}{{\mathbf C}}
\newcommand{\HH}{{\mathcal H}}

\newcommand{\BF}{{\cal B}_F^{sa}({\cal H}_0)}
\newcommand{\CF}{{ \cal C}_F^{sa}({\cal H}_0)}
\newcommand{\UF}{{ \cal U}_F({\cal H}_0)}

\DeclareMathOperator{\Ker	}{Ker}
\DeclareMathOperator{\Coker}{Coker}
\DeclareMathOperator{\trace}{trace}
\DeclareMathOperator{\sgn}{sgn}
\DeclareMathOperator{\specflow}{sf}

\title{The index of lattice Dirac operators and $K$-theory}

\author{Shoto Aoki \and Hidenori Fukaya         \and
  Mikio Furuta \and  Shinichiroh Matsuo \and Tetsuya Onogi\and Satoshi Yamaguchi
}

\institute{S. Aoki\at
Graduate School of Arts and Sciences, University of Tokyo
Komaba, Meguro-ku, Tokyo 153-8902, Japan,\\
Interdisciplinary Theoretical and Mathematical Sciences Program (iTHEMS), RIKEN, Wako 351-0198, Japan\\
\email{shotoaoki@g.ecc.u-tokyo.ac.jp}\\
\and
 H. Fukaya, T. Onogi, S. Yamaguchi\at
              Department of Physics, The University of Osaka, Osaka, Japan \\
              \email{hfukaya@het.phys.sci.osaka-u.ac.jp}\\
              \url{http://www-het.phys.sci.osaka-u.ac.jp/~hfukaya/}\\
              \email{onogi@phys.sci.osaka-u.ac.jp}\\
  	      \email{yamaguch@het.phys.sci.osaka-u.ac.jp}
           \and
           M. Furuta \at
           Graduate School of Mathematical Sciences, The University of Tokyo, Tokyo, Japan \\
           \email{furuta@ms.u-tokyo.ac.jp}
           \and
           S. Matsuo \at
           Graduate School of Mathematics, Nagoya University, Nagoya, Japan\\
           \email{shinichiroh@math.nagoya-u.ac.jp}\\
           \url{https://www.math.nagoya-u.ac.jp/~shinichiroh/}
}

\date{\flushright Preprint numbers: OH-HET-1236}

\begin{document}

\maketitle

\abstract{
We mathematically show an equality between the index of a Dirac operator on a flat continuum torus and the $\eta$ invariant of a lattice 
Dirac operator known as the Wilson Dirac operator with a negative mass when the lattice spacing is sufficiently small. Unlike the standard approach, our formulation using $K$-theory does not require modified chiral symmetry on the lattice. We prove that a one-parameter family of continuum massive Dirac operators and the corresponding Wilson Dirac operators belong to the same equivalence class of the $K^1$ group at a finite lattice spacing. Their indices, which are evaluated by the spectral flow or equivalently by the $\eta$ invariant at a finite mass, are proved to be equal.
}


\section{Introduction}

Lattice gauge theory is a useful tool for numerical simulation of particle physics.
Being defined on a discrete lattice, it is, however,
mathematically challenging to describe topological 
properties of the target continuum gauge theory. 
In this work, we prove by $K$ theory \cite{MR0233870,MR0285033} 
that the index of the Dirac operators \cite{Atiyah:1968mp}
can be extracted from a lattice approximation, as known as the Wilson Dirac operator \cite{Wilson1977},
when the lattice spacing is sufficently small.

First, we describe the $K$ groups in such a way that one can combine 
infinite-dimensional Hilbert space in continuum and finite ones on a lattice.
Here our formulation premises the equivalence of the definitions by 
Karoubi \cite{Karoubi} and that by Atiyah-Hirzebruch \cite{Atiyah1961VectorBA}
with slight modifications explained in Appendix \ref{Relation Karoubi}.
The target massless continuum Dirac operators can be treated
as elements of the $K^0$ group, which are characterized by the index 
defined in the standard manner with the $\mathbf{Z}_2$ grading (chirality) operator.
For each continuum Dirac operator, we can consider a one-parameter deformation by the mass term
and treat it as an element of the  $K^1$ group 
defined on a finite interval, identifying  it as the mass parameter region.
By the suspension isomorphism, these $K^0$ and $K^1$ groups are isomorphic,
and the corresponding index and the spectral flow agree.
The key point in this work is to consider the $K^1$ group, rather than the $K^0$ group,
to describe the lattice Dirac operators 
since the $\mathbf{Z}_2$ grading structure is difficult to realize in lattice gauge theory.

Next we define a lattice approximation of a given data in continuum setup.
We introduce the generalized link variables, 
which corresponds to the parallel transport between two points in mathematics. 
These simultaneously determine a gauge connection field in the continuum theory,
and the lattice link variables at arbitrary lattice spacing.
We define the Wilson Dirac operator with the generalized link variables
as an approximation of the continuum Dirac operaotor,
which qualifies to be a natural representative group element of the $K^1$ group.

Then the spectral flow of the continuum Dirac operator and
its lattice version are compared.
We construct a piece-wise polynomial map from 
the Hilbert bundle on a discretized lattice
to the continuum bundle, as well as its adjoint.
With these maps, we can give a gap in the spectrum 
of the lattice-continuum combined Dirac operator.
Our main theorem shows by contradiction that there exists a lattice spacing $a_1$
such that for any lattice spacing $a<a_1$ the $K^1$ element constructed on a lattice
is in the same equivalence class as the
original continuum element, and thus, shares the same spectral flow\footnote{
The equivalence between the index of the overlap Dirac operator
and the spectral flow of the Wilson Dirac operator was also known in physics.
It was found in a rather empirical way at the early stage \cite{Itoh:1987iy}
but later a mathematically rigorous equivalence was 
established \cite{Adams:1998eg} (see also the related works \cite{Kikukawa:1998pd,Luscher:1998kn,Fujikawa:1998if,Suzuki:1998yz}).
But as far as we know, 
the mathematical relevance of the Wilson Dirac operator as the element of 
the $K$ group
has not been discussed. 
}.


Here we refer the readers to two mathematical works with similar motivations to ours.
In \cite{Kubota:2020tpr}, the lattice approximation of analytic indices through
the higher index theory of almost flat vector bundles was established.
The Atiyah-Singer index theorem on lattices was directly formulated and proved in \cite{Yamashita:2020nkf}
using the algebraic index theorem of Nest and Tsygan.

One advantage of our method is that essentially 
the same approach can be applied to various systems.
It is straightforward to generalize our approach
to the Dirac operators having real or quaternion structures,
by considering $KO^i$ or $KSp^i$ groups with general degree $i$.
We can also apply our method to the case with a group acting
on the base manifold like orbifolds \cite{Imai:2022bke},
as well as the case with symmetries with anti-linear operation.
It is also interesting to consider the family version of this work.
See Appendix~\ref{app:generalizations} for more details.

Another interesting application is that to the systems with boundaries.
In the recent publications  \cite{Fukaya:2017tsq,Fukaya:2019qlf,Fukaya:2020tjk,Fukaya:2021sea}
by some of the authors and their collaborators
it was shown in continuum theory that the spectral flow of the massive
domain-wall fermion Dirac operator 
is related to the index defined on a manifold with boundary \cite{Atiyah:1975jf}
with the Atiyah-Patodi-Singer condition imposed.
We do not expect difficulty, either, in 
this generalization to the APS-type indices.
In fact, a positive result was already obtained in a perturbative analysis \cite{Fukaya:2019myi}.

As a final remark, let us summarize the developments 
in physics about the index of Dirac operators.
The difficulty was summarized by the Nielsen-Ninomiya theorem 
\cite{Nielsen:1980rz,Nielsen:1981hk,Nielsen:1981xu} which states 
that if one requires existence of 
the $\mathbf{Z}_2$ grading chirality operator which anticommutes with
a lattice Dirac operator, the Dirac operator
must have unphysical zero modes called doublers.
These doubler modes make the lattice Dirac operator no more elliptic.
One can avoid the doublers and define an index \cite{Hasenfratz:1998ri,Neuberger:1997fp}
by constructing a Dirac operator satisfying 
a relation known as the Ginsparg-Wilson relation \cite{Ginsparg:1981bj},
which realizes a ``modified'' $\mathbf{Z}_2$ grading structure \cite{Luscher:1998pqa}. 
Until recently, mathematically rigorous formulation has, however, been limited to 
those expressed by the heat-kernel regularization \cite{Adams:1998eg},
where the index is evaluated by the geometrical index or winding number.
Moreover, the Ginsparg-Wilson relation is so far limited 
to be on an even-dimensional periodic square lattices, whose continuum limit is a flat torus.
It is, therefore, difficult to formulate it in the case with boundary, for example,
where  the Ginsparg-Wilson relation is known to be
broken by any kind of boundary conditions \cite{Luscher:2006df}.

The rest of the paper is organized as follows.
In Sec.~\ref{sec:Ktheory}, we give a formulation of the $K^0$ and $K^1$ groups
such that the infinite rank Hilbert bundle and
the finite rank bundle are combined.
These are rather well known but
we take detailed steps to explain the $K^0$ and $K^1$ groups
to make the paper self-contained and easy to understand 
for non-mathematician readers.
We still stress that our approach for showing
the equivalence between the two $K$ group elements is not the standard one. 
Then we discuss possible generalization to the system with higher symmetries.
In Sec.~\ref{sec:latticeapproximation}, we explain how to construct
the lattice Wilson Dirac operator from the given continuum data.
We introduce the generalized link variables to 
relate the gauge field configurations in the two Dirac operators.
Then we prove our main theorem in Sec.~\ref{eq:main}.
In Sec.~\ref{eq:relationtoov}, we discuss relation between our work and 
the known results with the overlap Dirac operator, which is a lattice operator
satisfying the Ginsparg-Wilson relation.
In Sec.~\ref{sec:summary}, we give summary and discussion
on possible applications of this work.

\section{Hilbert bundles and $K$ groups}
\label{sec:Ktheory}

In this work, we use a classification of Dirac operators acting on
Hilbert spaces in terms of $K$-theory.
To compare the continuum Dirac operator and its
lattice approximation, we need 
a formulation in such a way that 
the Hilbert spaces with infinite and finite dimensions can be combined 
\cite{MR0233870,MR0285033}.
In Sec.~\ref{sec:K^0} and \ref{sec:K^1}, we describe
the $K^0$ and $K^1$ groups necessary for our main theorem.
Equivalence to the starndard definitions of $K^0$ and $K^1$
is not shown here but it follows from standard arguments.
We discuss possible generalizations to
those with higher symmetries and different degrees in Sec.~\ref{sec:K-applications}.

We take rather detailed steps to define the $K^0$ and $K^1$ groups
in order to make the paper self-contained 
and easy to understand for the readers who are not familiar with $K$-theory.
We emphasize that our approach to show
the equivalence between the two $K$ group elements
from different Hilbert spaces
is not the standard one
(see Definition~\ref{def of sim}).

\subsection{$K^0(X,A)$}
\label{sec:K^0}

First, we give a definition of the $K^0$ group.


\subsubsection{Representative elements of $K^0(X,A)$}

Let $X$ be a compact Hausdorff space and $A$ be a closed subspace of $X$.
We denote the pair by $(X,A)$.

\begin{definition}\label{def of triple}
  A triple $(\HH, h, \gamma)$ of the pair $(X,A)$  is a collection of the data
  $\HH,h$ and $\gamma$ which satisfies the following conditions.
\begin{enumerate}
\item
$\HH$ is a complex Hilbert bundle\footnote{The structure group is the group of unitary operators endowed with the operator norm.} 
over $X$. We denote the fiber at $x \in X$ by $\HH_x$.
\item
$h:\HH \to \HH$ is a family of bounded self-adjoint Fredholm operators $\{h_x:\HH_x \to \HH_x\}_{x \in X}$,
which is continuous with respect to the operator norm.
Namely, when it is expressed by a local trivialization in the neighborhood of $x$, 
$h_x$ is continuous in $x$ with respect to the operator norm.
\item
  $\gamma: \HH \to \HH$ is a family of unitary and self-adjoint operators $\{\gamma_x :\HH_x \to \HH_x\}_{x\in X}$
  which is continuous with respect to the operator norm.
  In particular, each $\gamma_x$ satisfies $\gamma_x^2={\rm id}$ and
  keeps the inner product of the vectors in $\HH_x$.
\item
  $h$ and $\gamma$ anticommute: $\{ \gamma, h\}=0$.
\item 
For $x \in A$, $\Ker h_x=\{0\}$.
\end{enumerate}
In our work, we need to consider unbounded operators, 
with which necessary prescription is given later in Sec.~\ref{sec:remarks}.
\end{definition}

\begin{definition}
  For two triples $(\HH, h,\gamma)$ and  $(\HH',h',\gamma')$ 
  of the pair $(X,A)$,
  we define their direct sum by 
  $$
  (\HH, h,\gamma) \oplus (\HH',h',\gamma')
  :=(\HH \oplus \HH', h \oplus h', \gamma \oplus \gamma').
  $$
\end{definition}

In this paper, we introduce a relation $\sim$ as follows.
\begin{definition}\label{def of sim}
For the pair $(X,A)$, two triples $(\HH,h,\gamma)$ and  $(\HH',h',\gamma')$
are in the relation
$$
(\HH,h,\gamma) \sim(\HH',h',\gamma'),
$$
when there exist another triple $(\hat{\HH},\hat{h},\hat{\gamma})$ for the same pair $(X,A)$
and a continuous one-parameter family of a bounded self-adjoint Fredholm operators
$$
\tilde{h}_t: \HH \oplus \HH' \oplus \hat{\HH} \to  \HH \oplus \HH' \oplus \hat{\HH}
$$
with the parameter space labeled by $0 \leq t \leq 1$,
which satisfy the following conditions.
\begin{enumerate}
\item
$(\hat{\HH},\hat{h},\hat{\gamma})$ is a triple of $(X,X)$. Namely,
for arbitrary $x \in X$, $\Ker \hat{h}_x=\{0\}$.
\item
$
\tilde{h}_0= -h \oplus h' \oplus \hat{h}.
$
\item 
$\tilde{h}_t$ anticommutes with $\gamma\oplus \gamma'\oplus\hat{\gamma}$.
\item
For any $(x,t) \in A \times [0,1]\cup X  \times \{1\}$,
$\Ker (\tilde{h}_t)_x =\{0\}$. This is equivalent to 
state that for the projection $p: X\times [0,1]\to X$ and $\tilde{h}=\{\tilde{h}_t\}$ in the range $t \in [0,1]$,
$( p^*(\HH\oplus \HH' \oplus \hat{\HH}), \tilde{h}, p^* (-\gamma \oplus \gamma' \oplus \hat{\gamma}))$
is a triple of the pair
$
(X\times [0,1],A\times [0,1] \cup X \times \{1\})
$
\footnote{The same (or essentially equivalent) condition can be stated as follows.
For the pair
$
(X\times [0,1],A\times [0,1] \cup X \times \{1\}),
$
there exists a triple whose restriction to
$(X \times \{0\}, A \times \{0\})$ is isomorphic to
$-h \oplus h' \oplus \hat{h}$. 
Since the proof of the equivalence will be rather technical, 
we do not employ this condition.}.
\end{enumerate}
\end{definition}
In the following, when $\HH$ and $\gamma$ are obvious,
we use a simplified notation $h\sim h'$ instead of $(\HH,h,\gamma) \sim(\HH',h',\gamma')$.

\begin{lemma}
  \label{lem:isoK0}
  Let $(\HH,h,\gamma)$ and $(\HH',h',\gamma')$ be two isomorphic triples
  of $(X,A)$, in the sense that there exists an isomorphism
  $S:\HH\stackrel{\cong}{\to}\HH'$ of the Hilbert bundles, satisfying
  $h'=ShS^{-1}$ and  $\gamma'=S\gamma S^{-1}$. Then we have the relation
  $(\HH,h,\gamma)\sim (\HH',h',\gamma')$.
\end{lemma}
\begin{proof}

Let us take $(\hat{\HH}, \hat{h},\hat{\gamma})=(0,0,\rm{id})$\footnote{Here $\hat{\HH}=0$ means the vector space
consisting the zero vector only. In this case, ${\rm Ker} \hat{h}=\{0\}$ is trivially true.}.
Then it is enough to show that setting
$$
\tilde{h}_t:= (1-t)
\left(
    \begin{array}{cc}
        -h& 0 \\
       0   &  h
    \end{array}
  \right)
+t 
\left(
    \begin{array}{cc}
      0  & S^{-1} \\
       S   &   0
    \end{array}
  \right),
$$
$( p^*(\HH\oplus \HH ), \tilde{h}, p^* (-\gamma \oplus \gamma))$
is a triple with respect to
$
(X\times [0,1],A\times [0,1] \cup X \times \{1\})
$.
Since the three operators on $\HH \oplus \HH = \HH \otimes \C^2$
$$
\left(
    \begin{array}{cc}
        -h& 0 \\
       0   &  h
    \end{array}
\right),\quad
\left(
    \begin{array}{cc}
      0  & S^{-1} \\
       S   &   0
    \end{array}\right),\quad
\left(
    \begin{array}{cc}
      -\gamma  & 0 \\
       0   &   \gamma
    \end{array}\right)
$$
all anticommute, so do
$\tilde{h}$ and $p^* (-\gamma \oplus \gamma)$.

When $t=0$ and $x \in A$ the kernel of $\tilde{h}_{t,x}=-h \oplus h$ is trivial.
It is sufficient to show that for arbitrary $t>0$ and $x \in X$
the kernel of $\tilde{h}_{t,x}$ is trivial.
From the anticommutation relation of the two terms in $\tilde{h}_t$,
$$
\tilde{h}_t^2:= (1-t)^2
\left(
    \begin{array}{cc}
        -h& 0 \\
       0   &  h
    \end{array}
  \right)^2
+t ^2
\left(
    \begin{array}{cc}
      0  & S^{-1} \\
       S   &   0
    \end{array}
  \right)^2
=\{(1-t)^2 h^2 + t^2 \} \otimes \left(
    \begin{array}{cc}
      {\rm id} & 0 \\
       0   &   {\rm id}
    \end{array}
  \right)
$$
holds. This is always a positive self-adjoint operator for $t>0$
and thus the kernel is $\{0\}$.

\end{proof}

\begin{lemma}
  Let $(\HH,h,\gamma)$ and $(\HH,h^\prime,\gamma)$ be triples of the pair $(X,A)$. If there exists a continuous one-parameter family of a boundded self-adjoint operator $h_t: \HH \to \HH$ such that $h_0=h'$, $h_1= h$, and $\{\gamma, h_t\}=0$, 
  then the two triples are isomorphic.
\end{lemma}

\begin{proof}
  Let us take $(\hat{\HH}, \hat{h},\hat{\gamma})=(0,0,\rm{id})$ and 
  \begin{align*}
    \tilde{h}_t = \begin{cases}
      \left(
    \begin{array}{cc}
        -h& 0 \\
       0   &  h_{2t}
    \end{array} 
  \right) & (0 \leq t \leq 1/2) \\
  (2-2t)\left(
    \begin{array}{cc}
        -h& 0 \\
       0   &  h
    \end{array} 
  \right)  + (2t-1) 
  \left(
    \begin{array}{cc}
      0  &  \rm{id}\\
       \rm{id}   &   0
    \end{array}
  \right)
  & (1/2 \leq t \leq 1)
    \end{cases}.
  \end{align*} 
This one-parameter family is continuous. It is not difficult to show that $\tilde{h}_t^2$ is positive for $t>1/2$;
  therefore, its kernel is trivial.
\end{proof}

\begin{proposition}
For the pair $(X,A)$, the relation $\sim$ is an equivalence relation.
\end{proposition}
\begin{proof}
  We show that the relation is symmetric, reflective and transitive
  in the next subsections by Lemmas~\ref{reflection}, \ref{symmetric}, and \ref{transitive}.
\end{proof}
With the above proposition, we give the definition 
\begin{definition}\label{def of K0}
  $K^0(X,A)$ is the whole set of the equivalence classes by $\sim$ for the triples of the pair $(X,A)$\footnote{
  Note that the triples themselves do not form a set since the cardinality of a set consisting of the triples
  is not bounded. But the equivalence
  classes form a set. One sees this  by checking for each equivalence class the existence of
  a representation of th form $(\HH,h,\gamma)$ such that $\HH_x$ is a separable Hilbert space,
  whose cardinality is bounded from above.}.
\end{definition}

\subsubsection{Equivalence relation $\sim$}

In this subsection, we show that the relation $\sim$
between two triples is an equivalence relation.

\begin{lemma}\label{reflection} (reflexivity)
  For a triple $(\HH,h,\gamma)$ of the pair $(X,A)$,
  the relation 
$$
(\HH,h,\gamma) \sim(\HH,h,\gamma)
$$
holds.
\end{lemma}

\begin{proof}
  The proof goes in parallel to that of Lemma~\ref{lem:isoK0} but with $S= {\rm id}$.
\end{proof}

\begin{lemma}\label{symmetric} (symmetry)
For two triples $(\HH,h,\gamma)$ and  $(\HH',h',\gamma')$
of the pair $(X,A)$, the relation
$
(\HH,h,\gamma) \sim(\HH',h',\gamma')
$
implies 
$(\HH',h',\gamma')\sim
(\HH,h,\gamma).
$ 
\end{lemma}
\begin{proof}
Changing the sign of $\tilde{h}_t\to -\tilde{h}_t$ and $\hat{h}\to -\hat{h}$
gives the relation  $(\HH',h',\gamma')\sim (\HH,h,\gamma)$.
\end{proof}

\begin{lemma}\label{transitive} (transitivity)
  For three triples $(\HH_1,h_1,\gamma_1)$, $(\HH_2,h_2,\gamma_2)$, $(\HH_3,h_3,\gamma_3)$
  of the pair $(X,A)$,
when 
$
(\HH_1,h_1,\gamma_1) \sim(\HH_2,h_2,\gamma_2)
$
and $(\HH_2,h_2,\gamma_2)\sim
(\HH_3,h_3,\gamma_3) 
$,
the relation 
$(\HH_1,h_1,\gamma_1)\sim
(\HH_3,h_3,\gamma_3) 
$
also holds.
\end{lemma}
\begin{proof}
  In Definition~\ref{def of sim}  we have introduced
  $(\hat{\HH}, \hat{h}, \hat{\gamma})$ and $\tilde{h}$
  to define the relation $(\HH,h,\gamma) \sim(\HH',h',\gamma')$.
  Let $(\hat{\HH}_{ij}, \hat{h}_{ij}, \hat{\gamma}_{ij})$
  and $\tilde{h}_{ij}$ be the corresponding data for
  the relation $(\HH_i,h_i,\gamma_i) \sim(\HH_j,h_j,\gamma_j)$
  where $(i,j)=(1,2)$ or $(2,3)$.
 We write $\tilde{h}_{ij}(t)$ for $\tilde{h}_{ij,t} (t\in [0,1])$.

Let us take another triple $(\hat{\HH}_{13}, \hat{h}_{13}, \hat{\gamma}_{13})$
of the pair $(X,A)$ given by
\begin{eqnarray*}
\hat{\HH}_{13} &:=&\hat{\HH}_{12} \oplus \hat{\HH}_{23} \oplus (\HH_2 \otimes \C^2)
\\
\hat{h}_{13} &:=& \hat{h}_{12} \oplus \hat{h}_{23} \oplus 
 \left(
    \begin{array}{cc}
      0 & {\rm id} \\
       {\rm id}   &   0
    \end{array}
  \right) \\
\hat{\gamma}_{13}
 &:=& \hat{\gamma}_{12} \oplus \hat{\gamma}_{23} \oplus 
 \left(
    \begin{array}{cc}
      -\gamma_2 & 0 \\
       0   &   \gamma_2
    \end{array}
  \right). \\
\end{eqnarray*}
Then $(\hat{\HH}_{13}, \hat{h}_{13}, \hat{\gamma}_{13})$ is actually a triple with respect
to $(X,X)$ since the kernel of $(\hat{h}_{13})_x$ is trivial for any $x \in X$.
As is in the proof of Lemma~\ref{reflection},
let us introduce a one-parameter family $\tilde{h}_{13}(t)$ with $t\in [0,1]$
of the operators on $\HH_1 \oplus \HH_3 \oplus\hat{\HH}_{13}$ in the following way.
\begin{enumerate}
 \item For $t\in [0,1/2]$ put
\[
\hat{h}_{13}(t) := 2t\left[\hat{h}_{12} \oplus \hat{h}_{23} \oplus 
 \left(
    \begin{array}{cc}
      -h_2 & 0 \\
       0   &   h_2
    \end{array}
  \right)\right]+ (1-2t)\hat{h}_{13}.
\]
Then the operator 
$$
 \tilde{h}_{13}(t):=-h_1 \oplus h_3 \oplus \hat{h}_{13}(t)
$$
on  $\HH_1 \oplus \HH_3 \oplus \hat{\HH}_{13}$ is a triple
for the pair $(X,A)$ at any value of $t\in [0,1/2]$.

\item 
Note that $\hat{h}_{13}(1/2)$ can be
interpreted as an operator $\hat{h}_{12} \oplus \hat{h}_{23}
\oplus (-h_2 \oplus h_2)$ on $(\hat{\HH}_{12} \oplus \hat{\HH}_{23})
\oplus (\HH_2 \oplus \HH_2)$. Therefore, we can rewrite
\begin{align*}
 \tilde{h}_{13}(1/2) =-h_1 \oplus h_3 \oplus \hat{h}_{13}(1/2) &= -h_1 \oplus h_3 \oplus (\hat{h}_{12} \oplus \hat{h}_{23}) \oplus 
 (-h_2 \oplus   h_2 )
 \\
 &\cong (-h_1 \oplus h_2 \oplus \hat{h}_{12}) \oplus ( - h_2 \oplus h_3 \oplus \hat{h}_{23}).
\end{align*}
The last form acts on $(\HH_1 \oplus \HH_2 \oplus \hat{\HH}_{12}) \oplus (\HH_2 \oplus \HH_3 \oplus \hat{\HH}_{23})$
and is equal to the operator $\tilde{h}_{12}(0) \oplus \tilde{h}_{23}(0)$.
Note from Lemma~\ref{lem:isoK0} that the isomorphism used here does not change the relation $\sim$.

\item
For $t\in [1/2,1]$ put 
\[
 \tilde{h}_{13}(t) := \tilde{h}_{12}(2t-1)\oplus \tilde{h}_{23}(2t-1),
\]
which acts on $(\HH_1 \oplus \HH_2 \oplus \hat{\HH}_{12}) \oplus (\HH_2 \oplus \HH_3 \oplus \hat{\HH}_{23})$.
Since $\tilde{h}_{12}(1) \oplus \tilde{h}_{23}(1)$ is a triple of the pair $(X,X)$, 
Lemma~\ref{transitive} follows. 
\end{enumerate}
\end{proof}

Generalizing Lemma~\ref{reflection},
we would like to add another lemma below. 

\subsubsection{Sum, product and functoriality of $K^0$}

\subsubsection*{Pullback functor}
Let $X$ and $Y$ be compact Hausdorff spaces and $A$ and $B$ be 
closed subsets of $X$ and $Y$, respectively.
For a continuous map $f:X \to Y$ satisfying $f(A)\subset B$
we abbreviate it by $f:(X,A) \to (Y,B)$.

For a continuous map $f:(X,A) \to (Y,B)$, we have a pullback
from a triple with respect to $(Y,B)$ to the one for $(X,A)$.
Let $f^*$ denote the pullback for the Hilbert bundles and the maps on them.
It is not difficult to show that $f^*$ maintains the 
equivalence relation $\sim$ where each element of the triple for  $(Y,B)$
is pullbacked to that for $(X,A)$.
From this we obtain the map
$$
f^* : K^0(Y,B) \to K^0(X,A), \qquad [(\HH,h,\gamma)] \mapsto  [(f^*\HH, f^* h,f^*\gamma)].
$$
It is also obvious that for two continuous maps $f:(X,A) \to (Y,B)$ and 
$g:(Y,B) \to (Z,C)$ as well as their composite $gf:(X,A) \to (Z,C)$, 
$(gf)^*=f^* g^*$ holds by construction.

Besides we can show that the sum and product structures defined below
are also kept by the pullbacks.
The above observation is summarized as follows.
\begin{proposition}
$K^0$ is a contravariant functor from the
category with pairs $(X,A)$ as the objects 
and continuous maps between them as the morphisms
to the category with rings as the objects and 
homomorphisms as the morphisms.
\end{proposition}

\subsubsection*{Sum}

\begin{definition}
We define the sum on $K^0(X,A)$ by
$$
[(\HH, h,\gamma)] + [(\HH',h',\gamma')]
:=[(\HH, h,\gamma) \oplus (\HH',h',\gamma')].
$$
\end{definition}

\begin{lemma}\label{inverse1}
\begin{enumerate}
\item
$[(\HH,h,\gamma)] +[(0,0,{\rm id})]=[(\HH,h,\gamma)] $.
\item
The sum $+$ satisfies the commutative and associative laws.
\item
$[(\HH, h,\gamma)] + [(\HH, -h,-\gamma)] =[(0,0,{\rm id})]$.
\end{enumerate}
\end{lemma}
\begin{proof}
The first two claims follow from the reflexivity of $\sim$.
In the proof of Lemma~\ref{reflection}, a triple with respect to $(X,X)$
is constructed from $(\HH, h,\gamma)\oplus (\HH, -h,-\gamma)$ for $(X,A)$
with a continuous deformation of the operator $h \oplus (-h)$.
The last claim follows from this.
\end{proof}

Now $K^0(X,A)$ forms an Abelian group with respect to the sum $+$.
We denote the identity element $[(0,0,{\rm id})]$ by $0$.
The inverse element of $[(\HH, h,\gamma)]$ is $[(\HH, -h,-\gamma)]$
From the definition of $\sim$, the Lemma below follows.
\begin{lemma}\label{zero}
For a triple $(\HH,h,\gamma)$ of the pair $(X,A)$,
the necessary and sufficient condition for $[(\HH,h,\gamma)]=0$ 
is that there exists a triple $(\hat{\HH},\hat{h},\hat{\gamma})$
with respect to $(X,X)$ such that the triple
$$(\HH \oplus \hat{\HH},h \oplus \hat{h}, \gamma \oplus\hat{\gamma})$$
for $(X,A)$ 
can be continuously deformed to a triple with respect to $(X,X)$
\footnote{This statement indicates that $\alpha=0$ when $\alpha+0=0$.}.
\end{lemma}

There is another expression for the inverse.
\begin{lemma}\label{inverse2}
$[(\HH, h,\gamma)]+[(\HH, h,-\gamma)]=0$ holds. Namely, 
$[(\HH, h,-\gamma)]$ is the inverse element of $[(\HH, h,\gamma)]$.
\end{lemma}
\begin{proof}
We use Lemma~\ref{zero}. 
The operator in a triple 
$$(\HH, h,\gamma) \oplus (\HH, h,-\gamma)
=\left(\HH \oplus \HH, 
\left(
    \begin{array}{cc}
      h & 0 \\
       0   &   h
    \end{array}
  \right),
\left(
    \begin{array}{cc}
     \gamma & 0 \\
       0   &   -\gamma
    \end{array}
  \right),
  \right)
$$
with respect to $(X,A)$ has the following
one-parameter continuous deformation with $0\leq t \leq 1$,
$$
\tilde{h}_t := (1-t)  
\left(
    \begin{array}{cc}
      h & 0 \\
       0   &   h
    \end{array}
  \right)
  +
t  
\left(
    \begin{array}{cc}
      0& \gamma \\
       \gamma   &  0
    \end{array}
  \right).
$$
Noting 
$$
\left(
    \begin{array}{cc}
      h & 0 \\
       0   &   h
    \end{array}
  \right), \quad
\left(
    \begin{array}{cc}
      0& \gamma \\
       \gamma   &  0
    \end{array}
  \right),\quad
\left(
    \begin{array}{cc}
      \gamma & 0 \\
       0   &   -\gamma
    \end{array}
  \right)
$$
  all anticommute, we can continuously deform
the triple $(\HH, h,\gamma)+(\HH, h,-\gamma)$
to that of $(X,X)$,
since the kernel of $\tilde{h}_1$ is trivial.
\end{proof}


\begin{remark}
From Lemma~\ref{inverse1} and  Lemma~\ref{inverse2},
$$
[(\HH, h,\gamma)]=[(\HH, -h,\gamma)] \in K^0(X,A)
$$
follows.
\end{remark}

\subsubsection*{Product}

\begin{lemma}

Let $A$ and $B$ closed subsets of compact Hausdorff space $X$.
For a triple $(\HH, h,\gamma)$ of the pair $(X,A)$
and $(\HH',h',\gamma')$ with respect to $(X,B)$, 
we define the product below as a triple with respect to $(X,A\cup B)$
\footnote{
The tensor product of two Hilbert bundles $\HH$ and $\HH^\prime $ is a Hilbert bundle whose 
fiber is $(\HH \otimes \HH^\prime)_x =\HH_x \otimes \HH^\prime_x$ for $x \in X$.
The tensor product $\HH_x \otimes \HH_x'$ between the Hilbert spaces $\HH_x$ and $\HH'_x$
is defined by a completion of their algebraically defined tensor product:
the completion is done with the inner product
the algebraic tensor product naturally has.
The algebraic tensor product of the bounded operators on  $\HH_x$ and $\HH'_x$
is canonically extended to the completion. We simply denote it by $\otimes$ between the operators.}.
$$
( \HH \otimes \HH', (h \otimes {\rm id}) + (\gamma \otimes h'), \gamma \otimes \gamma').
$$
Moreover, its equivalence class in $K^0(X,A \cup B)$
depends only on the equivalence classes in $K^0(X,A)$ and $K^0(X,B)$
before taking the product.
\end{lemma}
\begin{proof}
The above product is a triple of the pair $(X ,\emptyset)$,
since the three operators
 $$
 h \otimes {\rm id},\qquad \gamma \otimes h',\qquad \gamma \otimes \gamma'
 $$
 all anticommute. We can also show that
$$
 \{(h \otimes {\rm id}) + (\gamma \otimes h')\}^2= (h \otimes {\rm id})^2 + ({\rm id} \otimes h')^2 
$$
and the right-hand side is a sum of semi-positive self-adjoint operators.
The operator has a kernel only when both $h \otimes {\rm id}$ and ${\rm id} \otimes h'$ have kernels.
Therefore, the kernel is always trivial on $A \cup B$.
The uniqueness of the product is shown as follows.
Fixing $(\HH', h', \gamma')$ we replace $(\HH,h,\gamma)$
by another in the same equivalence class.
For $\hat{h}$'s and $\tilde{h}$'s in definition of the relation $\sim$,
we can directly confirm the equivalence of the two products.
The same applies when we replace $(\HH', h', \gamma')$ by another element 
with $(\HH, h, \gamma)$ fixed.
\end{proof}
The above Lemma defines the product of the $K^0$ groups.
\begin{definition}
Let $A$ and $B$ be closed subsets of a compact Hausdorff space $X$.
The product
$$
K^0(X,A) \times K^0(X,B) \to K^0(X, A\cup B)
$$
is defined by
$$
[(\HH, h,\gamma)]\cdot [(\HH',h',\gamma')] =
[
( \HH \otimes \HH', (h \otimes {\rm id}) + (\gamma \otimes h'), \gamma \otimes \gamma')].
$$
\end{definition}

\begin{lemma}
\begin{enumerate}
\item
The product is associative.
\item
The product and sum are left- and right-distributive.
\item
The element $[(\C, 0, {\rm id})] \in K^0(X ,\emptyset)$ is the identity:
$[(\HH, h,\gamma)]\cdot [(\C, 0, {\rm id})] = [(\C, 0, {\rm id})]\cdot [(\HH, h,\gamma)]=[(\HH, h,\gamma)]$.
\end{enumerate}
\end{lemma}
\begin{proof}
For $\alpha=[(\HH,h,\gamma)] \in K^0(X,A)$, 
$\alpha'=[(\HH',h',\gamma')] \in K^0(X,A')$ and \\
$\alpha''[(\HH'',h'',\gamma'')] \in K^0(X,A'')$,
both of $(\alpha\cdot \alpha')\cdot \alpha''$ and $\alpha \cdot(\alpha' \cdot \alpha'')$ are given by
the same element
$$
[(\HH \otimes \HH' \otimes \HH', h \otimes {\rm id} \otimes {\rm id} +\gamma \otimes h' \otimes {\rm id}
+ \gamma \otimes \gamma' \otimes h'', \gamma \otimes \gamma' \otimes \gamma'')].
$$
The left- and right-distributivity can also be directly checked.
It is also obvious from definition that $[(\C,0, {\rm id})]$ is the identity.
\end{proof}

\begin{lemma}
The product $K^0(X,A) \times K^0(X,B) \to K^0(X,A\cup B)$ is commutative.
Namely, the two  maps 
$K^0(X,A) \times K^0(X,B) \to K^0(X,A\cup B)$
and $K^0(X,B) \times K^0(X,A) \to K^0(X,A\cup B)$ are identical\footnote{
In generalization in Sec.~\ref{sec:K-applications} to $K^i(X,A)$ and $K^j(X,B)$
with nonzero degrees, we need an additional sign depending on $(-1)^{ij}$.
This sign comes from the orientation flip of the finite-dimensional vector space
$V$ which determined the Clifford algebra $Cl(V)$.
}. 
\end{lemma}
\begin{proof}
Take $\alpha=[({\cal H},h,\gamma)] \in K^0(X,A)$ and
$\alpha'=[({\cal H}',h',\gamma')] \in K^0(X,B)$.
Let ${\cal H}={\cal H}^0 \oplus {\cal H}^1$ and
 ${\cal H}'={{\cal H}'}^0 \oplus {{\cal H}'}^1$ be
 the ${\bf Z}_2$-gradings associated with $\gamma$ and $\gamma'$ respectively.
We define an isomorphism
 $
 S:{\cal H} \otimes {\cal H}' \to {\cal H}' \otimes {\cal H}
 $
by switching the left and the right in the tensor product
 together with the multiplication by $(-1)^{ij}$ on ${\cal H}^i \otimes {{\cal H}'}^j$.
 Then we have
 $$
 S(h' \otimes {\rm id}) S^{-1}= \gamma \otimes h',\quad
S( \gamma' \otimes h)S^{-1}= h \otimes {\rm 1},\quad 
S(\gamma' \otimes \gamma)S^{-1}=\gamma\otimes \gamma'
$$
which and Lemma~\ref{lem:isoK0} imply $\alpha' \alpha =\alpha \alpha'$.
\end{proof}

\begin{definition}
Consider two pairs of $(X,A)$ and $(Y,B)$ where
$X$ and $Y$ are compact Hausdorff spaces and 
$A$ and $B$ are their closed subsets $A \subset X$ and $B \subset Y$, respectively.
The projections
$$
p_X: X \times Y \to X, \qquad p_Y: X \times Y \to Y
$$
induce their pullbacks
$$
p_X^*: K^0(X,A) \to K^0( (X,A) \times (Y,B)),\qquad
p_Y^*: K^0(Y,B) \to K^0( (X,A) \times (Y,B)),
$$
by which we define the exterior product
$$
K^0(X,A) \times K^0(Y,B) \to K^0((X,A) \times (Y,B)).
$$
Specifically for $\omega_X \in K^0(X,A)$ and $\omega_Y \in K^0(Y,B)$,
$$
\omega_X \cdot \omega_Y = p_X^* \omega_X \cdot p_Y^* \omega_Y.
$$
Here
$(X,A) \times (Y,B):= (X \times Y, A \times Y \cup X \times B)$.
\end{definition}

\subsubsection{$K^0( \{{\rm pt}\}, \emptyset)$ group and the Fredholm index}

Let us consider the case where $X$ is a point $\{ {\rm pt}\}$ and $A$ is an empty set $\emptyset$.
When $(\HH,h,\gamma)$ is a triple of the pair $(\{ {\rm pt}\}, \emptyset)$,
$\HH$ is a single Hilbert space.
As is shown below, the elements of the $K^0( \{{\rm pt}\}, \emptyset)$
are classified by the Fredholm operator index,
which is topological and stable against  perturbation.


The next statement is well known as a starting point of the notion of the index.
\begin{theorem}(Fredholm index)\label{index}
The following 
abelian group isomorphism 
holds. 
$$K^0( \{ {\rm pt}\}, \emptyset)\cong \Z,\qquad
[(\HH, h,\gamma)] \mapsto \trace \gamma|_{\Ker h}.
$$
The generator corresponds to $[(\C, 0, {\rm id})]$.
\end{theorem}
\begin{proof}
  Here we give an argument according to our formulation of the $K^0$ group.
  Since the constructions in the following points 3 and 4 are the prototype
  of our main theorem (Theorem~\ref{goal}), we give a rather detailed argument.

  Let $(\HH,h,\gamma)$ be a triple of the pair $( \{ {\rm pt}\}, \emptyset)$,
  where $\HH$ is a single Hilbert space. We show by the following (outlines of) four steps
  that for $\nu=\trace \gamma|_{\Ker h}$ we can express
  $[(\HH,h,\gamma)]=\nu [(\C, 0,{\rm id})]$. 

\begin{enumerate}
\item
  First we show that when
$(\HH, h,\gamma) \sim (\HH', h',\gamma')$, we have
  $\trace \gamma|_{\Ker h}=\trace \gamma'|_{\Ker h'}$
  by the following three steps.
\begin{enumerate}
\item For a triple $(\hat{\HH},\hat{h},\hat{\gamma})$ with respect to
  the pair $(X,X)$, we can claim that $\trace \hat{\gamma}|_{\Ker \hat{h}}=0$
  since $\Ker \hat{h}=\{0\}$\footnote{For any $\phi \in \Ker ({\rm id}\pm \hat{\gamma})$, 
  we have $\hat{h}\phi \in \Ker ({\rm id}\mp \hat{\gamma})$ and contribution from $\phi$ and $\hat{h}\phi$
  cancel in the trace.}.
\item We consider $\check{\HH}:=\HH \oplus \HH' \oplus \hat{\HH}$ and operators on it.
  We will see below that for arbitrary bounded self-adjoint Fredholm operator $\check{h}$ on $\check{\HH}$
  which anticommutes with $\check{\gamma}:=-\gamma \oplus \gamma' \oplus \hat{\gamma}$
   and arbitrary positive real number $\Lambda$ under which $\check{h}^2$ has no spectrum 
other than finite eigenvalues with finite degeneracies,
   $\trace \check{\gamma}|_{\Ker \check{h}}$ is 
homotopy invariant on a set of Fledholm operators.

 \item From the above perturbative invariance, we can show that for the operator
   $-h \oplus h' \oplus \hat{h}$
   on $\check{\HH}$, the following trace is zero:
$$
\trace \check{\gamma}|_{\Ker(-h \oplus h' \oplus \hat{h})}
=-\trace \gamma|_{\Ker h} +\trace \gamma'|_{\Ker h'} +\trace \hat{\gamma}|_{\Ker \hat{h}}=0.
$$
\end{enumerate}

\item
  The perturbative invariance is shown as follows.
  When the domain of $h$ is limited to $\Ker(\gamma -{\rm id})$,
    the range is in $\Ker(\gamma +{\rm id})$.
    For this restriction:
    $h_{-+}:\Ker(\gamma -{\rm id}) \to \Ker(\gamma +{\rm id})$,
    we have
$$
 \trace \gamma|_{\Ker h} = \dim \Ker h_{-+} - \dim \Coker h_{-+}.
 $$
The left-hand side is the Fredholm index of $h_{-+}$. One can show that this quantity is invariant under perturbation of the Fredholm operator $h_{-+}$.


\item 


  Since $h$ is a Fredholm operator, $\Ker h$ is finite-dimensional.
  Let us denote the dimension of the $\gamma=\pm 1$ 
      subeigenspace in $\Ker h$ by   $\nu_\pm$.
  We consider another triple $(\HH', h', \gamma')$ 
defined by
$$
(\HH', h', \gamma'):=( \C^{\nu_+} \oplus \C^{\nu_-}, 0,  (+{\rm id}) \oplus (-{\rm id})).
$$
  When $\nu=\trace \gamma|_{\Ker h}=\nu_+ -\nu_-$, we can show that
  $[(\HH', h', \gamma')] = \nu [(\C, 0, {\rm id})]$.

Then if we can show that $(\HH',h',\gamma') \sim (\HH,h,\gamma)$,
we obtain $[(\HH,h,\gamma)]=\nu [(\C,0, {\rm id})]$.
For this, it is sufficient to show that
the bounded self-adjoint Fredholm operator $-h \oplus 0$
on $\HH \oplus \HH'$ can be continuously deformed to
another operator whose kernel is $\{0\}$\footnote{
  There is a more intuitive discussion to show
  $(\HH',h',\gamma') \sim (\HH,h,\gamma)$.
  Let us denote the restriction of domain of
  $\gamma$ onto $\Ker h$ and $(\Ker h)^\perp$ by $\gamma^0$ and $\gamma^\perp$, respectively.
  Also let $h^\perp$ be the restriction of domain of $h$ onto
  $(\Ker h)^\perp$. In this case,
$
[(\HH, h,\gamma)] \sim [(\Ker h,0 , \gamma^0)] + [((\Ker h)^\perp, h^\perp, \gamma^\perp)]
$ holds.
As $\Ker (h)^\perp=\{0\}$, the second term is zero.
Since $(\Ker h,0 , \gamma^0)$ is isomorphic to $(\HH', h',\gamma')$,
the equivalence relation $(\HH',h',\gamma') \sim (\HH,h,\gamma)$ holds.
This argument is easier than what we mention in the main text but
its application to the general $K$ groups with nonzero degrees is limited.
}.


\item
Since the dimension of the eigen subspace of the restriction
of $\Ker h$ onto 
$\gamma =\pm 1$ is $\nu_\pm$,
there exist isomorphisms $f_+$ and $f_-$ from
$\C^{\nu+}$ and $\C^{\nu_-}$ to the respective subspaces.
Let us choose $f_+$ and $f_-$ which keep the metric.
Their sum gives an injective map,
$$
f:=f_+ +f_-:  \HH' \to \HH,
$$
which preserves the metric and $f \gamma' =\gamma f$ holds.
For $0 \leq t \leq 1$, the following deformation is what we need
to show $(\HH',h',\gamma') \sim (\HH,h,\gamma)$:
$$
\left(
\begin{array}{cc}
-h' & 0 \\
0 &  h
\end{array}
\right)
+
t
\left(
\begin{array}{cc}
0 & f^* \\
f & 0
\end{array}
\right).
$$
\end{enumerate}
\end{proof}

The perturbative invariance is also explained as follows.
Let us denote the direct sum of the eigenspaces of $\check{h}^2$
whose eigenvalue are lower than $\Lambda$ by $E_{\check{h}^2 <\Lambda}$.
When $\check{h}^2$ has no spectrum below $\Lambda$
other than finite eigenvalues with finite degeneracies,
the following holds\footnote{This is a consequence
  of linear algebra on a finite-dimensional vector space :
  for a finite vector space $E$ with an inner product
  and an orthogonal transformation $\gamma_E \in {\mathrm End}\, E$
  which satisfies $\gamma_E^2={\rm id}$,
  when a symmetric transformation $h_E\in {\mathrm End}\, E$ anticommutes with $\gamma_E$,
  $\trace \gamma_E=\trace \gamma_E|_{\Ker h_E}$. Note that $h_E^2$ and $\gamma_E$ commute and
  the simultaneous orthogonalization is possible.
  Since $h_E$ is symmetric, the zero mode space of $h_E^2$ or $\Ker h_E^2$ is identical to
  $\Ker h_E$. For nonzero eigenspace with $\Lambda>0$ of $h_E^2$,
  the operation of $h_E$ gives an isomorphism between
  the $+1$ and $-1$ subeigenspaces of $\gamma_E$.
  Therefore, these nonzero eigenvalues do not contribute to the trace of $\gamma_E$.
}.
$$
\trace \check{\gamma}|_{E_{\check{h}^2 <\Lambda}} =\trace \check{\gamma}|_{\Ker \check{h}}.
$$

Suppose $\Lambda$ is not any eigenvalue of $\check{h}$.
Then the dimension of the subeigenspace of $\check{h}^2$
below $\Lambda$ is invariant against any infinitesimal perturbation
of $\check{h}$ with respect to the operator norm.
Moreover, the subeigenspace forms a finite rank
vector bundle on the parameter space of the perturbation.
The eigen decomposition with $\gamma$
leads to a direct-sum decomposition of the vector bundle.
Therefore, the left-hand side of the above equation is invariant
against the perturbation, and so does the right-hand side.

\subsubsection{Some remarks on the use of Hilbert bundles}
\label{sec:remarks}

Here we give some remarks on
our formulation with Hilbert bundles whose dimension is infinite.
\begin{remark} \label{finite dimensional}
\begin{enumerate}
\item
  In a standard formulation of the K theory in the literature  \cite{MR1043170,segal1968equivariant},
  one takes a finite rank vector bundle to define $K^0(X,A)$ rather than an
  infinite-dimensional Hilbert bundle.
  The discussion in this section goes in parallel
  with such finite-dimensional expressions,
  replacing the bounded self-adjoint Fredholm operator by
  a continuous self-adjoint operator on the finite rank vector
  bundle.
  The propositions and discussions in this section are also valid.
  We can prove that the standard finite-dimensional definition of
  $K^0(X,A)$ is isomorphic to that in our infinite-dimensional definition.

\item
  In particular, we can take a representative element using $\HH$ with a finite rank
  for any equivalence class.
  This indicates that $K^0(X,A)$ can be ``approximated'' by finite-dimensional
  vector bundles.
\item
  One advantage of introducing Hilbert bundles is
  that we can easily define the index of differential
  operators which anti-commute with $\gamma$.
\item
  Another advantage of introducing Hilbert bundles is
  found in the extension to $K^i(X,A)$ with $i\neq 0$.
  In such a formulation, there exists a case that
  the group element cannot be represented by data with
  finite rank $\HH$\footnote{
    For example, the nontrivial elements of $K^1(S^1,\emptyset)$ cannot be 
     represented by any finite rank vector bundle. 
  }. 
Therefore, it is useful to introduce the infinite-dimensional Hilbert space
  from the beginning. See Sec.~\ref{sec:K-applications} for the details.
\end{enumerate}
\end{remark}

So far, we have employed a bounded self-adjoint Fredholm operator
family denoted by $h$.
Here we extend the definition of the $K^0$ group
so that unbounded operators like differential operators are allowed.
Note that with this extension, the $K^0(X,A)$ is unchanged  \cite{Hohnhold2010FromMG} as explained below 
and each equivalence class still has a representative element with
  a finite rank $\HH$.

We first explain which topology we use for the space of operators.

Fix a Hilbert space ${\cal H}_0$ and we write
$\BF$ be the set of bounded self-adjoint Fredholm operators on ${\cal H}_0$ endowed with the norm topology.
Let $\CF$ be the set of closed self-adjoint Fredholm operators with dense domain, which are not necessary bounded.
We make use of the injective map $\rho: \CF \to \BF$  given by $\rho(h) :=h/(1+h^2)^{1/2}$, which is bounded and
its parity is odd such that $\rho(-h)=-\rho(h)$. 
The Riesz topology of $\CF$ is the weakest topology for $\rho$ to be continuous\footnote{The function $\rho$ could be any bounded continuous monotonically increasing function with odd parity such that $\rho(-h)=-\rho(h)$. 
It is in general nontrivial whether the sum of unbounded self-adjoint operators is ill-defined or not \cite{PEDERSEN1990428}.
With the map $\rho(h)$, we do not have this problem.
}.
We endow $\CF$ with this Riesz topology to formulate the continuity of operators\footnote{
Let $\UF$ be the set of unitary operators $u$ such that $u-1$ is Fredholm
endowed with the norm topology. 
If we use $\kappa:{\CF} \to {\UF}$ defined by the Cayley transform $\kappa(h)=(1+ i h)/(1- ih ) $, instead of $\rho$,
we could endow ${\CF}$ with the weakest topology for $\kappa$ to be continuous. 
This is the gap topology, which is strictly weaker than the Riesz topology.
As we explain in Appendix we could use $\UF$ to give a formulation of the $K^i$ groups. The two formulations, one using $\BF$ and the other using $\UF$ give the canonically isomorphic $K$ groups (under correct choice of Clifford symmetries).
 If we formulate the $K$ groups using $\UF$, it would be natural  to use the gap topology for $\CF$. We, however, chose a formulation using $\BF$ in this paper, and  it seems natural to choose the Riesz topology for $\CF$ here. 
}. 
\begin{proposition}
In the definition of $K^0(X,A)$ if we replace  ``bounded self-adjoint Fredholm operators'' with
 ``closed self-adjoint Fredholm operators with dense domain'',   then we obtain the same $K^0(X,A)$ group.
 Here we use the Riesz topology to formulate
 the continuity of families of closed self-adjoint Fredholm operators with dense domains.
\end{proposition}

\begin{proof}

From a standard argument of functional analysis,
the restriction map \\
$\rho|_{\BF}:\BF \to \BF$ is continuous and homotopic to the identity of $\BF$
linearly\footnote{It implies that the inclusion map $\BF \to \CF$ is a homotopy equivalent map for the Riesz topology (see \cite{lesch2004uniquenessspectralflowspaces}).}.
In particular for any triple $({\cal H},h,\gamma)$ with $h$ bounded, we have another triple
$({\cal H}, \rho(h), \gamma)$ with $\rho(h)$ bounded again. Note that
$[({\cal H},\rho(h),\gamma)]=[({\cal H},h,\gamma)]$ in the $K^0$ group.
Moreover, even when $h$ is not bounded, we always have a triple $({\cal H},\rho(h),\gamma)$ with $\rho(h)$ bounded, and
the class $[({\cal H},h,\gamma)]$ is a well-defined element of the $K^0$ group.
It implies that using the class $[({\cal H},h,\gamma)]$ we can extend the definition to unbounded operators $h$.

\end{proof}

\subsection{$K^1(X,A)$}
\label{sec:K^1}

In this section, we formulate the $K^1(X,A)$ group.

\subsubsection{Definition of $K^1(X,A)$ and its properties}

We can define $K^1(X,A)$ by the same steps in the definition of
$K^0(X,A)$ shown in the previous section but forgetting about
existence of $\gamma$.
\begin{definition}
$K^1(X,A)$ is defined by
$$
K^1(X,A):=\{ (\HH,h)\} / \sim
$$
\begin{itemize}
\item
  $(\HH, h)$ is a data which satisfies the following conditions\footnote{Note that $\{ (\HH,h)\}$ is not a set but its equivalence class is a set.}.
\begin{enumerate}
\item
$\HH$ is a Hilbert bundle over $X$.
\item
  $h:\HH\to \HH$ is a family 
of bounded self-adjoint Fredholm operators $h_x:\HH_x \to \HH_x$ ($x \in X$), 
which is continuous with respect to the operator norm.
\item 
  $\Ker h_x=\{0\}$ for $x \in A$.
\end{enumerate}
\item
The equivalence relation $\sim$ is defined as follows.
The relation $
(\HH,h) \sim(\HH',h')
$
holds when there exist
$(\hat{\HH},\hat{h})$ and a continuous family of bounded
self-adjoint Fredholm operators
with the parameter $0 \leq t \leq 1$ given by
$$
\tilde{h}_t: \HH \oplus \HH' \oplus \hat{\HH} \to  \HH \oplus \HH' \oplus \hat{\HH}
$$
and satisfy
\begin{enumerate}
\item
For arbitrary $x \in X$, $\Ker \hat{h}_x=\{0\}$.
\item
$
\tilde{h}_0= -h \oplus h' \oplus \hat{h}
$
\item
For any $(x,t) \in A \times [0,1]\cup X  \times \{1\}$, 
$\Ker (\tilde{h}_t)_x =\{0\}$.
\end{enumerate}
\end{itemize}
\end{definition}
We can prove that $\sim$ is an equivalence class
in a similar manner to $K^0(X,A)$.

For a continuous map $f:(X,A) \to (Y,B)$, the pullback is defined as
$$
f^* : K^1(Y,B) \to K^1(X,A), \qquad [(\HH,h)] \mapsto  [(f^*\HH, f^* h)].
$$
This $f^*$ is compatible with the sum and product defined below.

The sum on $K^1(X,A)$ is defined in the same way as
for $K^0(X,A)$.
\begin{definition}[Sum]
The sum on $K^1(X,A)$ is defined by
$$
[(\HH,h)] +[(\HH',h')]:= [(\HH,h) \oplus (\HH',h')].
$$
\end{definition}
The following product is also well defined.
\begin{definition}[Product]
  For closed subsets $A$ and $B$ in a compact Hausdorff space $X$,
  the product
$$
K^0(X,A) \times K^1(X,B) \to K^1(X, A\cup B)
$$
is defined by
$$
[(\HH, h,\gamma)]\cdot [(\HH',h')] =
[
( \HH \otimes \HH', (h \otimes {\rm id}) + (\gamma \otimes h'))].
$$
\end{definition}
From this definition, the proposition below follows.
\begin{proposition}
\begin{enumerate}
\item
  $(X,A) \mapsto K^1(X,A)$ is a contravariant functor
  with respect to the pullback.
\item
  When $B \supset A$, the above product
  $$
K^0(X,A) \times K^1(X,B) \to K^1(X, B)
$$
indicates that $K^1(X,B)$ becomes a module
with the action of $K^0(X,A)$, which forms a commutative ring.
\end{enumerate}
\end{proposition}

\begin{definition}[Exterior product]
  Consider two pairs of $(X,A)$ and $(Y,B)$ where
$X$ and $Y$ are compact Hausdorff spaces and 
  $A$ and $B$ are their closed subsets
  $A \subset X$ and $B \subset Y$, respectively.
The projections
$$
p_X: X \times Y \to X, \qquad p_Y: X \times Y \to Y
$$
induce their pullbacks
$$
p_X^*: K^0(X,A) \to K^0( (X,A) \times (Y,B)),\qquad
p_Y^*: K^1(Y,B) \to K^1( (X,A) \times (Y,B))
$$
by which the exterior product
$$
K^0(X,A) \times K^1(Y,B) \to K^1((X,A) \times (Y,B))
$$
is defined. 
Specifically for $\omega_X \in K^0(X,A)$ and $\omega^1_Y \in K^1(Y,B)$,
$$
\omega_X \cdot \omega^1_Y = p_X^* \omega_X \cdot p_Y^* \omega^1_Y.
$$
Here
$(X,A) \times (Y,B):= (X \times Y, A \times Y \cup X \times B)$.
\end{definition}

\subsubsection{The suspension isomorphism between $K^0$ and $K^1$}
\begin{definition}[Bott element]
We define the Bott element in $K^1(D^1,S^0)$ by
$$
\beta:=[(\underline{\C}, x\cdot )] \in K^1(D^1,S^0),
$$
where $\underline{\C}$ denotes a trivial line bundle over $D^1$
with the fiber $\C$, and the operator $x \cdot$ is
multiplication of $x$, which is the coordinate for
$D^1=\{x:  -1 \leq x \leq 1\}$.
\end{definition}
\begin{definition}[Suspension]
  For $\alpha \in K^0(X,A)$, the exterior product with the
  Bott element 
$\alpha \cdot \beta \in K^1((X,A) \times (D^1,S^0))$ is called the suspension.
\end{definition}
The explicit form of  the exterior product is given by
\footnote{Note that $h+ x \gamma$ in the right-hand side is
  an abbreviation. Precisely it is written as
$
 p_X^* h  + ({p_{D^1}}^* x)  p_X^* \gamma
$ making the pullback explicit.},
$$
[(\HH,h,\gamma)] \cdot \beta =[(p^*_X \HH, h + x \gamma)] \in K^1((X,A) \times (D^1,S^0)).
$$
The next theorem is a part of the Bott periodicity theorem. 
\begin{theorem}[Suspension isomorphism/Bott periodicity]\label{suspension}
  The following suspension map is an isomorphism
  with respect to the $K^0(X,A)$ module.
$$
K^0(X,A) \to K^1((X,A) \times (D^1,S^0)),\qquad  \alpha \to \alpha \cdot \beta
$$
Here $\beta$ is the Bott element.
\end{theorem}

Our formulation of the $K$ group follows \cite{Karoubi} of which details are discussed in Appendix~\ref{Relation Karoubi}. This theorem can be proved in the same way as the finite-dimensional case.

We will give a detailed description for
the case with $(X,A)=(\{{\rm pt}\}, \emptyset)$ later\footnote{
  The full statement of Theorem~\ref{suspension} for $(X,A)$ will be necessary when
  we extend the discussion on ``comparison between
  the Fredholm index of a Dirac operator
  and its discrete approximation'' to a family version.
}.

\subsubsection{$K^1(D^1,S^0)$ and spectral flow}

Consider $K^1(X,A)$ for the case where $X$ is $D^1=\{ x:-1\leq x \leq 1\}$, and
$A$ is $S^0=\{ -1,+1\}$\footnote{
  Since $D^1$ is contractible, the complex Hilbert bundle over $D^1$ is trivial.
  Note in our formulation that a finite rank Hilbert bundle is
  allowed, to which the Kuiper theorem \cite{MR0179792} guaranteeing
that the space of unitary operators on any 
infinite-dimensional Hilbert space is contractible,
does not apply. 
  We also note that it is not difficult to construct a global
  trivialization of the Hilbert bundle over $D^1$ from
  its local trivializations, without using the Kuiper theorem.
  When one trivialization is fixed, $h:\HH \to \HH$
  is nothing but a one-parameter family of the bounded self-adjoint Fredholm operators
  on a unique complex Hilbert space,
  which is continuous with respect to the operator norm. 
}.

From Theorem~\ref{suspension} and Theorem~\ref{index} we obtain
$$
K^1(D^1,S^0) \cong \Z 
$$
with  the Bott element $\beta$ as its generator.

Let us discuss its geometrical aspects.
In our formulation of $K^1(D^1,S^0)$,
it is sufficient to characterize it by the spectral flow
of a one parameter family of the self-adjoint Fredholm operators \cite{boossbavnbek2004unboundedfredholmoperatorsspectral}.
Recall that the spectral flow $\specflow (\HH,h) \in \Z$
is defined as follows.

\begin{enumerate}
\item

  Since $h$ is a bounded self-adjoint Fredholm operator and $D^1$ is compact,
  there exists $\Lambda_0$ such that for any $x \in D^1$, $h_x$ has
  no spectrum in the range $\{ \lambda: -\Lambda_0 \leq \lambda \leq \Lambda_0 \}$
  except for finite eigenvalues with finite degeneracies.

\item 

  Note that $h_x$ has no zero mode at $x= \pm 1$.
  Let us introduce a finite number of points 
  $x_0=-1 < x_1 < \cdots x_n=+1$ in $D^1$ for which
  we assign the values $\lambda_1, \ldots,\lambda_n$ 
  in the range
  $\{ \lambda: -\Lambda_0 < \lambda < \Lambda_0 \}$ such that
\begin{enumerate}
\item
$\lambda_1=\lambda_n=0$.
\item
For any $x \in D^1 $ in the range $x_{k-1} \leq x \leq x_k$, 
$\lambda_k$ is not an eigenvalue of $h_x$.
\end{enumerate}

\item

  For the $k$-th set $(x_k,\lambda_k)$ for $0<k<n$,
  we assign $\sgn_k$ and $d_k$ as follows.
  For $\lambda_k \neq \lambda_{k+1}$, we set
  \[
  \sgn_k = \frac{\lambda_k - \lambda_{k+1}}{|\lambda_k - \lambda_{k+1}|},
  \]
  and $d_k$ by sum of dimensions of the eigenspace
  with the eigenvalues in the range between $\lambda_k$ and $\lambda_{k+1}$.
  When $\lambda_k=\lambda_{k+1}$, we assign that $\sgn_k=0$ and $d_k=0$.
\end{enumerate}


\begin{definition}[Spectral flow]
  Let $\HH$ be a complex Hilbert bundle over $D^1$
  and $h$ be a continuous family of bounded self-adjoint Fredholm operators on $\HH$.
  The spectral flow of $h$ is defined by
  $$\specflow (\HH,h)=\sum_{0<k<n} \sgn_k d_k.$$
  we denote
  $
\specflow \alpha :=\specflow (\HH,h)
$
for $\alpha=[(\HH,h)] \in K^1(D^1,S^0)$.
\end{definition}

\begin{lemma}
  The spectral flow
  is independent of the choice of
  the sets $\{ x_k\}$ and $\{ \lambda_k\}$.
\end{lemma}

\begin{remark}
  \label{remark:unbounded}
  When $h$ is an unbounded self-adjoint Fredholm operator
  such that $h/(h^2+1)^{1/2}$ is a continuous family with respect to the operator norm,
  the spectral flow is defined in the same way as above.
\end{remark}

For the Bott element $\beta$,
$$
\specflow \beta=1.
$$
For the operator $x\cdot $ in $\beta$, we can set
  $x_0=-1 <x_1=-1/2<x_2=1/2< x_3=+1$ and assign 
  $\lambda_1=0, \lambda_2=2/3, \lambda_3=0$. Then
  we have $\sgn_1=-1, N_1=0$ and $\sgn_2=+1, N_2=1$.

The next theorem is well known.
\begin{theorem}\label{K and spectral flow}
The isomorphism below holds.
$$K^1(D^1,S^0)\cong \Z,\qquad
[(\HH, h)] \mapsto \specflow (\HH, h).
$$
\end{theorem}
\begin{proof}
We give an argument according to our formulation of the $K^1$ group.
  It is sufficient to show that $[(\HH, h)]= \nu\beta$
  when
  $\specflow(\HH,h)=\nu$.
Here we summarize the outline of the proof.

\begin{enumerate}

\item  

  First, we confirm that $\specflow$ is well-defined as a map from
  $K^1(D^1,S^0)$.
  This is shown by the invariance of $\specflow h$ against a continuous deformation of $h$.
  It is not difficult to show that for an infinitesimal variation of $h$,
  the same set $\{ x_k\}$ and $\{ \lambda_k\}$ can be used to compute the
  spectral flow, which is invariant.

\item 
  Next, we deform $h$ in $(\HH,h)$ as follows.
  
  The data for computing the spectral flow is taken as
  $\Lambda_0$, $x_0=-1 < x_1 < \cdots x_n=+1$ and $\lambda_1, \ldots,\lambda_n$.
  We introduce a continuous function $\mu_1(x)$
  which sufficiently approximates the piece-wise constant function
  $\mu_0(x)= \lambda_k (x_{k-1} \leq  x < x_k)$.
  We also introduce a monotonically increasing continuous
  function which also continuously depends on $x$,
  $\rho_x: \R \to \R$ such that $\rho_x(\lambda)= \lambda$ for
  $|\lambda| \geq \Lambda$ and $\rho_x(\mu_1(x))=0$.
  
  We can continuously deform the family $\{ h_x \}$ to $\{ \rho_x(h_x)\}_x$.
For a sufficiently precise approximation of $\mu_0$ by $\mu_1$,
\begin{enumerate}
\item
  Any $x$ such that $\Ker \rho_x(h_x) \neq 0$
  is located at an sufficiently close point to one of $x_k$'s.
  
\item
In the vicinity of $x_k$, the eigenvalues of $\rho_x(h_x)$ change as follows.
For $\sgn_k=\pm 1$, $d_k$ eigenvalues change their sign from $\pm$ to $\mp$.
When $\sgn_k=0$ no eigenvalue goes to zero.
\end{enumerate}

\item 

  There is a representative element  $(\HH',h')$
  which has a similar $x$-dependence of the eigenspectrum to that of $(\HH,h)$:
$$
(\HH',h'):=\left( \bigoplus_{0 < k < n} \underline{\C^{d_k}}, \bigoplus_{0<k<n} \sgn_k  (x -x_k){\rm id}_{d_k \times d_k}\cdot \right). 
  $$
  For $k$ with nonzero $\sgn_k=\pm 1$,
  the $d_k$ eigenvalues of $h'$ crosses zero at $x=x_k$
  from positive/negative to negative/positive.
When $\specflow(\HH,h)=\nu$, 
$[(\HH',h')]=\nu \beta$ holds.

\item 

  The last step is to show that $ (\HH',h')\sim(\HH, \{\rho_x(h_x)\}) $
  for the proof of $[(\HH, h)]= \nu \beta$.
  For this, we consider a deformation of
  the continuous family of the bounded self-adjoint Fredholm operators $\{ -h'_x\oplus  \rho_x(h_x) \}_{x \in D^1}$
  on $\HH' \oplus \HH$ to another operator family whose kernel is trivial everywhere in $D^1$.

  Note that nontrivial kernel appears only in the neighborhood of the discrete points
  $x_k \in D^1$. It is, thus, sufficient to deform the operators so that its kernel becomes trivial
  in the neighborhood of $x_k$. The direct sum of the near-zero eigenspaces of $\rho_x(h_x)$
  near $x_k$ which potentially develop the kernel has $d_k$ dimensions.
  Let $f_k$ be an isomorphism from $\C^{d_k}$ to the direct sum of the $d_k$ eigensections 
such that
  $f_k$ is continuous with respect to $x$ and keeps the metric.
  Note that $f_k$ is defined in the neighborhood of $x_k$ only
  but we can naturally extend it by zero with a continuous cut-off
  function to the whole range of $x\in D^1$.
  There is no overlap of the support of $f_k$'s at different $k$'s.
  We put the sum
$$
f:=\sum_{0< k < n} f_k :\HH' \to \HH.
$$
For $0 \leq t \leq 1$, we can define the operator family
$$
\left(
\begin{array}{cc}
-h'_x & 0 \\
0 &  \rho_x(h_x)
\end{array}
\right)
+
t
\left(
\begin{array}{cc}
0 & f^* \\
f & 0
\end{array}
\right),
$$
which achieves the  claim.
\end{enumerate}
\end{proof}

\subsection{$K$ and $KO$ groups with general degrees}
\label{sec:K-applications}

In the previous subsections, we have defined
the $K^0$ and $K^1$ groups, which are 
sufficient to give a proof of our main theorem.
Here we discuss possible generalizations
to Hilbert bundles with higher symmetries.

Note that the operator $\gamma$ in the triple $(\HH,h,\gamma)$ can be
regarded as a single Clifford generator which satisfies $\{\gamma, h\}=0$. 
In order to define $K^i(X,A)$ with a general degree $i$,
  one can use more general Clifford symmetry on $\HH$
which satisfies appropriate anti-commutativity with $h$,
and their appropriate equivalence relation $\sim$.

Next we consider the case where the Hilbert bundle is real.
In this case, the $KO^0$ and $KO^1$ groups can be
formulated exactly in the same way as in the previous 
subsections with the simple replacement of $\C$ by $\R$
in the definitions of the identity element or the Bott element.
We also denote $KO^{0,0}(X,A)=KO^0(X,A)$ and
$KO^{0,-1}(X,A)=KO^1(X,A)$.

As the Bott periodicity becomes more nontrivial than
the complex $K$ groups, it is worth giving the explicit forms of 
the $KO$ groups with higher degrees.
Here we employ a formulation in \cite{MR0233870}.

Let us first formulate the $KO^{1,0}$ group, which turns out to be isomorphic to $KO^1=KO^{0,-1}$.
\begin{definition}
$$
KO^{1,0}(X,A) := \{(\HH^{\R},h, (\epsilon,e)) \} /\sim
$$
where
\begin{enumerate}
\item
$\HH^{\R}$ is a real Hilbert bundle over $X$.
\item
$h:\HH^{\R} \to \HH^{\R}$ is a family of bounded self-adjoint Fredholm operators parametrized by $x\in X$ 
which is continuous with respect to the operator norm.
\item $\epsilon,e :\HH^{\R} \to \HH^{\R}$ are continuous families of orthogonal operators which satisfy $\epsilon^2={\rm id},e^2=-{\rm id}$. $h, \epsilon,e$ all anticommute among each other.
\item For arbitrary $x \in A$, $\Ker h_x =\{0\}$.
\end{enumerate}
\end{definition}
The definition of the equivalence relation $\sim$ is given in 
the same way as that of the $K$ groups.
In particular, any deformation keeps the condition that $h$ anticommutes with both of $\epsilon$ and $e$.

\begin{lemma}
\label{lem:KO1KO10}
The following natural isomorphism exists.
$$KO^{0,-1}(X,A) \cong KO^{1,0}(X,A).$$
\end{lemma}
\begin{proof}
The isomorphism is given by the map
$$
KO^{0,-1}(X,A) \longrightarrow KO^{1,0}(X,A),
\quad
[(\HH^{\R}_0,h')] \mapsto [(\HH^{\R}_1, h, (\epsilon, e))]
$$
$$
\HH^{\R}_1:=\HH^{\R}_0 \otimes \R^2,\quad
h:=h' \otimes
\left(
\begin{array}{cc}
{\rm id} &  0\\
 0 & -{\rm id}
\end{array}
\right),\quad
\epsilon:=
\left(
\begin{array}{cc}
0 & {\rm id} \\
{\rm id} & 0
\end{array}
\right),\quad
e:=
\left(
\begin{array}{cc}
0 & -{\rm id} \\
{\rm id} & 0
\end{array}
\right),
$$
and its inverse,
$$
KO^{0,-1}(X,A) \longleftarrow KO^{1,0}(X,A),
\quad
[(\HH^{\R}_0,h')] \leftarrow [(\HH^{\R}_1, h, (\epsilon, e))]
$$
$$
\HH^{\R}_0:=\Ker(\epsilon e -{\rm id} :\HH^{\R}_1 \to \HH^{\R}_1),\quad
h:=h'|_{\HH^{\R}_0}. \quad
$$
\end{proof}

Now let us generalize the definition in the following way.
\begin{definition}[$KO^{p,q}(X,A)$ for general $p$ and $q$]

For $p\geq 0,q\geq -1$, we can define\footnote{
When $q=-1$, we do not have any $\epsilon_i$, {\it i.e.},
$KO^{p,-1}(X,A):=\{(\HH^{\R},h, (e_1,\ldots, e_p)) \} /\sim$.
}
$$
KO^{p,q}(X,A):=\{(\HH^{\R},h, (\epsilon_0,\epsilon_1,\ldots,\epsilon_q ,e_1,\ldots, e_p)) \} /\sim.
$$
\end{definition}

\begin{lemma}
 There is an isomorphism 
$$
KO^{p,q}(X,A) \stackrel{\cong}{\to} KO^{p+1,q+1}(X,A).
$$
\end{lemma}
With this isomorphism, we can define the $KO^i$ group\footnote{To be precise, the definition is given as a direct limit, $p,q\to \infty$ with $p-q=i$ fixed.}
 in a similar way to  Lemma~\ref{lem:KO1KO10}.
\begin{definition}
 $$
KO^i (X,Y) := KO^{p,q}(X,Y)\qquad (i=p-q).
$$
\end{definition}

The product for $KO^*$ is constructed in a similar way as that for
$K^0$  \cite{Karoubi}.
The well-definedness and the associativity of the  product are shown
straightforwardly.
The exterior product is also defined via pullbacks:
$$
KO^{i}(X,A) \times KO^{j}(Y,B) \to KO^{i+j}((X,A) \times (Y,B)).
$$
The commutativity of the product holds with sign $(-1)^{ij}$.

Note that the inverse element of 
$[(\HH^{\R},h, (\epsilon_0,\epsilon_1,\ldots,\epsilon_q ,e_1,\ldots, e_p))]$
is given by 
\begin{align*}
 [(\HH^{\R},-h, (-\epsilon_0,-\epsilon_1,-\epsilon_2,\ldots,-\epsilon_q ,-e_1,\ldots, -e_p))].
\end{align*}
For $q\ge 0$, the inverse element is also described in a similar way to Lemma~\ref{inverse2} by
\begin{align*}
[(\HH^{\R},h, (-\epsilon_0,\epsilon_1,\epsilon_2, \ldots,\epsilon_q ,e_1,\ldots, e_p))].
\end{align*}

The suspension isomorphism for  the $KO$ groups is given as follows \cite{Karoubi}.
\begin{definition}[Bott element]
We define the Bott element in $KO^{0,-1}(D^1,S^0)=KO^1(D^1,S^0)$ by
$$
\beta_{\R}:=[(\underline{\R}, x\cdot )] \in KO^{0,-1}(D^1,S^0)=KO^1(D^1,S^0),
$$
where $\underline{\R}$ denotes a trivial line bundle over $D^1$
with the fiber $\R$, and the operator $x \cdot$ is
multiplication of $x$, which is the coordinate for
$D^1=\{x:  -1 \leq x \leq 1\}$.
\end{definition}
\begin{definition}[Suspension]
  For $\alpha \in KO^i(X,A)$, the exterior product with the
  Bott element 
$\alpha \cdot \beta_{\R} \in KO^{i+1}((X,A) \times (D^1,S^0))$ is called the suspension.
\end{definition}
The explicit form of the exterior product for $\alpha \in KO^{p,q}(X,A)$ with
$q\ge0$
is given by
\begin{align*}
 [(\HH^{\R},h,\epsilon_0,\cdots,\epsilon_q,e_1,\cdots, e_p)] \cdot \beta_{\R} 
=&[(p^*_X \HH^{\R}, h + x \epsilon_0, \epsilon_1,\cdots,\epsilon_q,e_1,\cdots, e_p)] \\
 &\in KO^{p,q-1}((X,A) \times (D^1,S^0)).
\end{align*}
The next theorem is a part of the Bott periodicity theorem. 
\begin{theorem}[Suspension isomorphism/Bott periodicity]\label{suspensionKO}
  The following suspension map is an isomorphism
  with respect to the left $KO^{*}(X,A)$ module.
$$
KO^{i}(X,A) \to KO^{i+1}((X,A) \times (D^1,S^0)),\qquad  \alpha \to \alpha \cdot \beta_{\R}.
$$
Here $\beta_{\R}$ is the Bott element.
\end{theorem}


The definitions of $K^i(X,A)$ or $KO^i(X,A)$ for $i\neq 0$ with Clifford algebra symmetry
  also allow unbounded operators. Note that each equivalence class 
  does NOT always have representative elements with a finite rank $\HH$ or $\HH^{\R}$.

\section{Lattice approximation of Dirac operator and the main theorem}
\label{sec:latticeapproximation}

In this section, we present our main theorem in this work
where we compare two Dirac operators, one is
defined on a continuum torus, and the other is
a lattice approximation\footnote{
It is challenging to consider a lattice approximation of
curved manifolds since it is nontrivial to describe their curvature.
In this work, we focus on the square lattices, which
naturally approximate flat tori.
}.
We construct the two Dirac operators and treat
them together with the associated Hilbert bundles
as the elements of the $K^1$ group.
The main theorem states that they are in the same
equivalence class to represent the same $K^1$ group element
when the lattice spacing is sufficiently small.



\subsection{Set up in continuum theory}

We consider an $n$-dimensional torus $T^n=\R^n/\Z^n$
and a complex vector bundle $E\to T^n$ equipped with a Hermitian fiber metric.
Therefore, the structure group $G$ is unitary.
This bundle $E$ can be twisted, or can have nontrivial connections,
and the projection to the base space is denoted by $\pi:E\to T^n$.
Note that $E$ may contain not only the color degrees of freedom,
which is taken in some representation of the structure group $G$,
but also those of spinors and flavors in physics.
Here we do not specify them unless they potentially cause a confusion.

We summarize the details of the setup in the continuum theory below.
\begin{itemize}
\item $T^n=\R^n/ \Z^n$. We set the size $L=1$ in every direction and assign a flat metric.

\item $E \to T^n$: a complex rank $r$ vector bundle with a smooth Hermitian fiber metric.
We denote its structure group by $G$.
  For smooth sections $u,v$ of $E$ and the inner product $(*,*)$ in the fiber space,
  we denote their inner product by 
$$
\langle u |v \rangle:= \int_{T^n} (u,v) d^n x, \quad
||v||_{L^2}:= {\langle v | v \rangle}^{1/2},
$$
where $d^n x$ denotes the volume element of $T^n$
with the standard flat metric.

\item $\gamma:E \to E$: $\Z/2$-grading self-adjoint orthogonal operator
satisfying $\gamma^2={\rm id}$.

\item $c_i:E \to E$ $(i=1,\ldots, n)$ : the self-adjoint orthogonal operators
with the condition $c_i^2={\rm id}$, which represent the Clifford algebra.
They satisfy $\{\gamma, c_i\}=0$ and $ \{c_i,c_j\}=0 \,(i \neq j)$.
  
\item A fixed connection on 
  $E$ which preserves the above structure. 
      The connection determines the covariant derivative that we denote 
      in the $i$-th direction by $\nabla_i^{\rm cont.}$.
      Its explicit form will be given later.

\item A complex Hilbert space $\HH$ defined by
  a completion of the smooth sections of $E$ by the norm $||*||_{L^2}$,
  $$\HH:= L^2(E).$$
  This is an infinite-dimensional complex Hilbert space\footnote{
  The elements of $\HH$ coincide with the equivalence classes of
  the $L^2$-finite functions where the functions are equivalent
  when they match almost everywhere.
  This coincidence can be understood as follows.
  Let us take an element in our Hilbert space $\HH$ defined by completion,
  which is defined by a Cauchy series $u_1,u_2,u_3 ,\ldots$ of smooth sections
  with respect to the $L^2$ norm.
  When we take a subseries of it, they converge almost everywhere.
  Namely, the set of non-convergent points has measure zero and
  the subseries converges to a measurable function with a finite $L^2$ norm.
  When we take two such subseries, they match almost everywhere.
  Conversely, for a $L^2$-finite measurable function, there exists
  a Cauchy series $u_1,u_2,u_3,\ldots$ as above.
  For another series $v_1,v_2,v_3,\ldots$, the two series represent the
  same element of the completion. Namely, the alternate series
  $u_1,v_1,u_2,v_2,u_3,v_3 \ldots$ is also a Cauchy series.
  }.

\item The $L_1^2$ norm defined as follows.
$$
 || v||_{L^2_1} :=\Big( ||v||^2_{L^2} +  \frac{1}{m_0^2}\sum_{i=1}^n \int_{T^n} | \nabla_{i}^{\rm cont.} v |^2 d^n x\Big) ^{1/2},
 $$
 where $m_0$ is an arbitrary real number,
 which is often taken as a typical scale of the physical system we focus on.
 Note that there are non-smooth elements in $\HH$ whose $L_1^2$ norm is not finite\footnote{
 To be precise, the following statement is known to be not always true.
 For a Cauchy series $u_1,u_2,u_3 ,\ldots$ of smooth sections in $\HH$,
 there exists a Cauchy series $v_1,v_2,v_3 ,\ldots$ with a finite $L^2_1$ norm
 such that the alternate series $u_1,v_1,u_2,v_2,u_3,v_3 \ldots$ is
 a Cauchy series with respect to the $L^2$ norm (When this statement is true,
 the series $u_1,u_2,u_3 ,\ldots$ converges to an element with a finite $L^2_1$ norm.).
 }.
Therefore, the $L_1^2$ norm is necessary 
in order to define the domain of the Dirac operator.
Here we have made a simplest choice  
from several different definitions\footnote{
  The other definitions are given, for example,
  by local parallel transports with respect to the fixed connection,
  or by differentials of a flat trivial bundle into which $E$ is embedded.
  }.

\item $L^2_1(E)$: the completion of smooth sections of $E$ with respect to the $L_1^2$ norm
  $||*||_{L^2_1}$. There exists a natural injection map
$$L^2_1(E) \to \HH=L^2(E).$$
\end{itemize}

\begin{definition}
  The Dirac operator on $T^n$ is defined by
  $$
\gamma D_{T^n}: L^2_1(E) \to \HH, \qquad 
\gamma D_{T^n} v = \gamma \sum_i c_i \nabla_{i}^{\rm cont.} v,
$$
for $v\in L_1^2(E)$.
\end{definition}

Since $D_{T^n}$ is an unbounded operator, Remark \ref{remark:unbounded} applies in the following.
Note that $D_{T^n}$ anticommutes with a self-adjoint operator $\gamma : E\to E$
which satisfies $\gamma^2={\rm id}$.
Here we also introduce a self-adjoint mass term $m{\rm id}: E\to E$
with $m\in \R$.
In the following, we denote the operator by $m$ for simplicity.
To make a $K$-group element, we consider $\gamma(D_{T^n}+m)$, which is
self-adjoint\footnote{
Note that the $K$-groups can be defined with skew-adjoint operators \cite{key285033m},
but we employ a physicist-friendly notation where
the Dirac operator is skew-adjoint while the mass term is self-adjoint
and only $\gamma(D_{T^n}+m)$ is self-adjoint. 
}.
The standard Fredholm index of $D_{T^n}$ is given by
$$
{\rm index} :=\trace \gamma|_{\Ker(\gamma D_{T^n})} \in \Z.
$$
This index characterizes the element
$$
[(L^2(E), \gamma D_{T^n} ,\gamma)] \in K^0(\{ {\rm pt} \}, \emptyset) \cong \Z .
$$
In the same way as in our discussion in the previous section, the index is equal to
the spectral flow, characterizing
$$
[(p^*(L^2(E)), \gamma (D_{T^n} + m))]] \in K^1(D^1, S^0) \cong \Z,
$$
where the parameter $m$ denotes the coordinate for $D^1=[-M,M]\;(M >0)$
and $p^*$ is the pullback of the projection to the massless point, $p:[-M,M]\to \{m=0\}$.

\subsection{Wilson Dirac operator in lattice theory}

Here we summarize the setup on a lattice.
All the quantities below depend on the lattice spacing $a$.
But the $a$ dependence is suppressed for simplicity
except for those we need to manifestly distinguish from
the corresponding continuum quantities.

\paragraph{Data in lattice theory}

We construct a lattice data for each $N \in \Z_{>0}$
from those in continuum theory defined above.
\begin{itemize}
  
\item A flat square lattice $T_a^n=a\Z^n/ \Z^n \subset T^n$ with the lattice spacing $a=\frac{1}{N}$. Note that the lattice size $Na=1$ is fixed.

\item $E_a:= E|_{T^n_a}$.

\item $\gamma:E_a \to E_a$: the restriction of the continuum $\Z/2$-grading operator onto $E\to E_a$.
  Note that $\trace \gamma =0$ always holds on the lattice.

\item $c_i:E_a \to E_a$ $(i=1,\ldots n)$ : the restriction of the continuum $c_i$'s onto $E\to E_a$.

\item $U_i$ $(i=1,\ldots, n)$ defined by
  the parallel transport on $E$ from each lattice point
  to its nearest-neighbor site
  in the negative $i$-th direction. The explicit form of the parallel transport is presented in Sec.~\ref{sec:generallink}.
  We assume that $U_i$ is compatible with the Clifford operator $c_j$ (for any $j=1,\ldots n$) and $\gamma$.

      Differential operation is approximated by the forward difference
operator
$$\nabla_i:=\frac{U_i-{\rm id}}{a} :\Gamma(E_a) \to \Gamma(E_a),$$
or its conjugate,
$$
-\nabla^*_i=\frac{{\rm id}-U_i^{-1}}{a} :\Gamma(E_a) \to \Gamma(E_a),$$
which is the backward difference operator.

\item $\HH_a:=L^2(E_a)$ where the $L^2$ norm is defined by

$$
  ||v_a ||_{L^2} :=  \Big( 
  a^n\sum_{ z \in T^n_a} |v_a(z)|^2 \Big)^{1/2},
$$
for $v_a \in \HH_a$.
\end{itemize}





Using $\nabla_i$, we also define the $L_1^2$ norm on the lattice by
$$
||v_a||_{L^2_1} := \left[||v_a ||^2_{L^2} + 
  \frac{a^n}{m_0^2}\sum_{ z \in T^n_a} \sum_{i=1}^n |(\nabla_i v_a)(z)|^2\right]^{1/2}.
$$
With a fixed lattice spacing, there is no essential difference between
the $L^2$-finite space and the $L^2_1$-finite space.
But in the continuum limit $a\to 0$,
the two converge to the different continuum counterparts.


Note that $\nabla_i$ and $-\nabla^*_i$
are not skew-adjoint, while their continuum counterpart $\nabla^{\rm cont.}_i$ is.
In order to make a skew-adjoint discretization of $\nabla^{\rm cont.}_i$,
we take their ``average'' to define a naive skew-adjoint lattice Dirac operator,

$$
D_a:=\sum_i D_i,\quad D_i:= c_i\frac{\nabla_i - \nabla^*_i}{2}.
$$

As is mentioned in the introduction, the operator $D_a$ has doubler modes. 
To avoid their appearance, we introduce the Wilson term \cite{Wilson1977}
defined by
$$
W:= \sum_i W_i,\quad W_i:= - \frac{\nabla_i + \nabla^*_i}{2}
$$
For $W_i$, the following lemma holds.

\begin{lemma}
  \label{lemma:positivityW}
$$\nabla_i^*\nabla_i=\nabla_i\nabla_i^*=\frac{1}{a^2}({\rm id}-U_i-U_i^{-1}+{\rm id})
=-\frac{1}{a}(\nabla_i+ \nabla_i^*)$$
or we can write
$$
W_i=\frac{a}{2} \nabla_i^* \nabla_i,
$$
which is semi-positive and so does the Wilson term $W=\sum_i W_i$.
\end{lemma}

\begin{definition}
The Wilson-Dirac operator \cite{Wilson1977} is defined by
$$
D_{W,a}:=D_a+ W.
$$
Note that $\gamma D_{W,a}$ is Hermitian but
the $\Z/2$ grading structure is lost: 
$\{D_{W,a},\gamma\}\neq 0$.

\end{definition}

For the modes with small momenta, the Wilson term 
is a good approximation of the Laplacian operator
with a coefficient proportional to the lattice spacing $a$,
whose effect vanishes in the $a\to 0$ limit.
However, it is essential to eliminate 
the doubler modes
with wave numbers near the cutoff $\sim 1/a$.

Mathematically, the Wilson term guarantees 
the ``ellipticity''. 
In fact, we will prove a theorem,
which corresponds to the G\aa rding's inequality 
in the next section.

Since it is difficult to maintain the $\Z/2$-grading operator $\gamma$ on the lattice,
the $K^1(D^1, S^0)$ group is more useful than $K^0(\{\rm pt\},\emptyset)$ in the following analysis.



\subsection{Generalized link variables}
\label{sec:generallink}

So far we have considered a general set of link variables given by $\{U_i\}$.
To identify $E_a$ as an approximation of the continuum vector bundle $E$ with the rank $r$,
we need to relate these link variables and connection in $E$.
In the physics literature, one often gives a connection
(vector potential) first,
and then the link variables are given by 
the parallel transport (or Wilson line in physics) 
between the two lattice points.
Here we take an opposite direction. 
We first introduce the link variables between two 
points on the continuum torus, which we call the generalized link variables,
and then we determine the connection from them.
With this approach, we can avoid details of the
local trivializations of $E$ and ambiguity in 
the paths to define the Wilson lines.

Let us fix a small constant $a_0>0$
and consider a thin stripe $W \subset T^n \times T^n$:
$W=\{(x,y)|\; x\in T^n, y\in T^n, |x-y|<a_0 \}$.
Let $\pi_1 :W\to T^n$ be the projection map onto the first 
component: $(x,y) \mapsto x$ and
$\pi_2 :W\to T^n$ be that onto the second component:
 $(x,y) \mapsto y$.

We define the general link variable by a smooth section $U$
of ${\rm Hom}(\pi^*_2 E,\pi_1^* E)$,
{\it i.e.},
\[
 U(x,y) \in {\rm Hom}(E_y,E_x),
 \]
 which depends smoothly on $x$ and $y$.
We assume that $U(x,y)$ satisfies
   the conditions\footnote{The standard parallel transport or Wilson line 
given by a smooth connection on 
the shortest path between $x$ and $y$ satisfies the both conditions.
For our main theorem, the second condition $U(y,x)=U(x,y)^{-1}$ 
is not essential but simplify the analysis.}
$U(x,x)={\rm id}$ and 
$U(y,x)=U(x,y)^{-1}$.
We also assume that
the action of link variables are compatible with the fiber metric.

The generalized link variables $U(x,y)$ must commute with 
all the symmetry operations on $E$ such as the $\Z_2$ grading operator,
the Clifford generators, 
 and so on.
Let us consider a symmetry group $G_{\rm sym.}$ and the operation of the element $\alpha\in G_{\rm sym.}$ 
The operation is represented by a set of operators acting on $E$, for example,
$O_\alpha=\{ O_\alpha(x) \in End(E_x)\}$ which smoothly depends on $x$
and preserves the fiber metric (and thus the operation is unitary),  
to which we require 
\[
 U(x,y)O_\alpha(y)=O_\alpha(x)U(x,y).
\]

We can fix the connection of $E$ by identifying 
$U(x,y)$ as the generator of the parallel transport.
Note that $U(x,y)$ and the associated connection
are defined without any local trivializations.
However, it is useful to derive
the local expression of the connection one-form as follows.

Choosing a local trivialization $\varphi :\Omega_\alpha \times \mathbf{C}^r \to \pi^{-1}(U_\alpha)$
on an open patch $x, y \in \Omega_\alpha \subset T^n$, 
the logarithm of $U(x,y)$ can be expanded as
\[
\varphi^{-1}(x)U(x,y)\varphi(y)=\exp\left[-A_i\left(\frac{x+y}{2}\right)(x-y)^i +R(x-y)\right], 
\]
where we identify the skew-Hermitian matrix $A_i(x)=\lim_{y\to x}A_i((x+y)/2):\R^n \to M_r(\bf{C})$ 
as the local connection gauge field 
  and the residual\footnote{
    The residual $R(x-y)$ is a series
    $R(x-y)=\sum_{i,j}A^{(2)}_{ij}((x+y)/2)(x-y)^i(x-y)^j+\sum_{i,j,k}A_{ijk}^{(3)}((x+y)/2)(x-y)^i(x-y)^j(x-y)^k+\cdots$ where $A^{(l)}_{ij\cdots}:\R^n \to M_r(\bf{C})$ are
    skew-Hermitian matrices because $U(x,y)$ is unitary.
    The condition $U(y,x)=U(x,y)^{-1}$ implies that the sign flip
    $A^{(l)}_{ij\cdots} \to - A^{(l)}_{ij\cdots}$
    and the coordinate swap $x\leftrightarrow y$ are equivalent.
    Therefore, $A^{(l)}_{ij\cdots}$ with even $l$ must be zero.
    }
    is suppressed to $|R(x-y)|\leq const.|x-y|^3$.
The covariant derivative is defined as $\varphi^{-1}(x)\nabla^{\rm cont.}_i\varphi(x) = \partial_{x_i} +A_i(x)$,
with which the connection is globally determined on $E$.

Changing the local trivialization from $\varphi(x)$ to $\varphi'(x)$ 
corresponds to the gauge transformation.
The gauge variable $g(x)=\varphi^{\prime -1}(x)\varphi(x)$ 
is in some representation of the structure group $G$.
The link variables, local operators, and the connection gauge field transform as
\begin{align*}
 \varphi^{-1}(x)U(x,y)\varphi(y)&\to g(x)\varphi^{-1}(x)U(x,y)\varphi(y)g^{-1}(y),\\\;\; 
\varphi^{-1}(x)O_\alpha(x)\varphi(x)&\to g(x)\varphi^{-1}(x)O_\alpha(x)\varphi(x)g^{-1}(x),\\
A_i(x)&\to g(x)[A_i(x)+\partial_{x_i}]g(x)^{-1}.
\end{align*}


In the same way, the triangle Wilson loop $U_W(x,y,z)=U(x,y)U(y,z)U(z,x)$ is evaluated as
\begin{align*}
 \varphi^{-1}(x)U_W(x,y,z)\varphi(x) &= e^{-A_i\left(\frac{x+y}{2}\right)(x-y)^i+\cdots}
 e^{-A_i\left(\frac{y+z}{2}\right)(y-z)^i+\cdots}e^{-A_i\left(\frac{z+x}{2}\right)(z-x)^i+\cdots}
\\& = \exp\left[
-\frac{1}{2}\sum_{i,j}\partial_{x_j}A_i(x)\{(x^i-y^i)(y^j-x^j)\right.\\&
+(y^i-z^i)(y^j-x^j)
+(y^i-z^i)(z^j-x^j)+(z^i-x^i)(z^j-x^j)\}
\\ & +\frac{1}{2}[A_i(x),A_j(x)]\{(x^i-y^i)(y^j-z^j)+(x^i-y^i)(z^j-x^j)
\\ &\left.
+(y^i-z^i)(z^j-x^j)\} 
+\cdots \right]
\\& =
\exp\left[\frac{1}{2}F_{ij}(x)(x-y)^i(x-z)^j+R'(x,y,z)\right]
\end{align*}
where $|R'(x,y,z)|\leq const.(|x-y|^3+|y-z|^3+|z-x|^3)$.
Here $F_{ij}(x)$ is the curvature tensor at $x$.

It is also important to note that the continuum covariant derivative
of the link variable is suppressed to $O(|x-y|)$:
\begin{align*}
 \varphi^{-1}(x)\nabla^{\rm cont.}_i U(x,y)\varphi(y)&=
\lim_{\Delta x_i\to 0}\frac{U(x,x+\Delta x_i e_i)U(x+\Delta x_i e_i, y)-U(x,y)}{\Delta x_i}
\\&=\lim_{\Delta x_i\to 0}\frac{U_W(x,x+\Delta x_i e_i, y)-{\rm id}}{\Delta x_i}U(x,y)
\\&=-\frac{1}{2}F_{ij}(x)(x-y)^j+R''(x,y),
\end{align*}
where $|R''(x,y)|\leq const.|x-y|^3$. 
Since $U(x,y) \in {\rm Hom}(E_y,E_x)=E_x \otimes E_y^\ast$, we can take the covariant derivative with respect to $x$ of $U(x,y)$ in the $i$-th direction in the same way as the section of $E$.

Therefore, it is useful to evaluate the following bound,
which does not depend on the choice of the trivializations:
\begin{lemma}
\label{lemma:curvaturebound}
 There exists a positive constant $F$ such that the inequality
\[
  |{\rm id}-U_W(x,y,z)| \leq F|x-y||x-z|,
\]
is satisfied for any $\{(x,y,z)| (x,y),(y,z),(z,x)\in W\}$
and so is
\[
 |\nabla^{\rm cont.}_i U(x,y)| \leq F|x-y|,
\]
for any $(x,y) \in W$.
\end{lemma}

Any link variables on a finite lattice with the lattice spacing $a\le a_0$
are simply determined by choosing $x$ to be one site $z$ and
$y$ to be its nearest neighbor $z+e_i a$ in the $i$-th direction:
\[
 U(z, z+e_i a).
 \]
We can identify the operator $U_i$ in the Wilson Dirac operator
as the {\it parallel transport} determined by $\{U(z, z+e_i a)\}$,
and $U_i^{-1}$ as that by $\{U(z-e_i a, z)^{-1}\}$.
Specifically the forward and backward covariant differences on $\phi \in L^2_1(E_a)$ 
are given by
\begin{align*}
\nabla_i \phi(z) &= \frac{U_i-{\rm id}}{a}\phi(z)=\frac{1}{a}\left[U(z,z+e_ia)\phi(z+e_ia)-\phi(z)\right],\\
-\nabla_i^* \phi(z)&= \frac{{\rm id}-U_i^{-1}}{a}\phi(z) = \frac{1}{a}\left[\phi(z)-U(z-e_ia,z)^{-1}\phi(z-e_i a)\right],
\end{align*}
respectively.
It is then not difficult to confirm that $D_{W,a}$ transforms covariantly:
\[
 \varphi^{-1}(z)D_{W,a}(z,z')\varphi(z') \to g(z)\varphi^{-1}(z)D_{W,a}(z,z')\varphi(z') g^{-1}(z') \;\;\; z,z' \in T^n_a,
\]
under the gauge transformation.

The system may contain nontrivial global holonomy on $T^{n}$.
The global holonomy in the $i$-th direction, for example, is defined by
\begin{align*}
  H_i(x)= U(x+ N a e_i , x+ (N-1) a e_i  ) \cdots U(x+ 2ae_i, x+ ae_i) U(x+ ae_i,x) .
\end{align*}
Since $x+Na e_i$ is identified with $x$, the family $H_i= \{ H_i(x) \}_x $ defines a vector bundle homomorphism. Even if $E$ is a trivial bundle equipped with a flat connection, $H_i$ may still be nontrivial. 

So far we have not looked into the finer structure of
the vector bundle $E$ (and $E_a$).
In physics on a flat torus, 
$E$ is given by a tensor product bundle,
\[
E = E_c\otimes S \otimes \underline{F},
\]
where $E_c$, which is generally twisted,
denotes the color degrees of freedom in some representation 
of the structure group $G$.
As a simple example, let us take $G=SU(N_c)$ and choose its
fundamental representation with $N_c$ dimensions.
The other two bundles $S$ and $\underline{F}$ are
those of spinor (with rank $r_S=2^{[n/2]}$) and
flavor ($r_F$) degrees of freedom, respectively.
On a flat torus, both bundles $S$ and $\underline{F}$ are trivial.
The fermion field at $x\in T^n$ with the total
$r=N_c r_S r_F$ degrees of freedom is often expressed by
$\psi^c_{sf}(x)$ in a local trivialization of $E$ around $x$
where $c,s,f$ are color, spinor and flavor indices, respectively.

In this tensor product bundle case, the generalized link variables
are given by
\[
U(x,y) = U_c(x,y)\otimes {\rm id}_{S} \otimes {\rm id}_{\underline{F}},
\]
where
\[
U_c(x,y) \in {\rm Hom}((E_c)_y,(E_c)_x).
\]
With the local trivialization
$\varphi_c :\Omega_\alpha \times \mathbf{C}^{r_c} \to \pi_c^{-1}(\Omega_\alpha)$
on a patch $\Omega_\alpha \subset T^n$
with the projection $\pi_c : E_c\to T^n$,
$\varphi_c(x)^{-1}U_c(x,y)\varphi_c(y) \in M_{N_c}(\mathbf{C})$
corresponds to the standard link variables in physics.
In the same way, the gauge field in the same local trivialization is
\[
A_i(x) = A^c_i(x) \otimes {\rm id}_{S} \otimes {\rm id}_{\underline{F}},
\]
where $A^c_i(x)$ is a anti-Hermitian $N_c\times N_c$ matrix,
as well as the gauge transformation function is given by
\[
g(x) = g_c(x) \otimes {\rm id}_{S} \otimes {\rm id}_{\underline{F}},
\]
with $g_c(x) \in M_{N_c}(\mathbf{C})$.
Finally we express the Clifford generator by
\[
c_i = {\rm id}_c \otimes \gamma_i \otimes {\rm id}_{\underline{F}},
\]
where $\gamma_i$ is the standard $r_S\times r_S$ Dirac matrix,
satisfying $\gamma_i^2={\rm id}_S$.




\subsection{Main theorem}

We fix a constant $M>0$ and put
$$
X:=[-M,M],\qquad A:= \partial X=\{ -M, M\}.
$$
Note that $X$ includes the massless point $m=0$
where the continuum Dirac operator may develop zero modes.

Taking $m \in X$ as a parameter, we consider 
the following two one-parameter families
$$
\gamma D^m:= \gamma (D_{T^n} +m), \qquad  
\gamma D_{W,a}^m:= \gamma (D_{W,a} +m),
$$
whose domains are 
$$
\gamma D^m : L^2_1(E) \to \HH,\qquad
\gamma D_{W,a}^m: L^2_1(E_a) \to \HH_a.
$$

\begin{lemma}
$\gamma D^m$ has the following properties.
\begin{itemize}
 \item $\gamma D^m$ is a Fredholm operator for any $ m\in X=[-M,M]$.


\item $\Ker (\gamma D^m)=\{0\}$ for $m \in A=\{ - M, M\}$\footnote{
$\Ker (\gamma D^m)$ is defined as the set of all smooth solutions to the massive Dirac equation.
In general, the weak solutions of $\gamma D^m$ are smooth. 
Here the weak solution means a
measurable section $\phi$ of the Hilbert bundle
which satisfies $\langle (\gamma D^m)^* \beta |\phi\rangle =0$ for arbitrary smooth section $\beta$,
where $(\gamma D^m)^*$ is a formal adjoint of $\gamma D^m$.
  Since $\gamma D^m$ is elliptic
in our set up, the set of all smooth solutions equals to  the set of all weak solutions.}.
\end{itemize}

\end{lemma}
\begin{proof}
For any $m \in \R$, $\gamma D^m$ is Fredholm from a general argument.
Since $\{ c_i, \gamma\}=0$, $ \gamma D_{T^n}$ and $\gamma$ anticommute
and $ (\gamma D^m)^2= (\gamma D_{T^n})^2 +(\gamma m)^2 = (\gamma D_{T^n})^2 + m^2$.
For any $\phi \in \Ker (\gamma D^m)$,
$$
0= \langle \phi | (\gamma D^m)^2 | \phi \rangle
=\langle \phi |  (\gamma D_{T^n})^2+ m^2 | \phi \rangle
=||\gamma D_{T^n} \phi ||_{L^2}^2 + m^2 ||\phi||_{L^2}^2.
$$
Noting that each term above is positive,
we conclude that $\phi=0$ for $m \in A$.
\end{proof}

\begin{theorem}\label{goal}
Let us define a lattice-continuum combined Dirac operator
$H_{\rm com}(m,t)$ which acts on $\HH_a \oplus L^2_1(E)$ by
$$
H_{\rm com}(m,t):=\left(
    \begin{array}{cc}
        -\gamma D_{W,a}^m & tf_a^*\\
       tf_a   &  \gamma D^m 
    \end{array}
  \right)
=
\left(
    \begin{array}{cc}
        -\gamma D_{W,a} & 0\\
       0  &  \gamma D_{T^n} 
    \end{array}
  \right)
+
m 
\left(
    \begin{array}{cc}
        -\gamma & 0\\
      0  &   \gamma
    \end{array}
  \right)
+t
\left(
    \begin{array}{cc}
        0 & f_a^*\\
       f_a   &  0
    \end{array}
  \right).
$$
There exist $a_1>0$ and a linear map $f_a: \HH_a \to \HH$ 
such that for arbitrary lattice spacing $a=1/N$ satisfying $0< a \le a_1$,
the followings hold.
\begin{itemize}
\item
For any $m\in X=[-M,M]$ the kernel of $H_{\rm com}(m,t=1)$ is trivial.

\item
For any $m\in A=\{-M,M\}$, the kernel of $H_{\rm com}(m,t)$ is trivial
at any $t\in [0,1]$.
\end{itemize}

\end{theorem}
We illustrate in Fig.~\ref{fig:mt} 
the staple-shaped parameter region in the $m$-$t$ plane
where we will prove that $H_{\rm com}(m,t)$ defined above is invertible.

\begin{figure*}[tbh]
  \centering
  \includegraphics[scale=0.3]{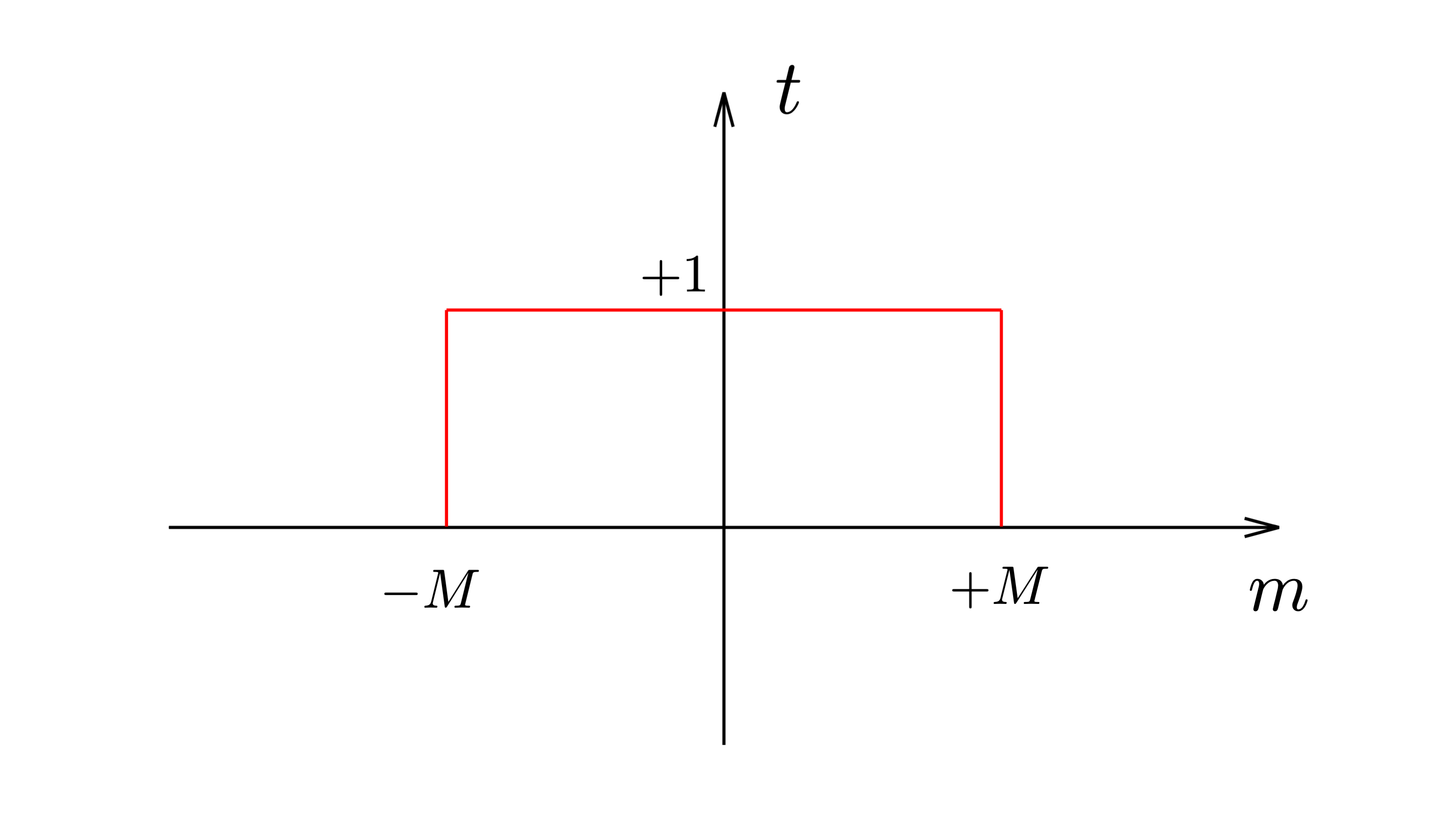}
  \caption{
 The stape-shaped parameter region in the $m$-$t$ plane
 where we prove that the lattice-continuum combined 
Dirac operator $H_{\rm com}(m,t)$ is invertible.
  }
  \label{fig:mt}
\end{figure*}

Then the theorem below follows
\footnote{Here, $\gamma D_{T^n}$ is an unbounded operator.
We can replace it by, for instance, $\gamma D_{T^n}/[(\gamma D_{T^n})^2+M_0^2]^{1/2}$ with some mass parameter $M_0$ 
which is  independent of $a$
to make a bounded operator to which the theorems directly apply.}.
\begin{theorem}(Index of lattice Dirac operator)\label{goalgoal}
For a sufficiently small lattice spacing $a=1/N$, the following equivalence holds.
\[
[(p^* \HH, \gamma (D_{T^n}+ m)] = [(p^* \HH_a,\gamma (D_{W,a} + m)]\in K^1(D^1,S^0),
\]
where $p^*$ is the pullback of the projection map $p: [-M,M]\to \{m=0\}$.
\end{theorem}
Note here by the suspension isomorphism that
$$
K^1(D^1,S^0) \cong K^0(\{ * \}, \emptyset) \quad
[(p^* \HH, \gamma (D_{T^n}+ m)] \leftrightarrow [(\HH,\gamma D_{T^n}, \gamma)]
$$
holds. The original Fredholm index of the Dirac operator for the given continuum data
characterizes the right-hand side.
We denote it by ${\rm Ind}[\gamma D_{T^n}]$.

Note here that the spectral flow of the Hermitian massive Dirac operator can be expressed by
the variation of the $\eta$ invariant.
Therefore, we have
\[
 {\rm sf}(\HH_a,\gamma D^m_{W,a} ) = \frac{1}{2}\left[\eta(\gamma D^m_{W,a}(m=+M))-\eta(\gamma D^m_{W,a}(m=-M))\right].
\]
Also, we can show that $\eta(\gamma D^m_{W,a}(m))=0$ for any $m>0$ 
(see Appendix~\ref{app:gapWilson} ). 
Therefore we obtain a theorem below.
\begin{theorem}(Index and eta invariant of Wilson Dirac operator)\label{goalgoalgoal}
For a sufficiently small lattice spacing $a=1/N$, the following equality holds.
\[
{\rm Ind}[\gamma D_{T^n}] = -\frac{1}{2}\eta(\gamma D^m_{W,a}(m=-M)).
\]
\end{theorem}


Before going into the proof, 
let us discuss why $K^1(D^1, S^0)$ works
while $ K^0(\{ {\rm pt} \}, \emptyset)$ does not.
Suppose the original continuum operator $\gamma D_{T^n}$ has $\nu_+$ zero modes with the chirality $\gamma=1$
and $\nu_-$ zeros with $\gamma=-1$.
These $\nu_\pm$ zero modes are also the eigenstates of
$\gamma (D_{T^n} + m)$ with the eigenvalue $\pm m$.
Every other eigenvalue $\lambda_m$ whose absolute value is greater than $|m|$
appears in a pair with $-\lambda_m$,
since the eigenmode $\psi_{\lambda_m}$
satisfies
\begin{align*}
  \gamma (D_{T^n}+m )  \psi_{\lambda_m} &=\lambda_m \psi_{\lambda_m},\\
  \gamma (D_{T^n}+m ) (D_{T^n})  \psi_{\lambda_m} &= -
  \lambda_m (D_{T^n}) \psi_{\lambda_m},
\end{align*}
which follows from the anti-commutation relation
$\{\gamma (D_{T^n}+m), D_{T^n}\}=0$.
Since $[\gamma (D_{T^n}+m)]^2=(\gamma D_{T^n})^2+m^2$,
The eigenvalues are $\lambda_m=\pm \sqrt{\lambda_1^2+m^2}$.

The Dirac eigenvalue spectrum as a function of $m$ is
illustrated in Fig.~\ref{fig:SpecFlow}.
The upper panel shows the original Dirac eigenvalues in the continuum theory,
while the lower panel shows that of a lattice approximation of it,
where the chiral symmetry is broken and the Dirac operator does not
anticommutes with $\gamma$.
While it is difficult to define the massless point and
to identify the number of zero modes,
the spectral flow can still be identified by the
lines crossing zero between $m\in [-M,M]$.
$ K^0(\{ {\rm pt} \}, \emptyset)$ corresponds to counting
the points at $m=0$, while
$K^1(D^1, S^0)$ counts the lines between $-M$ and $+M$.
The latter is easier and stabler against variation of the
Dirac operators and violation of the chiral symmetry.

\begin{figure*}[tbh]
  \centering
  \includegraphics[scale=0.4]{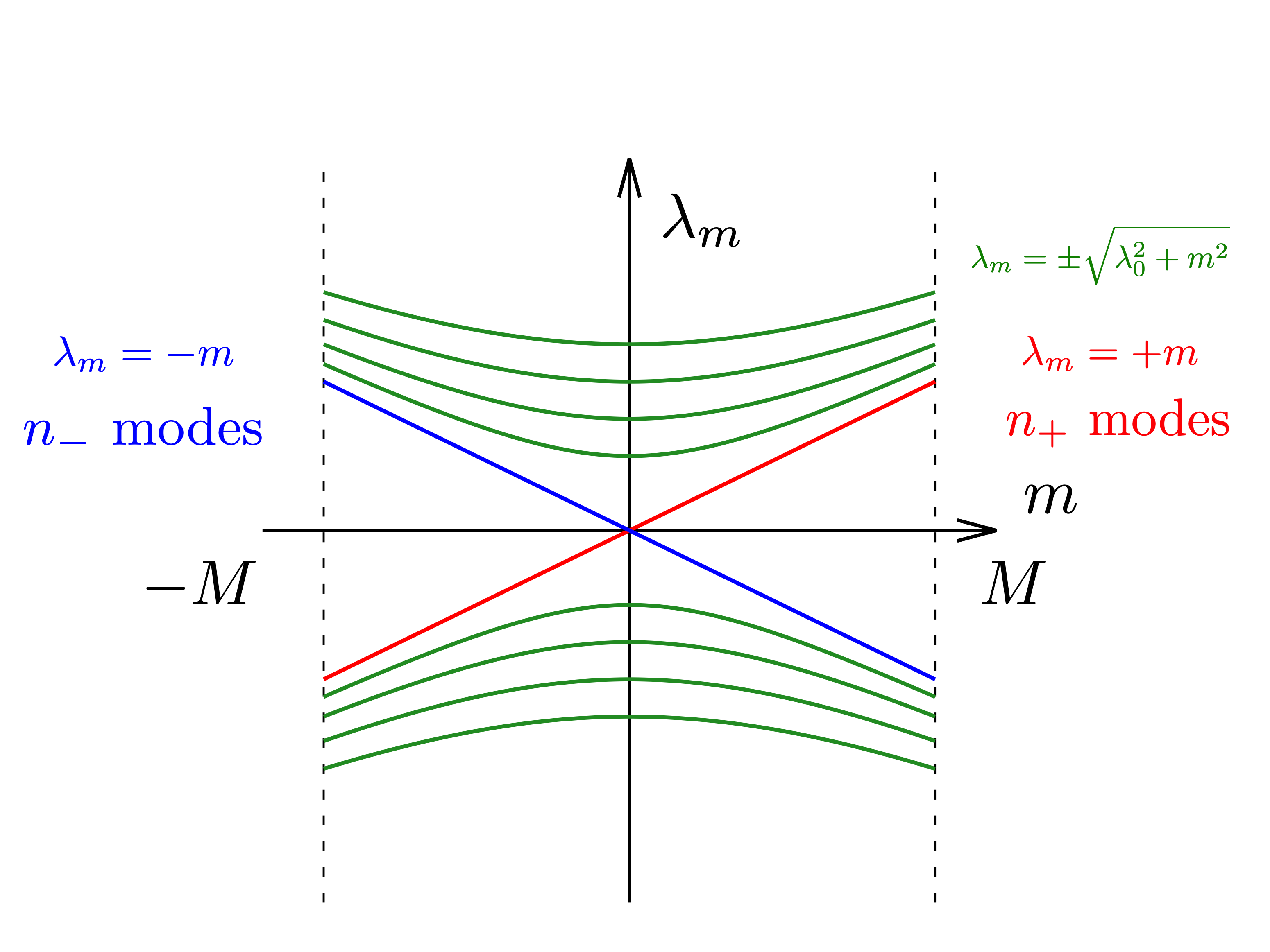}
  \includegraphics[scale=0.4]{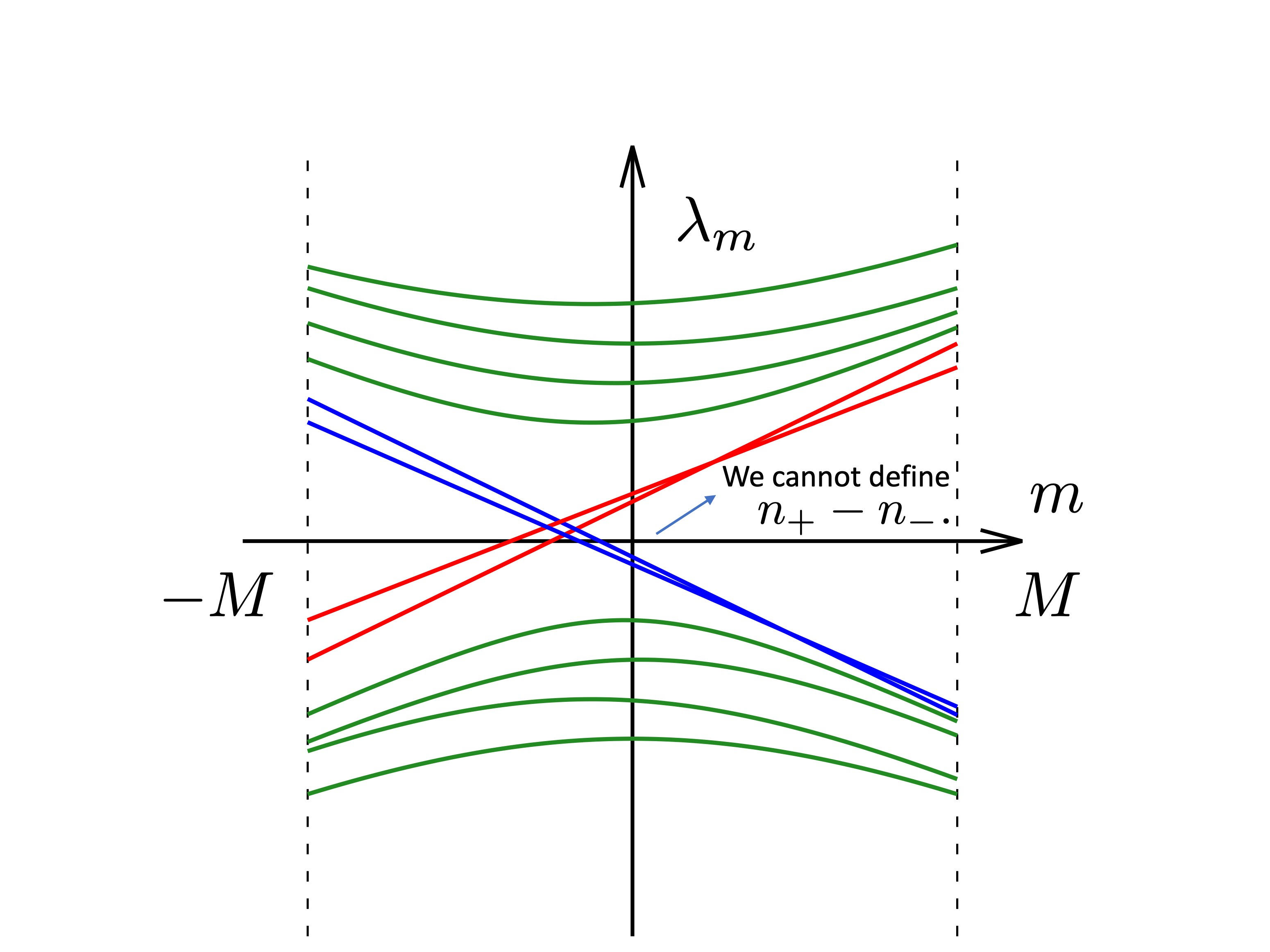}
  \caption{
    Dirac eigenvalue spectrum as a function of $m$.
    The upper panel illustrates the original one in the continuum theory,
    while the lower panel shows that of a lattice approximation
    where the chiral symmetry is broken.
    The spectral flow can still be identified  while
    it is difficult to count the zero eigenvalues at $m=0$.
  }
  \label{fig:SpecFlow}
\end{figure*}

\section{Proof of the main theorem}
\label{eq:main}

\subsection{Maps between continuum and lattice Hilbert spaces}
Let us construct a linear map between the infinite-dimensional Hilbert
bundle $\HH$ in continuum theory and the finite-dimensional 
Hilbert bundle $\HH_a$ on the lattice. 

We denote the coordinate on $T^n$ by $x=(x_1,\ldots x_n)$
with the standard flat metric
where each component is in the range $0\le x_i\le 1$
and $x_i=0$ and $x_i=1$ are identified.
On the lattice, we denote the lattice coordinate by $z=(z_1,\ldots z_n)$
where each component $z_i$ is an integer multiplied by 
$a$ less than $Na=1$.
Let $E\to T^n$ be a complex rank-$r$ vector bundle equipped with a Hermitian metric.
We may set an open covering of $T^n$ by $\{U_\alpha\}$
on which we fix a local trivialization 
of the bundle but the properties below do not depend on
details of $\{U_\alpha\}$.

Let us define a map $\rho_a: [0,1] \to \R$ by  
$\rho_a(t):= \frac{1}{a}\max \{ 0, 1- t/a, 1-(1-t)/a\}$ and consider 
its multi-dimensional product as a function of $x\in T^n$,
$$
\rho_a(x):=\prod_{i=1}^n \rho_a(x_i).
$$
We will use this $\rho_a(x)$ as a cut-off function
and the following properties are essential for the proof of the main theorem.

Denoting the partial derivative with respect to $x_i$ 
by $\partial_i$, we have for the lattice site $z \in T_a^n$
$$
\partial_{i} \rho_a(x-z)
=
\left\{
\begin{array}{ll}
        + \frac{1}{a^2}\prod_{j \neq i} \rho_a(x_j-z_j) & \qquad (-a<x_{i}-z_i <0) \\
       -  \frac{1}{a^2}\prod_{j \neq i} \rho_a(x_j-z_j)    &\qquad (0<x_{i}-z_i <a) \\
       0   & \qquad (|x_i-z_i |>a)
    \end{array}.
\right.
$$

\begin{lemma}
  For arbitrary $x \in T^n$,
$$a^n\sum_{z\in T^n_a
} \rho_a(x-z) =1.$$
and for arbitrary $z \in T^n_a$,
$$
\int_{x \in T^n
} \rho_a(x-z) d^n x=1.
$$
\end{lemma}

\begin{lemma}
\label{lem:B}
  For $B=\{0 \} \cup \{ \pm e_k \mid 1 \leq k \leq n\}$,
$$
a^n\sum_{e \in  B } \int_{x  \in T^n } \rho_a(x) \rho_a(x-ae )d^n x  =1.
$$
\end{lemma}

Using the above cut-off function, we construct the 
lattice field from the continuum counterpart in the neighborhood.
For a given continuum field $\psi_1(x)\in L^2(E)$,
we define the lattice field $\phi_1$ by
$$
 \phi_1(z):=\int_{x \in T^n} \rho_a(z-x) U(x,z)^{-1}\psi_1(x)d^n x .
$$
Note that $\rho_a(z-x)$ is nonzero only when $x$ is inside a unit hypercube
of the lattice where $z$ is one of the vertices.
We denote this map by $f_a^*: \psi_1 \to \phi_1$, which is mathematically
easier to evaluate than its adjoint defined below.


Next we define the map from a lattice field $\phi \in L^2(E_a)$ to the continuum
Hilbert space,
$$
f_a: \phi \to \psi
$$
where $\psi$ is given by
$$
\psi(x):= a^n\sum_{z \in T^n_a} \rho_a(x-z)U(x,z) \phi(z).
$$
Note again that $\rho_a(x-z)$ is nonzero only when $|x_i-z_i|<a$ for all $i$.

Since the following two are equal
 \begin{align*}
   a^n\sum_{z \in T^n_a} \phi(z)^* \phi_1(z)&=a^n\sum_{z \in T^n_a} \int_{x \in T^n} \phi(z)^*\rho_a(z-x)U(x,z)^{-1} \psi_1(x)d^n x,\quad\\
 \int_{x \in T^n} \psi(x)^* \psi_1(x) d^n x &=a^n\int_{x \in T^n} \sum_{z \in T^n_a} \phi(z)^*U(x,z)^{-1} \rho_a(x-z) \psi_1(x)d^n x,
 \end{align*}
we confirm that $f_a$ and $f_a^*$ are adjoints to each other.
From the above equality indicating $\langle \phi| f_a^* \psi_1\rangle = \langle f_a \phi | \psi_1 \rangle$, 
$f_a$ and $f_a^*$ have the same finite operator norm $||f_a||=||f_a^*||$.

\subsection{Key properties of $f_a, f_a^*$ and $D_{W,a}, D_{T^n}$}
\label{sec:properties}

Here we summarize the key properties of $f_a, f_a^*$ and $D_{W,a}, D_{T^n}$
used in the proof of our main theorem.
As the proofs of the key propositions are rather technical,
we separate them in Appendix~\ref{app:proofprop}.


\begin{proposition}[Boundedness of $f_a$ and $f_a^*$] 
\label{property:fabound}
For $k=0,1$, the operator norm of 
$f_a :L^2_k(E_a) \to L^2_k(E)$ is uniformly bounded with respect to $a$.
Likewise, the operator norm of 
$f_a^* :L^2_k(E) \to L^2_k(E_a)$ is uniformly bounded with respect to $a$.
\end{proposition}

\begin{proposition}[Continuum limit of $f_a^* f_a$]
  \label{property:f*f}
There exists $C>0$ which is independent of $a$ such that
for any $a$ and $v_a \in \HH_a$
$$
|| f_a^* f_a v_a -v_a ||_{L^2}^2 \leq C a || v_a ||_{L^2_1}^2.
$$
\end{proposition}
Intuitively, the continuum limit of $f_a^* f_a$ is the identity 
for $v_a \in L^2_1(E_a)$.

\begin{proposition}[Convergence of $f_a f_a^* \to id$]
  \label{property:weakfafa*}

For arbitrary $\psi \in C^\infty(E) \subset \HH$, in the limit $a \to 0$,
 $f_a f_a^*\psi$ strongly converges to $\psi$:
$$ 
f_a f_a^* \psi \longrightarrow \psi \quad (L^2).
$$
In particular, for arbitrary $\psi, \psi' \in C^\infty(E)\subset \HH$,
$$
\langle f_a^* \psi' | f_a ^* \psi \rangle \to \langle \psi' | \psi \rangle
$$
in the $a\to 0$ limit.
\end{proposition}

We note that $f_a f_a^* $ is mathematically easier to handle
than $f_a^* f_a$. For the latter, a precise evaluation of the scaling
with respect to the lattice spacing $a$ is essential.


\begin{proposition}[Convergence of $D_{W,a} \to D_{T^n}$]
  \label{property:DWweak}
For arbitrary $\psi \in C^\infty(E)\subset \HH$, in the limit $a \to 0$,
$$ 
f_a (D_{W,a})^* f_a^* \psi \longrightarrow (D_{T^n})^* \psi  \quad (L^2).
$$
Namely, for arbitrary $\psi, \psi' \in C^\infty(E)$,
$$
\langle f_a^* \psi' | (\gamma D_{W,a})^* |f_a ^* \psi \rangle \to \langle \psi' | (\gamma D_{T^n})^* | \psi \rangle
$$
in the $a\to 0$ limit.
\end{proposition}

\subsection{A priori estimate for lattice Dirac operator}
\label{sec:DWapriori}

The following theorem corresponds to the elliptic estimate or
G\aa rding's inequality in the continuum theory.

\begin{theorem}[A priori estimate for the Wilson Dirac operator]
\label{a priori estimate}
There exist two positive $a$-independent constants $a_2>0$ and $C>0$ such that
the following inequality uniformly holds for any finite lattice spacing satisfying $0<a\le a_2$
and arbitrary $\phi \in \Gamma(E_a)$.
$$
\sum_i|| \nabla_i \phi||^2 \leq 2 || \gamma D_{W,a} \phi ||^2 +  C  || \phi||^2.
$$
\end{theorem}


\begin{proof}
 Let us decompose $(\gamma D_{W,a})^2$ as follows.
$$
(\gamma D_{W,a})^2=\sum_i [(\gamma D_i)^2+ W_i^2] + R_{DD}+R_{WD} + R_{WW},
$$
where
$$
R_{DD}=\sum_{i<j} \{ \gamma D_i,\gamma D_j\},
\quad
R_{WD}=\gamma\sum_{i,j} [ W_i, \gamma D_j],\quad
R_{WW}=\sum_{i \neq j} W_iW_j.
$$

In the simplest case where all $U_i$ are the identity,
$R_{DD}=R_{WD}=0$ and $W_i$ and $W_j$ commute.
Moreover, $R_{WW}$ becomes a semi-positive operator,
which leads to an inequality 
$$
(\gamma D_{W,a})^2 \geq \sum_i  W_i^2= \sum_i[\nabla_i - (-\nabla^*_i)]^2.
$$
From this we have
$$
||(\gamma D_{W,a}) \phi ||^2 \geq \sum_i || \nabla_i \phi - (-\nabla^*_i) \phi||^2.
$$
This indicates that when the left-hand side for $\phi$ 
is small, the forward difference operator $\nabla_i $
and the backward one $(-\nabla^*_i)$ are close to each other, 
with respect to the $L^2$ norm.

Our goal is to achieve a similar inequality 
for a general case.
The following qualities which hold for any non-flat $U_i$'s are useful.

\begin{align}
\label{eq:D2+W^2}
 (\gamma D_i)^2 + W_i^2 &= -\Big\{ \frac{\nabla_i -\nabla_i^*}{2} \Big\}^2
+ \Big\{ \frac{\nabla_i + \nabla_i^*}{2} \Big\}^2
=\frac{1}{2}(\nabla_i \nabla_i^* + \nabla_i^* \nabla_i),
\\
\label{eq:DD}
 \{ \gamma D_i, \gamma D_j \}&=-\frac{c_i c_j}{4}[\nabla_i -\nabla_i^*, \nabla_j -\nabla_j^*] \;\;\;\mbox{for any $i,j$,}
\\
\label{eq:WD}
 [W_i, \gamma D_j] &=-\frac{\gamma c_j}{4}[ \nabla_i + \nabla_i^*,\nabla_j -\nabla_j^* ] \;\;\;\mbox{for any $i,j$,}
\\
\label{eq:WW}
W_i W_j & =\frac{a}{2}\nabla_i^*\nabla_i \frac{1}{2}(-\nabla_j -\nabla_j^*) \nonumber\\
&= \frac{a}{2}\nabla_i^*[ \nabla_i,  \frac{1}{2}(-\nabla_j -\nabla_j^*)]
  +\frac{a}{2}\nabla_i^* \frac{1}{2}(-\nabla_j -\nabla_j^*)\nabla_i \nonumber\\
&= 
- \frac{a}{4}R_{ij}
  +\frac{a^2}{4}\nabla_i^* \nabla_j^* \nabla_j \nabla_i.
\end{align}
Here we have used a definition
$$ R_{i j}:=\nabla_i^*( [ \nabla_i,  \nabla_j ] + [\nabla_i, \nabla_j^*])$$
and a formula $\nabla_j + \nabla_j^*=- (a/2) \nabla_j^* \nabla_j$.


Note in the above equations that (\ref{eq:D2+W^2}) and the second term of 
(\ref{eq:WW}) are semi-positive and other parts are all expressed by
the commutators among  the difference operators which vanishes in the flat case.

The following lemma holds.
\begin{lemma}
  There exists an $a$-independent constant $C_0 >0$ such that the operator norms
  are uniformly bounded by $C_0$:
$$
||[ \nabla_i,\nabla_j] ||, ||[\nabla_i,\nabla_j^*]||, ||[\nabla_i^*,\nabla_j^*]|| \leq C_0,\;\;\;i,j=1,2,\cdots n.
$$
\end{lemma}

From (\ref{eq:DD}) (\ref{eq:WD}) and (\ref{eq:WW}), 
the following corollary holds.
\begin{corollary}
\label{col:Dis}
For arbitrary $\phi \in \Gamma(E_a)$,
\begin{eqnarray*}
| \langle \phi | \{ \gamma D_i, \gamma D_j\} |\phi \rangle |  &\leq& C_0 ||\phi||^2 \quad (i<j), \\
| \langle \phi | [W_i, \gamma D_j] |\phi \rangle |  &\leq& C_0 ||\phi||^2, \\
| \langle \phi | R_{ij} | \phi \rangle |  &\leq& 2 C_0 ||\nabla_i \phi || ||\phi|| \quad (i \neq j)
\end{eqnarray*}
hold.

\end{corollary}

The first two immediately follow from (\ref{eq:DD}) and  (\ref{eq:WD}).
For the last inequality, we use a partial integration together with (\ref{eq:WW}).

Let us summarize the computation so far done.
\begin{lemma} \label{estimate of error term}
Let us decompose $(\gamma D_{W,a})^2$ as
$$
(\gamma D_{W,a})^2= \frac{1}{2}\sum_i (\nabla_i^* \nabla_i + \nabla_i \nabla_i^*)
+\frac{a^2}{4} (\nabla_i \nabla_j)^* (\nabla_i \nabla_j) +\tilde{R}
$$
where 
$$
\tilde{R}=R_{DD} +R_{WD} -(a/4)\sum_{i \neq j} R_{ij},
$$
which satisfies the following inequalities.
There exist two positive constants $a_1 >0$ and $C>0$ which were determined by
$n$ and $C_0$ only such that for any $0< a\le a_1$
$$
\langle \phi | \tilde{R} |\phi \rangle \leq \frac{1}{2} \Big(\sum_i ||\nabla_i \phi||^2  + C ||\phi||^2 \Big).
$$
\end{lemma}
\begin{proof}
From Collorary \ref{col:Dis}
we have
$$
\langle \phi | \tilde{R} | \phi \rangle \leq \frac{n(n-1)}{2} C_0 ||\phi||^2 + n^2 C_0 ||\phi||^2
+ \frac{
(n-1)
a}{2}  C_0 \sum_i ||\nabla_i \phi ||  ||\phi||.
$$
For the last term, we can use $ 2 a||\nabla_i \phi||||\phi||  \leq a^2||\nabla_i \phi||^2 +||\phi||^2 $.
  Setting $a_1$ such that $nC_0a_1^2\leq 2$ and taking $C>(7n^2/2-3n/2)C_0$,
  the lemma holds.
\end{proof}

\begin{proof}
From Lemma~\ref{estimate of error term},
$$
|| \gamma D_{W,a} \phi||^2= \frac{1}{2} \sum_i (||\nabla_i \phi ||^2 + ||\nabla_i^* \phi||^2)
+\frac{a^2}{4} ||\nabla_i \nabla_j \phi||^2 + \langle  \phi | \tilde{R} | \phi \rangle.
$$
Noting that
$\nabla_i^*=-U_i^* \nabla_i$, $||\nabla_i^* \phi|| = ||\nabla_i \phi||$ and
the second term is non-negative, we have
$$
|| \gamma D_{W,a} \phi||^2 \geq  \sum_i ||\nabla_i \phi ||^2  - | \langle  \phi | \tilde{R} | \phi \rangle |
$$
Using the estimate given in Lemma~\ref{estimate of error term}, we obtain
$$
|| \gamma D_{W,a} \phi||^2\geq \sum_i ||\nabla_i \phi ||^2
- \frac{1}{2}\big( \sum_i ||\nabla_i \phi||^2  + C ||\phi||^2 \big).
$$
\end{proof}

\end{proof}

\subsection{Assumptions for the proof by contradiction}
\label{sec:assumptions}

Suppose that Theorem~\ref{goal}  does not hold.
Then there should be a series labeled by $i=1,2,\ldots$ composed by
\begin{itemize}
 \item $a_i=1/N_i \to 0$,
 \item $(m_i, t_i) \in 
\{(m,t) : m \in X, t=1 \} \cup\{(m,t) : m \in A, 0 \leq t \leq 1 \} $, 
which is the staple-shaped region shown in Fig.~\ref{fig:mt},
 \item $(v_{a_i}, w_{a_i}) \in \HH_{a_i} \oplus L^2_1(E)$\footnote{From a general argument of regularity,
$w_{a_i}$ is an element of $L^2_1(E)$.},
\end{itemize}
which satisfy
\begin{enumerate}
\item
$$|| v_{a_i}||_{L^2}^2 + || w_{a_i}||_{L^2}^2 =1$$
and
\item
$$
\left(
    \begin{array}{cc}
        -\gamma D_{W,{a_i}}^{m_i} & t_if_{a_i}^*\\
       t_if_{a_i}   &  \gamma D_{T^n} 
    \end{array}
  \right)
\left(
    \begin{array}{c}
      v_{a_i} \\
      w_{a_i} 
    \end{array}
  \right)
=0.
$$
\end{enumerate}

Taking subsequences, we can assume without loss of generality
 that $(m_i, t_i)$ converges to  a point $(m_\infty,t_\infty)\neq (0,0)$:
$$
m_i \to m_\infty,\qquad t_i \to t_\infty.
$$

From Proposition \ref{property:fabound} and Theorem \ref{a priori estimate},
$f_{a_i} v_{a_i}, w_{a_i} \in \HH$ are uniformly $L^2_1$ finite.
Therefore taking subsequences, we can assume without loss of generality
that $f_{a_i} v_{a_i} $ weakly converges in $L^2_1(E)$ to a vector $v_\infty \in \HH$ 
and $w_{a_i}$ weakly converges in $L^2_1(E)$ to another vector $w_\infty  \in \HH$\footnote{
Note in any finite sequence in a Hilbert space that 
there exists a weakly converging subsequence. 
}.
From the Rellich theorem, the weakly converging series in the $L_1^2$ space
is a strongly converging series in $L^2$.

We also note that $t_if_{a_i}^* w_{a_i}$ has a finite norm
and so does $\gamma D_{W,{a_i}}^{m_i}v_{a_i}$.
This indicates that $v_{a_i}$ has a finite $L^2_1$ norm.

\subsection{Weak limit of the equation}

\begin{lemma}\label{weak convergence of equation}
With the assumptions made in the previous subsection for the proof by contradiction,
the following equation must hold in the continuum limit.
\begin{equation*}
 \left(
    \begin{array}{cc}
        -\gamma D^{m_\infty} & t_\infty \\
       t_\infty   &  \gamma D^{m_\infty} 
    \end{array}
  \right)
\left(
    \begin{array}{c}
       v_\infty \\
       w_\infty 
    \end{array}
  \right)
=0.
\end{equation*}

\end{lemma}
\begin{proof}
Our goal is to show the two equations,
\begin{align}
  \label{eq:weaklim}
\left\{
\begin{array}{rlcc}
 -\gamma D^{m_\infty} v_\infty & +t_\infty  w_\infty &=&0 \\
 t_\infty  v_\infty    &+ \gamma D^{m_\infty}  w_\infty &=&0
\end{array}
\right. .
\end{align}
Note that
$$
\left\{
\begin{array}{rlcc}
 -\gamma D^{m_i}_{W,a_i} v_{a_i}  & + t_i f_{a_i}^* w_{a_i} &=&0 \\
 t_i f_{a_i} v_{a_i}    &+ \gamma D^{m_i}  w_{a_i} &=&0
\end{array}
\right.
$$
holds as shown in the assumptions in Sec.\ref{sec:assumptions}.

The weak limit in $L^2$ of the second equation
is that of Eq.~(\ref{eq:weaklim}).

Multiplying $f_{a_i}$ to the first equation, we obtain an equation
in $\HH$,
\begin{equation}
  \label{eq:latconv}
 -f_{a_i} \gamma D_{W,a_i}^{m_i} v_{a_i}  +t_i  f_{a_i}f_{a_i}^* w_{a_i} =0. 
\end{equation}
 For arbitrary $\beta \in C^\infty(E)$, we have
 $$
 \langle \beta |t_i  f_{a_i}f_{a_i}^* w_{a_i}\rangle =
 t_i \langle f_{a_i}f_{a_i}^* \beta |  w_{a_i}\rangle \to
 t_\infty \langle \beta |  w_\infty\rangle,
 $$
 where the last arrow indicates the limit $t_i \to t_\infty$,
 strong convergence of $f_{a_i}f_{a_i}^* \beta \to \beta$ shown in
 the Proposition \ref{property:weakfafa*},
 and strong convergence of $ w_{a_i} \to w_\infty$ in $L^2$.
 Thus, the second term of Eq.~(\ref{eq:latconv}),
 $t_i  f_{a_i}f_{a_i}^* w_{a_i}$ weakly converges
 to $t_\infty w_\infty$.

Note for any $\beta_i \in L^2(E_{a_i})$ that
$$
|\langle \beta_i|\left(f_{a_i}^*f_{a_i} |v_{a_i}\rangle - |v_{a_i}\rangle\right)|
\le \sqrt{||\beta_i||_{L^2}^2||f_{a_i}^*f_{a_i} v_{a_i}-v_{a_i}||_{L^2}^2}\le C\sqrt{a_i||\beta_i||_{L^2}^2||v_{a_i}||_{L_1^2}^2}
$$
from the Proposition \ref{property:f*f} with a positive constant $C$.
Choosing $\beta_i = (\gamma D_{W,a_i}^{m_i})^*f_{a_i}^*\beta \in L^2(E_{a_i})$ for
a $\beta \in C^\infty(E)$, we have 
\begin{align*}
  \langle (\gamma D_{W,a_i}^{m_i})^*f_{a_i}^*\beta | f_{a_i}^*f_{a_i} v_{a_i}\rangle
  -\langle (\gamma D_{W,a_i}^{m_i})^*f_{a_i}^*\beta |v_{a_i}\rangle \to 0
\end{align*}
in the $a_i\to 0$ limit.
Since
\begin{align*}
\langle (\gamma D_{W,a_i}^{m_i})^*f_{a_i}^*\beta | f_{a_i}^*f_{a_i} v_{a_i}\rangle
&= \langle f_{a_i}(\gamma D_{W,a_i}^{m_i})^*f_{a_i}^*\beta | f_{a_i}v_{a_i}\rangle
\\
&\to \langle (\gamma D^{m_i})^* \beta | v_\infty \rangle
= \langle \beta|\gamma D^{m_i} v_\infty\rangle ,
\end{align*}
from the Proposition~\ref{property:DWweak} and strong convergence of
$f_{a_i}v_{a_i}\to v_\infty$, and
$$
\langle (\gamma D_{W,a_i}^{m_i})^*f_{a_i}^*\beta |v_{a_i}\rangle =
\langle \beta | f_{a_i} \gamma D_{W,a_i}^{m_i}v_{a_i}\rangle,
$$
the first term of Eq.~(\ref{eq:latconv}) weakly converges to
$-\gamma D^{m_i}v_\infty$. Thus, Eq.~(\ref{eq:weaklim}) holds.
\end{proof}


\begin{lemma} \label{limit of norm}
  Under the assumptions in Sec.~\ref{sec:assumptions},
$$
|| v_\infty ||_{L^2}^2 + || w_\infty||_{L^2}^2=1
$$
holds.
\end{lemma}
\begin{proof}
  Under the assumptions in Sec.~\ref{sec:assumptions},
$$|| v_{a_i}||_{L^2}^2 + || w_{a_i}||_{L^2}^2 =1.$$
  From the strong convergence of
  $w_{a_i} \to w_\infty$ in $L^2$, we have$|| w_{a_i} ||^2_{L^2} \to ||w_\infty ||_{L^2}^2$.
  Since $f_{a_i} v_{a_i}$ is uniformly finite in $L^2_1$ and from the Proposition~\ref{property:f*f}
  we observe
$$
|  ||f_{a_i} v_{a_i} ||_{L^2}^2  - || v_{a_i}||_{L^2}^2  \to 0
$$
and from the strong convergence of $f_{a_i}v_{a_i} \to v_\infty$ in $L^2$,
$$
||f_{a_i} v_{a_i} ||_{L^2}^2 \to || v_\infty ||^2
$$
holds.
\end{proof}


The proof is achieved by contradiction.

From Lemma~\ref{weak convergence of equation}
the vector $(v_\infty, w_\infty)^T$ in $L^2_1(E) \otimes \C^2$
is an element of the  kernel of the operator
$$
\left(
    \begin{array}{cc}
        -\gamma D^{m_\infty} & -t_\infty \\
       t_\infty   &  \gamma D^{m_\infty} 
    \end{array}
  \right)=
\left(
    \begin{array}{cc}
        -\gamma D_{T^n} & 0 \\
      0  & \gamma D_{T^n}
    \end{array}
  \right)
  +
   \left(
    \begin{array}{cc}
       -m_\infty \gamma  & 0 \\
      0  & m_\infty \gamma
    \end{array}
  \right)
  +
 \left(
    \begin{array}{cc}
       0  & - t_\infty \\
       t_\infty  & 0
    \end{array}
  \right),
$$
whose three terms all anticommute among one another. 

From a general argument that solutions of a smooth elliptic operator are smooth,
$(v_\infty, w_\infty)^T \in C^\infty(E) \otimes \C^2$.
On the solution, the above operator squared gives zero, too.
Since
$$
\left(
    \begin{array}{cc}
        -\gamma D^{m_\infty} & -t_\infty \\
       t_\infty   &  \gamma D^{m_\infty} 
    \end{array}
  \right) ^2=
\gamma D_{T^n}^2 \otimes
\left(
    \begin{array}{cc}
        {\rm id} & 0 \\
      0  & {\rm id}
    \end{array}
  \right)
  +
  m_\infty^2 \otimes
   \left(
    \begin{array}{cc}
       {\rm id}  & 0 \\
      0  & {\rm id}
    \end{array}
  \right)
  + t_\infty^2 \otimes
 \left(
    \begin{array}{cc}
       {\rm id}  & 0  \\
      0 & {\rm id}
    \end{array}
  \right)
$$
we have
\begin{eqnarray*}
0&=& (v_\infty,w_\infty) \left(
    \begin{array}{cc}
        -\gamma D^{m_\infty} & -t_\infty \\
       t_\infty   &  \gamma D^{m_\infty} 
    \end{array}
  \right) ^2 
  \left(
    \begin{array}{c}
       v_\infty \\
      w_\infty
    \end{array}
  \right)
\\ 
 &=&|| \gamma D_{T^n} v_\infty ||_{L^2}^2 +|| \gamma D_{T^n} w_\infty ||_{L^2}^2 + (m_\infty^2 + t_\infty^2) (||v_\infty||_{L^2}^2+||w_\infty||_{L^2}^2)\\
&=& || \gamma D_{T^n} v_\infty ||_{L^2}^2 +|| \gamma D_{T^n} w_\infty ||_{L^2}^2 + (t_\infty^2 + m_\infty^2),
\end{eqnarray*}
where we have used Lemma~\ref{limit of norm}.
This requires that $t_\infty=0$ and $m_\infty=0$, which contradicts the assumption
$(m_\infty, t_\infty)\neq (0,0)$, and proves Theorem~\ref{goal}, 
from which Theorems~\ref{goalgoal} and \ref{goalgoalgoal}  immediately follow.

\section{Relation to the index of the overlap Dirac operator}
\label{eq:relationtoov}

In lattice gauge theory, it has been known in even dimensions that the index
can be formulated by the overlap Dirac operator \cite{Neuberger:1997fp}
in the same manner as the Fredholm index.
In this section, we compare our new formulation with the
overlap Dirac index, and discuss the consistency.

The overlap Dirac operator is defined by
\[
 D_{\rm ov} := \frac{1}{a}\left[{\rm id}+\gamma {\rm sgn}(H_W(-M))\right],
\]
where $H_W(m)$ is the Hermitian massive Wilson Dirac operator
$H_W(m)= \gamma(D_W+m)$ and we have set its mass to a negatively
large value $-M=-1/a$.

From the definition, it is easy to see that 
$D_{\rm ov}$ satisfies the Ginsparg-Wilson relation \cite{Ginsparg:1981bj},
and a ``modified'' chiral symmetry \cite{Luscher:1998pqa} 
with the relation
\[
 \gamma D_{\rm ov} +  D_{\rm ov}\gamma ({\rm id}-aD_{\rm ov}) = 0,
\]
which corresponds to the anticommutation relation of the continuum Dirac operator.

The index is defined by  a modified $\Z_2$ grading operator
\[
 \Gamma := \gamma({\rm id}-aD_{\rm ov}/2) = \frac{\gamma}{2}-\frac{1}{2}{\rm sgn}(H_W(-M)).
\]
Since $\Gamma$ anticommutes with $H_{\rm ov}=\gamma D_{\rm ov}$,
every non-zero eigenmode $\phi_i$ of $H_{\rm ov}$ 
satisfies $\langle \phi_i| \Gamma \phi_i \rangle=0$,
including the case $\Gamma \phi_i =0$ (when $D_{\rm ov}\phi_i = (2/a)\phi_i $), 
which eliminates the doubler contributions.
Then the index
\[
 {\rm Tr}\Gamma = {\rm Tr}\Gamma|_{{\rm Ker}D_{\rm ov}},
\]
converges to the standard Fredholm index in the continuum theory \cite{Hasenfratz:1998ri,Neuberger:1997fp}, 
where the factor $({\rm id}-aD_{\rm ov}/2)$ plays a similar role to the
heat-kernel regulator of the trace in the conventional analysis.
It is interesting to note that $D_{\rm ov}$ is defined by
a unitary operator $u=\gamma {\rm sgn}(H_W(-M))$, which
may be useful to consider this index as an analog of
Karoubi's formulation of the $K$ group \cite{Karoubi}.
See Appendix~\ref{Karoubi}.

Noting that
$
{\rm Tr}\gamma=0
$
holds in the finite lattice Hilbert space,
the index can be expressed by
\[
  {\rm Tr}\Gamma = -\frac{1}{2}{\rm Tr} \;{\rm sgn}[H_W(-M)] = -\frac{1}{2}\eta(H_W(-M)),
  \]
  where the right-hand side is nothing but that of Theorem \ref{goalgoalgoal},
which equals to the spectral flow of
  the Hermitian Wilson Dirac operator $H_W(m)$ from $m=-M$ to $m=M$.
Namely, our formulation of the lattice Dirac operator index in this work is
consistent with the overlap Dirac index.

This equivalence between the overlap Dirac index
and the spectral flow of the massive Wilson Dirac operator
has been known in physics
\cite{Adams:1998eg,Kikukawa:1998pd,Luscher:1998kn,Fujikawa:1998if,Suzuki:1998yz}
but its mathematical relevance was not rigorously discussed.
In this work, we have given a mathematical role to
the massive Wilson Dirac operator and its Hilbert space on the lattice 
as an element of the $K^1$ group. 
In the $K^1$ group, unlike the standard definition in the $K^0$ group,
the $\Z_2$ grading operator 
or chiral symmetry play no role but 
the Dirac operator spectrum is still able to 
describe the index through the 
suspension isomorphism to the $K^0$ group.

\section{Summary and discussion}
\label{sec:summary}

In this work, we have shown that the massive Wilson Dirac operator on a lattice
can be treated as a mathematically well-defined object.
Taking its mass term as the base space of the Hilbert bundle,
we have proved that the one parameter family of the continuum
Dirac operator on a flat torus $T^n$
and the lattice approximation 
belong to the same equivalence
class as an element of the $K^1$ group, 
when the lattice spacing is sufficiently small.
This has lead to the equality between the index of 
the original continuum massless Dirac operator and
the $\eta$ invariant of the Wilson Dirac operator
with a negative mass.

Our formulation does not require the  $\mathbf{Z}_2$ grading 
or the chirality operator to evaluate the index,
since the $K^1$ group does not require the chiral symmetry from the beginning.
For the proof, we have directly compared the continuum
and lattice Dirac operators without passing through the geometric index.
We expect that this physicist-friendly expression using the $K^1$ group
has wider applications
than the standard Fredholm-like formulation of the index.
Here we summarize again our outlook in the two possible directions.

The first possible application is 
the case with higher symmetries.
It will be straightforward to apply our formulation
to the Hilbert space having real or quaternion structures.
The corresponding $KO$ or $KSp$ groups will be defined 
in a similar way as discussed in Sec.~\ref{sec:K-applications}.
It will also be interesting to consider 
the case with a group acting on the base manifold,
the case with symmetries with anti-linear operation, and
the family version of the index.

The second application is 
the case with boundary.
In continuum theory some of the authors and their collaborators
have shown in
the previous works \cite{Fukaya:2017tsq,Fukaya:2019qlf,Fukaya:2020tjk,Fukaya:2021sea}
that the spectral flow of a Dirac operator with a position-dependent mass term
can describe the index of the massless Dirac operator
on a manifold with boundary.
The generalization is nontrivial both physically and mathematically 
in that the boundary of the target manifold is capped to form a closed manifold,
and the domain-wall of the mass term is put instead of the original boundary.
The domain-wall is given by the sign function $\epsilon$ so that
the extended part has a positive mass, while the negative mass is assigned
for the original manifold region.
The spectral flow from a trivially massive Dirac operator (with a positive mass $+M$)
to that for the domain-wall fermion
was proven to be equal to the original index
with the APS boundary condition.
In particular, in even dimensions, the spectral flow can be
given by difference of the $\eta$ invariant of
the massive Dirac operator at $m=\pm M$:
\[
{\rm Ind}(\gamma D)|_{\rm APS} = -\left[\frac{1}{2}\eta(\gamma (D+M \epsilon) )
-\frac{1}{2}\eta(\gamma (D+M))\right],
\]
where $D$ is the original massless Dirac operator.

We call it a ``physicist-friendly'' formulation
of the index since it does not require any unphysical boundary condition 
nor exact chiral symmetry,
in contrast to the original Fredholm index
defined with a unphysical APS boundary condition to maintain the chiral symmetry.
It will be interesting to study
the lattice version of the extension to the
case with boundary. 
We refer the readers to \cite{Fukaya:2019myi} in which 
a promising result was already obtained in a perturbative analysis.

\begin{acknowledgements}
We thank Mayuko~Yamashita for her significant contribution at the early stage of this work.
We thank Sinya~Aoki, Yoshio~Kikukawa, and Yosuke~Kubota for useful discussions.
This work was partly supported by JSPS KAKENHI 
Grant Numbers JP21K03222, JP21K03574, JP22H01219, JP23K03387, JP23K22490, JP23KJ1459.
\end{acknowledgements}

\appendix

\section{Generalizations: $KO$, equivariant, and family versions}
\label{app:generalizations}

We explain how to generalize our argument to various settings.

To formulate our main theorem Theorem 3.7 we fixed our setup in Sections 3.1, 3.2 and 3.3.
The setup amounts to the following data $(E,\gamma,  \{c_i\}_{1 \leq i \leq n}, U)$, which we will call
``$K^0({\rm pt})$ version'' because Theorem 3.7 is an equality in $K^1(D^1,S^0) \cong K^0({\rm pt})$.

\paragraph{\bf Data for $K^0({\rm pt})$ version}
\begin{itemize}
\item[1]
$E$ is a smooth complex vector bundle over $T^n={\bf R}^n/{\bf Z}^n$ equipped with inner product.
\item[2]
$\gamma$ is a self-adjoint operator on $E$ satisfying $\gamma^2=1$.
\item[3]
$ \{c_k\}_{1 \leq k \leq n}$ are
self-adjoint operators on $E$  satisfying $\{ \gamma,c_k \}=0$, 
$c_k^2=1$ and $\{ c_k, c_l \}=0 (k \neq l)$.
\item[4]
$U$ is a smooth section of ${\rm Hom}(\pi_2^* E, \pi_1^* E)$ on an open neighborhood 
$W$ of the diagonal subset of $T^n \times T^n$  satisfying $U(x,x)={\rm id}$ and $U(y,x)=U(x,y)^{-1}$
if $(x,y),(y,x) \in W$. We assume that $U(x,y): E_y \to E_x$ preserves the inner products.
\end{itemize}
From this data, for sufficiently small $a=1/N$ $(N \in {\bf N}$), we have the associated link variables $\{ U_i\}_{i=1,\ldots, n}$ on the lattice $T^n_a=a{\bf Z}^n/{\bf Z}^n \subset T^n$, and the associated
continuum covariant derivatives  $\{\nabla_i^{\rm cont}\}_{1 \leq i \leq n}$ at the same time.

We can generalize or  modify the setting to various directions. For instance, we have:

\paragraph{\bf Data for $KO^{p,q}({\rm pt})$ version}  ($p,q \geq 0$)
\begin{itemize}
\item[1]
$E$ is a smooth real vector bundle over $T^n={\bf R}^n/{\bf Z}^n$ equipped with inner product.
\item[2]
$\{ \epsilon'_j\}_{0 \leq j \leq p}$ are self-adjoint operators on $E$ and
$\{ e'_i \}_{1 \leq i \leq q}$ are skew-adjoint operators on $E$ satisfying
${\epsilon'_j}^2=1$ and ${e'_j}^2=-1$. We assume that they all anticommute each other.
\item[3]
$ \{c_k\}_{1 \leq k \leq n}$ are
self-adjoint operators on $E$  satisfying $\{ c,c_k \}=0$ for $c \in \{ \epsilon'_j\}_{0 \leq j \leq p}
\cup \{ e'_i \}_{1 \leq i \leq q}$, 
$c_k^2=1$ and $\{ c_k, c_l \}=0\;(k \neq l)$.
\item[4]
$U$ is a smooth section of ${\rm Hom}(\pi_2^* E, \pi_1^* E)$ on an open neighborhood 
$W$ of the diagonal subset of $T^n \times T^n$  satisfying $U(x,x)={\rm id}$ and $U(y,x)=U(x,y)^{-1}$
if $(x,y),(y,x) \in W$. We assume that $U(x,y): E_y \to E_x$ preserves the inner products.
\end{itemize}
We put $\gamma:= \epsilon'_0$ and 
$$ 
e_j:=\gamma \epsilon'_j  \, (1\leq j \leq p), \quad \epsilon_i:=\gamma e'_i \, (1 \leq i \leq q).
$$
Note that $\gamma^2=1$, $e_j^2=-1$, $\epsilon_i^2=1$ and they all anticommute each other
\footnote{ The  conversion from  $\epsilon'_j, e'_i$ to $e_j, \epsilon_i$ looks artificial and technical,
which reflects a notational difference between mathematics and physics.
In this paper we use skew-adjoint Dirac operators following the standard convention of lattice gauge theory while we use self-adjoint Fredholm operators to formulate the $K$ groups following the standard convention
in mathematics. 
If we used self-adjoint Dirac operator  $\sum c'_k \nabla_k$ constructed from mutually anticommuting skew-adjoint operators $\{ c'_k \}$ satisfying ${c'_k}^2=-1$, then we could directly use our formulation of the $KO$ groups in terms of the data  $({\cal H}^{\bf R}, h, \epsilon_i, e_j)$. Or  if  we used the formulation of the $KO$ group by the data $({\cal H}^{\bf R}, s, \epsilon'_j, e'_i)$ with a skew adjoint Fredholm operator $s$, 
then we could directly use our skew-adjoint Dirac operator $D_{T^n}$ 
and the Wilson Dirac operator $D_{W,q}$ without  $\gamma$ factor.
}.
We write ${\cal H}^{\rm R}$ for the $L^2$-sections of $E$, and
${\cal H}_a^{\rm R}$ for the sections of $E_a:=E|_{T_a^n}$.
We define $D_{T^n}$ and $D_{W,a}$ by the same formula in the $K^0({\rm pt})$ version. 
Then  both $\gamma D_{T^n}$ and $\gamma D_{W,a}$ anticommute with each of $\{ \epsilon_i\}_{1 \leq i \leq q}$ and 
$\{ e_j \}_{1 \leq j \leq p}$.
Then our whole arguments go through for this real case with the extra Clifford action of $\{ \epsilon_i\}_{1 \leq i \leq q}$ and 
$\{ e_j \}_{1 \leq j \leq p}$, and we obtain:

\begin{theorem} {\rm ($KO^{p,q}({\rm pt})$ version for $p,q \geq 0$)} \label{theorem for p,q}

For a sufficiently small lattice spacing $a=1/N$, we have the equality
$$
[(p^*{\cal H}^{\rm R}), \gamma (D_{T^n}+m), \epsilon_1,\ldots,\epsilon_q, e_1,\ldots, e_p ] =
[(p^*{\cal H}_a^{\rm R}), \gamma (D_{W,a}+m),\epsilon_1,\ldots,\epsilon_q, e_1,\ldots, e_p ]  
$$
in the group $KO^{p,q-1}(D^1,S^0)\cong KO^{p,q}({\rm pt})$.

\end{theorem}

In the above statement we used the suspension isomorphism.

Using the construction of Lemma 2.36, we have a simple bijective correspondence between  the data for the $KO^{p,q}({\rm pt})$ version and that for the $KO^{p+1,q+1}({\rm pt})$ version in the following way.
If $(E, \{\epsilon'_j \}_{0 \leq j \leq p}, \{e'_i\}_{1\leq i \leq q}, \{c_k\}_{1 \leq k \leq n},U)$ is a data
for the $KO^{p,q}({\rm pt})$ version, then the next  $(\hat{E}, \{{\hat{\epsilon}}'_j \}_{0 \leq j \leq p+1}, \{{\hat{e}}'_i\}_{1\leq i \leq q+1}, \{\hat{c}_k\}_{1 \leq k \leq n},U)$ is a data
for the $KO^{p+1,q+1}({\rm pt})$ version.
$$
\hat{E}:=E \otimes {\bf R}^2,\,
{\hat{\epsilon}}'_{p+1} :=  {\rm id}_E \otimes 
\begin{pmatrix}
0 &  1 \\
 1 & 0 \\
\end{pmatrix}
,\,
{\hat{e}}'_{q+1}:=:{\rm id}_E \otimes 
\begin{pmatrix}
0 & - 1 \\
 1 & 0 \\
\end{pmatrix},\,
$$
$$
{\hat{\epsilon}}'_j := \epsilon'_j \otimes 
\begin{pmatrix}
 1 & 0 \\
0 & - 1 \\
\end{pmatrix},\,
{\hat{e}}'_i:= e'_i \otimes 
\begin{pmatrix}
 1 & 0 \\
0 & - 1 \\
\end{pmatrix},\,
\hat{c}_k:= c_k \otimes 
\begin{pmatrix}
 1 & 0 \\
0 & - 1 \\
\end{pmatrix}.
$$

The converse correspondence is given by taking the invariant part of the action of $\epsilon'_{p+1} e'_{q+1}$.
It is straightforward to see that 
the two mutually corresponding data give two mutually corresponding  equalities in Theorem~\ref{theorem for p,q}:
one in  $KO^{p,q}({\rm pt})$ and the other in $KO^{p+1,q+1}({\rm id})$.

We remark three consequences of this  simple algebraic correspondences.
\begin{remark}
\begin{enumerate}
\item
For any integer $n$ we can regard the ``$KO^{p,q}({\rm pt})$ versions''  for $p,q \geq 0$ satisfying $p-q=n$
as a single  ``$KO^n(\rm pt)$ version''.
\item
The $KO^{p,-1}({\rm pt})$ version, which is without any $\epsilon'_j$,   can be treated  via the $KO^{p+1,0}({\rm pt})$ version.
\item
The $K^{1}({\rm pt})$ version, whose data is $(E,  \{ c_k\}_{1 \leq k \leq n})$ without any symmetries $\epsilon'_j, e'_i$,  is treated as the complex version of the $KO^{0,-1}({\rm pt})$ version.
\end{enumerate}
\end{remark}

Using quaternionic vector bundle and quaternionic Hilbert bundle, we have the $KSp^{p,q}$ groups or the $KSp^n$ groups for $n=p-q$. 
``The $KSp$ version'' and the equality of indices are formulated and given just by replacing real vector bundle with quaternionic vector bundle, and real Hilbert bundle with quaternionic Hilbert bundle.
Since there exists a simple algebraic isomorphism $KSp^n\cong KO^{n+4}$, the formulation and statement could be reduced to
the $KO$ versions.

As generalizations of these, when an isometric involution on $T^n$ is given, we could also use two variants : Atiyah's $KR$ theory and Dupont's $KQ$ theory \cite{MR0206940,MR0254839}.

Similarly we have an equivariant version and family version for the data and the equality of indices.
We can combine  various versions to produce more variants.
Since the statements of the theorem about the equality are quite parallel
to the other versions, we just give what kind of data are necessary to formulate the equivariant version and the family version.

Let $G$ be  a finite group acting on the flat torus $T^n$.  Suppose  $G$ preserves the lattices $T^n_a$, as  sets, for infinitely many $a=1/N$. For example, the group generated by the rotation around a coordinate axis by angle $\pi/2$, or
some nonsymmorphic space group actions satisfy this condition.

\paragraph{\bf Data for $K_G^0({\rm pt})$ version} {\rm ($G$ preserves infinitely many $T^n_a$'s)}
\begin{itemize}
\item[1]
$E$ is a $G$-equivariant smooth complex vector bundle over $T^n={\bf R}^n/{\bf Z}^n$ equipped with inner product.
\item[2]
$\gamma$ is a self-adjoint $G$-equivariant operator on $E$ satisfying $\gamma^2=1$.
\item[3]
$ \{c_k\}_{1 \leq k \leq n}$ are
self-adjoint $G$-equivariant operators on $E$  satisfying $\{ \gamma,c_k \}=0$, 
$c_k^2=1$ and $\{ c_k, c_l \}=0 (k \neq l)$
\item[4]
$U$ is a $G$-invariant smooth section of ${\rm Hom}(\pi_2^* E, \pi_1^* E)$ on an open neighborhood 
$W$ of the diagonal subset of $T^n \times T^n$  satisfying $U(x,x)={\rm id}$ and $U(y,x)=U(x,y)^{-1}$
if $(x,y),(y,x) \in W$. We assume that $U(x,y): E_y \to E_x$ preserves the inner products.
\end{itemize}

Let $\Omega$ be a compact Hausdorff space.
Suppose we have a $T^n$-bundle ${\cal T}$ over $\Omega$ whose structure group is $G$ satisfying
the condition for the $K_G^0({\rm pt})$ version. We write $T^n_\omega$ for the fiber at $\omega \in \Omega$.

If we have a family of Data for $K^0({\rm pt})$ on the each fiber $T^n_\omega$ which vary continuously with respect to $\Omega$,
then we have the $K^0(\Omega)$ version.

\paragraph{\bf Data for $K^0(\Omega)$ version} {\rm  ($\Omega$ is  compact Hausdorff) }
\begin{itemize}
\item[1]
$E$ is a continuous complex vector bundle over the total space ${\cal T}$ equipped with inner product.
\item[2]
$\gamma$ is a self-adjoint operator on $E$ satisfying $\gamma^2=1$.
\item[3]
$ \{c_k\}_{1 \leq k \leq n}$ are
self-adjoint  operators on $E$  satisfying $\{ \gamma,c_k \}=0$, 
$c_k^2=1$ and $\{ c_k, c_l \}=0 (k \neq l)$
\item[4]
$U$ is a section of ${\rm Hom}(\pi_2^* E, \pi_1^* E)$ on an open neighborhood 
$W$ of the diagonal subset of ${\cal T} \times_{\Omega} {\cal T}$  satisfying $U(x,x)={\rm id}$ and $U(y,x)=U(x,y)^{-1}$
if $(x,y),(y,x) \in W_\omega$. We assume that $U(x,y): E_y \to E_x$ preserves the inner products.
\item[5]
We assume that all the data below vary smoothly along each $T^n_\omega$\footnote{See, for example, 
\cite{MR0279833} for the details of the formulation. }.
\end{itemize}
 
 As for the equivariant version and the family version,
 we have not pursued the full scope of our strategy
 due to the limitations given below.
  
 \begin{remark}
 \begin{enumerate}
 \item
Since we fix the standard flat metric on $T^n$ in this paper, 
our argument does not cover the cases for finite subgroups  of $GL(n,{\bf Z})$ which does not
preserve the standard flat metric.
\item
Our argument does not cover the cases for fiber bundles ${\cal T} \to \Omega$
whose structure group $G$ does not preserve the standard flat metric.
 \end{enumerate}
 \end{remark}


\section{Karoubi's formulation of $KO$ groups}\label{Karoubi}

We briefly discuss the relation between two formulations of the $KO$ groups: one is the formulation we use in this paper,
and the other is Karoubi's formulation \cite{Karoubi}.
There are at least two possible choices of the topology of the space of unbounded self-adjoint operators: Riesz topology and the gap topology. 
These two are loosely related to these two possible formulations of the $KO$ groups.
We also explain the Ginsparg-Wilson relation, which is known as a symmetry in lattice gauge theory, is understood in the context of Karoubi's formulation.

\subsection{Relation between the two formulations of the $KO$ groups}\label{Relation Karoubi}

Let $n$ be an integer, $X$  a compact Hausdorff space and  $A$ a closed subset of $X$. 
A standard strategy to define the group $KO^n(X,A)$ is as follows \cite{Karoubi}.
For a pair of non-negative\footnote{The formulation $(h)$ below is actually valid for $p \geq 0$ and $q \geq -1$,
while the formulation $(\epsilon)$  for $p \geq  0$ and $q \geq 0$,
and the formulation $(u)$ are for $p \geq -1$ and $q \geq 0$.} integers $p$ and $q$.
we first construct the groups $KO^{p,q}(X,A)$, 
which allow canonical functorial isomorphisms $KO^{p,q}(X,A) \stackrel{\cong}{\to} KO^{p+1,q+1}(X,A)$.
It implies that we have a well-defined group $KO^n(X,A)$ isomorphic to $KO^{p,q}(X,A)$ for $n=p-q$. 
Several equivalent constructions are known for the group $KO^{p,q}(X,A)$.
Therefore we have several distinct descriptions of the elements of $KO^{p,q}(X,A)$. 
It is easy to extend this argument to the $K$ group by ignoring the reality.

We here recall two of them, which we denote the formulations $(h)$ and $(\epsilon)$, respectively.
The formulation $(h)$ is the one we have employed in this paper. 
The formulation $(\epsilon)$ is due to Karoubi \cite{Karoubi}.
We also give a slightly modified version of  $(\epsilon)$, 
which we call the formulation $(u)$.
While the equivalence between the formulations $(\epsilon)$ and $(u)$ is apparent from their definitions,
the equivalence between the formulations $(h)$ and $(\epsilon)$ is less evident, though well known
(see \cite{Gomi} and its references, for example). 
In this Appendix we compare the data of the representative elements of $KO^{p,q}(X,A)$ 
for  the formulations $(h)$, $(\epsilon)$ and $(u)$,
but do not describe the equivalence relations among them.

In all the formulations below,  ${\cal H}^{\bf R}$ is a real Hilbert bundle over $X$.
In the formulations  ($h$) and ($\epsilon$), 
we use the orthogonal operators $\{ \epsilon_i\}_{0 \leq i \leq q}, \{e_j\}_{1 \leq j \leq p}$ on ${\cal H}^{\bf R}$ which are
mutually anticommuting  and satisfy
$\epsilon_i^2=1$ for every $i$ and $e_j^2=-1$ for every $j$.
Similarly in the formulation ($u$),
we use the operators $\{ \epsilon'_j\}_{1 \leq j \leq p},  \{e'_i\}_{0 \leq i \leq q}$  satisfying similar properties, e.g., 
${\epsilon'_j}^2=1$ and ${e'_i}^2=-1$.

\paragraph{Formulation ($h$)}
An element of $KO^{p,q}(X,A)$ is represented by 
$$({\cal H}^{\bf R}, h, \epsilon_0, \epsilon_1 \ldots, \epsilon_q,  e_1,e_2, \ldots, e_p)$$ 
which satisfies:
\begin{enumerate}

 \item $h$ is a bounded, self-adjoint and Fredholm\footnote{Alternatively we can use the (stronger) property `` $1-h^2$ is compact'' }
 operator on ${\cal H}^{\bf R}$.
 
\item $chc^{-1}=-h$ for $c=\epsilon_i \, (0 \leq i \leq q)$ and $e_j \, (1 \leq j \leq p)$.
 
 \item ${\rm Ker} \,h_x=0$ for $x \in A$.
 
\end{enumerate}

\paragraph{Formulation ($\epsilon$)}
An element of $KO^{p,q}(X,A)$ is represented by 
$$({\cal H}^{\bf R}, (\epsilon, \epsilon_0), \epsilon_1 \ldots, \epsilon_q,  e_1,e_2, \ldots, e_p)$$ 
which satisfies:
\begin{enumerate}

 \item $\epsilon$ is a self-adjoint operator on ${\cal H}^{\bf R}$ satisfying $\epsilon^2=1$ such that   
 $\epsilon +\epsilon_0$ is Fredholm\footnote{Alternatively we can use the (stronger) property ``$\epsilon-\epsilon_0$ is compact''.}.

 \item $c \epsilon c^{-1}=-\epsilon$ for  $c=\epsilon_i\, (1\leq i \leq q)$ and $c=e_j \, (1 \leq j \leq  p)$
 
 \item ${\rm Ker} \,(\epsilon+ \epsilon_0)_x=0$ for $x \in A$.
 \end{enumerate}

\paragraph{Formulation ($u$)}

An element of $KO^{p,q}(X,A)$ is represented by 
$$({\cal H}^{\bf R}, u, \epsilon'_0, \epsilon'_1 \ldots, \epsilon'_p,  e'_1,e'_2, \ldots, e'_q)$$
which satisfies:
\begin{enumerate}
\item
$u$  is an orthogonal operator on  ${\cal H}^{\bf R}$ such that
$u-1$ is Fredholm\footnote{Alternatively we can use the (stronger) properties ``$u+1$ compact''.}.
\item
$cuc^{-1}=u^{-1}$ for $c=\epsilon'_j \, (0 \leq j \leq p)$ and  $e'_i \, (1 \leq i \leq q)$
\item ${\rm Ker} \,(u-1)_x=0$ for $x \in A$.
\end{enumerate}

The equivalence between the formulations ($\epsilon$) and ($u$) is explicitly given by:
\begin{itemize}
\item  ($\epsilon$) $\Rightarrow$ ($u$):
$u=-\epsilon_0 \epsilon,\, \epsilon'_0=\epsilon_0$, 

$\epsilon'_i=\epsilon_0 e_i  (1\leq i \leq p), \, e'_j=\epsilon_0 \epsilon_j (1 \leq j \leq q)$
\item  ($u$) $\Rightarrow$ ($\epsilon$):
$\epsilon=-\epsilon_0 u, \epsilon_0=\epsilon'_0$,

 $e_i=\epsilon_0 \epsilon'_i  (1\leq i \leq p),  \epsilon_j=\epsilon_0 e_j' (1 \leq j \leq q)$
\end{itemize}
Note that, since $\epsilon+\epsilon_0=\epsilon_0(-u+1)$, the operator $\epsilon+ \epsilon_0$ is Fredholm
if and only if $u-1$ is Fredholm\footnote{Similarly $\epsilon-\epsilon_0$ is compact  if and only if $u+1$ is compact.}.

The correspondence from the data in $(h)$ to that in $(u)$ is given by
\footnote{
The spectrum of $\epsilon_0 h$ coincides with that of $i h$ with multiplicities for eigenvalues. We put an unitary operator $v=\frac{1}{\sqrt{2}}(1+i\epsilon_0) $, then $v \epsilon_0 h v^{-1}=ih$.
}
\footnote{
Fix a continuous function $\rho: i {\bf R} \to U(1)$ of the form $\rho(it)= \exp (\pi i\tau(t))$
for  a continuous non-decreasing odd function  $\tau:{\bf R} \to [-1,1]$ such that $\tau(t)$ is strictly increasing
in a neighborhood of $t=0$.
Then we have a map from the data in ($h$) to that in  $(u)$ by the relation $u:=\rho( \epsilon_0 h)$. 
In particular if we choose $\rho(it)=(1+it)/(1-it)$, then we have the above $u=(1+\epsilon_0 h)/(1-\epsilon_0 h)$. If we use the alternative property ``$1-h^2$ is compact'', instead of ``h is Fredholm'' in the formulation, then we use
$\tau$ that satisfies the extra conditions $\tau(t)=1$ for $t\geq 1$ and $\tau(t)=-1$ for $t \leq -1$, so that $u+1$ becomes compact. 
} 
$$u:=\frac{1+\epsilon_0 h}{1-\epsilon_0 h},\, 
 \epsilon'_0=\epsilon_0,\,
 \epsilon'_i=\epsilon_0 e_i\, (1\leq i \leq p), \, e'_j=\epsilon_0 \epsilon_j (1 \leq j \leq q),$$
which is not bijective.
The correspondence from the data in $(h)$ to that in ($\epsilon$) is given by
the above relation combined with the algebraic equivalence  between ($u$) and ($\epsilon$).

\subsection{Gap topology}\label{gap topology}

Fix a real Hilbert space ${\cal H}_0^{\bf R}$ and a self-adjoint operator $\epsilon_0$ on ${\cal H}_0^{\bf R}$ satisfying $\epsilon_0^2=1$.
Let ${\cal H}_0={\cal H}_0^{\bf R} \otimes {\bf C}$ be the complexification of ${\cal H}_0^{\bf R}$.
Let $\CF$ be the set of closed self-adjoint Fredholm operators on ${\cal H}_0$ with dense domain, 
which are not necessarily bounded.
Let $\UF$ be the set of unitary operators $u$ on ${\cal H}_0$ such that $u-1$ is Fredholm.
Let $\CF^{\epsilon_0}_{\bf R}$ be the subset of $\CF$ consisting of  self-adjoint operators $h$ on ${\cal H}_0^{\bf R}$ satisfying $\epsilon_0 h \epsilon_0^{-1}=-h$. 
Let $\UF^{\epsilon_0}_{\bf R}$ be the subset of $\UF$ consisting of orthogonal operators $u$ on ${\cal H}_{\bf R}$ satisfying $\epsilon_0 u \epsilon_0^{-1}=u^{-1}$. 

We write $\kappa:{\CF} \to {\UF}$ for the Cayley transform $\kappa(h)=(1+ i h)/(1- ih ) $.
We endow $\UF$ with the norm topology.  Recall that the gap topology of $\CF$ is the weakest topology of ${\CF}$ for  $\kappa$ to be continuous. We endow $\CF^{\epsilon_0}_{\bf R}$ with the restriction of the gap topology and simply call it the gap topology again.
The map $h \mapsto u$  in the previous section  is extended to unbounded operators and we have a map
$$
u: \CF^{\epsilon_0}_{\bf R} \to \UF^{\epsilon_0}_{\bf R}, \qquad h \mapsto u(h):=\frac{1+\epsilon_0 h}{1-\epsilon_0 h}.
$$

\begin{proposition}
The  weakest topology of $\CF^{\epsilon_0}_{\bf R} $ for $h \mapsto u(h)$ to be continuous is equal to the gap topology.
\end{proposition}

\begin{proof}
Both $u(h)$ and $\kappa(h)$ are unitary operators and their self-adjoint part and skew-adjoint part are proportional to each other, respectively:
$$ \frac{1}{2}(u(h)+u(h)^*)=\frac{1-h^2}{1+h^2}=\frac{1}{2}(\kappa(h)+\kappa(h)^*)$$

$$ \frac{1}{2}(u(h)-u(h)^*)=\frac{2 \epsilon_0 h}{1+h^2}, \quad \frac{1}{2}(\kappa(h)-\kappa(h)^*)=\frac{2i h}{1+h^2.}$$

Since the convergence of a sequence of unitary operators is equivalent to 
the simultaneous convergences of their self-adjoint parts and their skew-adjoint parts,
the claim follows.
\end{proof}

\subsection{Ginsparg-Willson relation}
\label{app:GWrel}
Suppose $D$ and $u$ are two bounded operators on a Hilbert space ${\cal H}_0$ which satisfy $u=1- aD$ for a positive real number $a$.
The next lemma is a consequence of a simple algebraic calculation.
\begin{lemma}
Let $c$ be either a self-adjoint operator on ${\cal H}_0$ satisfying $c^2=1$ or
 a skew-adjoint operator  on ${\cal H}_0$ satisfying $c^2=-1$.
Then the following two properties are equivalent.
\begin{enumerate}
\item
$u$ is a unitary operator satisfying $c u c^{-1}=u^{-1}$ such that $1-u$ is Fredholm.
\item
$D$ is a Fredholm operator satisfying $c D c^{-1}=D^*$ and  $D c+ c D=a D c D$.
\end{enumerate}
\end{lemma}
The first property appears in  the formulation ($u$) in the section~\ref{Karoubi}.
The second relation in the second property is called the Ginsparg-Wilson relation in  lattice gauge theory.
The lemma implies that we could rewrite the Clifford symmetry in the formulation ($u$) using the Ginsparg-Wilson relations for $c=\epsilon'_j, e'_i$.

\section{Triviality of positively massive Wilson Dirac operator}
\label{app:gapWilson}
\begin{proposition}
    For a positively massive Wilson Dirac operator $\gamma(D_{W,a}+m)=\gamma(D_a+ W+m)$ for $m>0$,
    the eta invariant $\eta(\gamma(D_{W,a}+m))$ is always zero.
\end{proposition}

\begin{proof}
  Square of the operator $[\gamma(D_{W,a}+m)]^2$ and its ``derivative'' is
  \begin{align*}
    [\gamma(D_{W,a}+m)]^2 &= -D_a^2+m^2 +2mW,\\
    \frac{\partial}{\partial m}[\gamma(D_{W,a}+m)]^2 &= 2m({\rm id} + W).
  \end{align*}
  As shown in Sec.~\ref{lemma:positivityW}, the Wilson term $W$ is a semi-positive operator,
  which guarantees that the gap of the massive Dirac operator increases as the mass.
 Since  the eta invariant at $m\to +\infty$ is zero:
\begin{align*}
 \lim_{m\to +\infty}\eta(\gamma(D_{W,a}+m))=
\lim_{m\to +\infty}{\rm Tr}\frac{\gamma(D_{W,a}+m)}{\sqrt{[\gamma(D_{W,a}+m)]^2}}
= \lim_{m\to +\infty}{\rm Tr}\frac{\gamma m}{m} = {\rm Tr}\gamma =0,
\end{align*}
the proposition above follows. 
Here, we use the fact that $\gamma $ is defined on a finite dimensional vector space.
\end{proof}

\section{Proof of the propositions in Sec.~\ref{sec:properties}}
\label{app:proofprop}

In this appendix, we present the proofs of the propositions
stated in Sec.~\ref{sec:properties}.

\subsection{Proof of Proposition~\ref{property:fabound}}
In a unit $n$-dimensional hypercube 
$$
(0,a)^n:=\{ x=(x_1,\ldots,x_n)^T : 0<x_i <a \quad (i=1, \ldots, n)\},
$$
the continuum field $\psi (x)$ with $ x\in (0,a)^n$ satisfies
$$
|\psi (x) | \leq \sum_{z_1,\ldots,z_n=0,a} |\phi(z)|
$$
as well as
$$
|\psi(x)|^2 \leq 2^n \sum_{z_1,\ldots,z_n=0,a} |\phi(z)|^2.
$$
Since the same inequality holds for any other hypercubes, and noting that each cube has 
$2^n$ vertices, we have
$$
||\psi||^2_{L^2} \leq (2^n)^2 ||\phi||^2_{L^2}.
$$
Thus, $f_a$ is $L^2$-bounded.

For $x \in (0,a)^n$, let us consider the continuum covariant derivative in the $x_1$ direction
$\nabla^{\rm cont.}_1$ of $\psi(x)$,
\begin{eqnarray*}
\nabla^{\rm cont.}_1 \psi(x)
&=&
 +a^{n-1}\sum_{z_2,\ldots, z_n=0,a } U(x,z+e_1 a)\phi(z+e_1 a) \prod_{i\neq 1} \rho_a(x_i- z_i)/a
\\
 &&-a^{n-1}\sum_{z_2,\ldots, z_n=0,a } U(x,z)\phi(z) \prod_{i\neq 1} \rho_a(x_i -z_i)/a \\
&&
+ a^n\sum_{z \in T^n_a} \rho_a(x-z)[\nabla^{\rm cont.}_1 U(x,z)] \phi(z)
\\
 &=&
 a^{n-1}\sum_{z_2,\ldots, z_n=0,a }U(x,z) (\nabla_1\phi)(z) \prod_{i\neq 1}  \rho_a(x_i -z_i)\\
&&
+ a^n\sum_{z \in T^n_a} \rho_a(x-z)[\nabla^{\rm cont.}_1 U(x,z)] \phi(z)
\\
&&
+ a^{n-1}\sum_{z_2,\ldots, z_n=0,a } \frac{U(x,z+e_1 a)-U(x,z)U(z,z+e_1 a)}{a}\phi(z+e_1 a) \prod_{i\neq 1} \rho_a(x_i- z_i)
\end{eqnarray*}
where we have set $z=(0,z_2,\ldots,z_n)$.
Since 
$U(x,z+e_1 a)-U(x,z)U(z,z+e_1 a)=U(x,z+e_1a)[{\rm id}-U_W(z+e_1a,x,z)]$,
we use Lemma~\ref{lemma:curvaturebound} to obtain an inequality 
 $$
 |\nabla^{\rm cont.}_1 \psi(x)|
  \leq \sum_{z_2,\ldots, z_n=0,a} \left[|(\nabla_1\phi)(z)   | +c_1 a |\phi(z)| +c_2 a|\phi(z+e_1a)|\right],
 $$
where $c_1$ and $c_2$ are $a$-independent positive constants.
We also have
 $$
  |\nabla^{\rm cont.}_1 \psi(x)|^2
  \leq 2 ^{n+2} \sum_{z_2,\ldots, z_n=0,a} \left[ |(\nabla_1\phi)(z)   |^2 + c_1^2 a^2 |\phi(z)|^2 + c_2^2 a^2|\phi(z+e_1a)|^2\right],
 $$
and for an integral over the unit hypercube
 $$
 \int_{x \in (0,a)^n}   |\nabla^{\rm cont.}_1 \psi(x)|^2 d^nx 
 \leq
  2 ^{n+2} a^n \sum_{z_2,\ldots, z_n=0,a} \left[  |(\nabla_1\phi)(z)   |^2 + c_1^2 a^2 |\phi(z)|^2 + c_2^2 a^2|\phi(z+e_1a)|^2\right]
 $$
Since the same inequality holds for any other hypercubes, we have
$$
||\nabla^{\rm cont.}_1 \psi||^2_{L^2} \leq (2^{n-1} 2^{n+2}) \left[||\nabla_1 \phi||^2_{L^2}+(c_1^2+c_2^2)a^2||\phi||_{L^2}^2\right],
$$
The above inequality indicates that
the restriction of the lattice field onto $L^2_1$
is preserved by $f_a$ to the $L^2_1$ range of the continuum field space.

Next we consider an equality
\begin{align*}
  \nabla_1 \phi_1(z) =& U(z,z+e_1 a)
\frac{1}{a}\int_{x} \rho_a(z + e_1a -x)U(x, z+e_1a)^{-1}\psi_1(x) d^n x
\\& -  \frac{1}{a}\int_x \rho_a(z-x) U(x,z)^{-1}\psi_1(x) d^n x
\\
 =& \frac{1}{a}\int_{x} \rho_a(z + e_1a -x) U(x,z)^{-1}\psi_1(x) d^n x
-  \frac{1}{a}\int_x \rho_a(z-x) U(x,z)^{-1}\psi_1(x) d^n x
\\&
+\int_{x} \rho_a(z + e_1a -x) \frac{U(z,z+e_1 a)U(x, z+e_1a)^{-1}-U(x,z)^{-1}}{a}\psi_1(x) d^n x
\\
=&\frac{1}{a}\int_x \rho_a(z -x)U(x,z)^{-1}\left[U(x, x+e_1a)\psi_1(x+e_1a) -\psi_1(x)\right] d^n x
\\&
+\int_{x} \rho_a(z -x) \frac{U(z,z+e_1 a)U(x+e_1a, z+e_1a)^{-1}-U(x+e_1a,z)^{-1}}{a}\psi_1(x+e_1a) d^n x
\\&
+\int_{x} \rho_a(z -x) \frac{U(x+e_1a,z)^{-1}-U(x,z)^{-1}U(x, x+e_1a)}{a}\psi_1(x+e_1a) d^n x
\\
=&\int_x \rho_a(z -x)U(x,z)^{-1}\nabla'_1 \psi_1(x) d^n x
\\&
+\int_{x} \rho_a(z -x) \frac{U(z,z+e_1 a)U(x+e_1a, z+e_1a)^{-1}-U(x,z)^{-1}U(x, x+e_1a)}{a}\psi_1(x+e_1a) d^n x
\end{align*}
where we have defined $\nabla'_1 \psi_1(x) = \left[U(x, x+e_1a)\psi_1(x+e_1a) -\psi_1(x)\right]/a$.

Using Lemma~\ref{lemma:curvaturebound} to evaluate the operator norm of
\[
U(z,z+e_1 a)U(x+e_1a, z+e_1a)^{-1}-U(x,z)^{-1}U(x, x+e_1a)=[U_W(z,z+e_1a,x+e_1a)-U_W(z,x,x+e_1a)]U(z,x+e_1a),
\]
and the Cauchy-Schwarz inequality, 
we have
 \begin{align*}
   a^n | \nabla_1 \phi_1(z) |^2 
 &\leq 2^n \int_{x\in c(z)} \left[|\nabla'_1 \psi_1(x)+ c_3 a \psi_1(x+e_1a)|^2\right] d^n x,
\\
&\leq 2^{n+1} \int_{x\in c(z)} \left[|\nabla'_1 \psi_1(x)|^2+ c_3^2 a^2 |\psi_1(x+e_1a)|^2\right] d^n x,
\\
&\leq 2^{n+
2
} \int_0^1 dt \int_{x\in c(z)} \left[|\nabla^{\rm cont.}_1 \psi_1(x+te_1a)|^2
+ c_3^{\prime 2} a^2 |\psi_1(x+te_1a)|^2
+c_3^2 a^2 |\psi_1(x+e_1a)|^2\right] d^n x,
 \end{align*}
where the range $c(z)$ denotes a hyper cube with sides $2a$ whose  center is located at $z$,
and $c_3, c_3'$ are  $a$-independent positive constants. 
We have also used
 $$
 \nabla'_1 \psi_1(x) = \int_0^1 dt \left[U(x,x+t e_1a)\nabla_1^{\rm cont.} 
+ [\nabla_{y_1}^{\rm cont.}U(y+t e_1a,x)|_{y\to x}]^{-1}\right]\psi_1(x+t e_1a)
 $$
and Lemma~\ref{lemma:curvaturebound} to obtain the bound
\[
 |\nabla'_1 \psi_1(x)|^2 < 2\int_0^1 dt 
\left[|\nabla^{\rm cont.}_1 \psi_1(x+t e_1a)|^2 + |aF \psi_1(x+t e_1a)|^2\right].
\]

Summing over the whole lattice sites $z\in T^n_a$, we obtain
 $$
 ||\nabla_1 \phi_1 ||_{L^2}^2 \leq 2^{2n+
2
}\left[|| \nabla^{\rm cont.}_1\psi_1||_{L^2}^2 + 
(c_3^2+c_3^{\prime 2})
]a^2 ||\psi_1||_{L^2}^2\right].
 $$
Thus $f_a^*$ also preserves the $L_1^2$ norm.

Finally we show that 
 $f_a$ 
is injective.
  For a open hypercube $(0,\delta a)^n$ with a fixed size $\delta$ chosen in the range
$0<\delta< 1/2$,
we have an inequality
 \begin{eqnarray*}
 |\psi(x)| &\geq&
  a^n|\rho_a(x) \phi(0)| - 
a^n\sum_{\substack{z_1\ldots, z_n =0,a \\ (z_1,\ldots, z_n) \neq (0,\ldots,0)}}
 |\rho_a(x-z) \phi(z)| \\
  &\geq&
  (1-\delta)^n |\phi(0)| -
  \delta\sum_{\substack{z_1\ldots, z_n =0,a \\ (z_1,\ldots, z_n) \neq (0,\ldots,0)}}
  |\phi(z)|.
 \end{eqnarray*}
Its square satisfies
 \begin{eqnarray*}
  |\psi(x)|^2 &\geq& 
  (1-\delta)^{2n}|\phi(0)|^2 -
  2(1-\delta)^n \delta|\phi(0)| 
  \sum_{\substack{z_1\ldots, z_n =0,a \\ (z_1,\ldots, z_n) \neq (0,\ldots,0)}} |\phi(z)| \\
  & \geq &
  (1-\delta)^{2n}|\phi(0)|^2 -
  (1-\delta)^n 
 \delta\sum_{\substack{z_1\ldots, z_n =0,a \\ (z_1,\ldots, z_n) \neq (0,\ldots,0)}}
 (|\phi(0)|^2+|\phi(z)|^2).
 \end{eqnarray*}
 From this we obtain
 $$
 \frac{1}{(\delta a)^n}
 \int_{(0,\delta a)^n} |\psi(x)|^2 dx 
 \geq
 (1-\delta)^{2n} |\phi(0)|^2- 
 (1-\delta)^n
 \delta \sum_{\substack{z_1\ldots, z_n =0,a \\ (z_1,\ldots, z_n) \neq (0,\ldots,0)}}
 (|\phi(0)|^2+|\phi(z)|^2)
 $$
 Summing up similar inequalities for $2^n$ different hypercubes including each lattice site $z/a\in \Z^n$,
we obtain 
 $$
 \frac{1}{\delta^n}
 ||\psi||_{L^2}^2 \geq
 2^n \{ (1-\delta)^{2n} -
 2 \delta (1-\delta)^n  (2^n -1)\} ||\phi||_{L^2}^2.
 $$ 
 Taking $\delta < \frac{1}{2^{n+1}(2^n -1)}$, 
 $$
 || \psi||_{L^2}^2 \geq C ||\phi||_{L^2}^2,
 $$
 holds with a positive constant
 $C=(2 \delta)^n (1-\delta)^n [(1-\delta)^n-2 \delta(2^n -1)]$.
 Then Proposition \ref{property:fabound} follows.

 \subsection{Proof of Proposition~\ref{property:f*f}}
We consider 
\begin{align*}
f_a^*f_a v_a(z)=&
a^n\sum_{z'\in T^n_a}\int d^nx \rho_a(z-x)\rho_a(x-z')U(x,z)^{-1}U(x,z')v_a(z')\\
=& a^n\sum_{e \in B}\int d^nx \rho_a(x-z)\rho_a(x-z-ea)U(x,z)^{-1}U(x,z+ea)v_a(z+ea).
\end{align*}
From Lemma \ref{lem:B},
\begin{align*}
 v_a(z)=a^n\sum_{e \in B}\int d^nx \rho_a(x-z)\rho_a(x-z-ea)v_a(z).
\end{align*}

Taking the difference of the two, we have
\begin{align*}
 f_a^*f_a v_a(z)-v_a(z)=&
a^n\sum_{e \in B, e\neq 0}\int d^nx \rho_a(x-z)\rho_a(x-z-ea)[
U(z,z+ea)
-{\rm id}]v_a(z+ea)\\
&+ a^n\sum_{e \in B, e\neq 0}\int d^nx \rho_a(x-z)\rho_a(x-z-ea)[U(x,z)^{-1}U(x,z+ea)-
U(z,z+ea)]v_a(z+ea)
\\=&
a^n\sum_{e \in B, e\neq 0}\int d^nx \rho_a(x-z)\rho_a(x-z-ea)[a\nabla_ev_a(z)]\\
&+ a^n\sum_{e \in B, e\neq 0}\int d^nx \rho_a(x-z)\rho_a(x-z-ea)U(z,z+ea)[U_W(z+ea,z,x)-{\rm id}]v_a(z+ea)
\end{align*}
where we have defined 
\begin{align*}
\nabla_{e}=\left\{
\begin{array}{cc}
\nabla_k  & \mbox{for $e=+e_k$}\\
\nabla_k^* & \mbox{for $e=-e_k$}
\end{array}
\right..
\end{align*}

Note that $\rho_a(x-z)$ has a nonzero support in 
a hyper cube with sides $2a$ around $x=z$, we have
\[
 \int d^nx \rho_a(x-z)\rho_a(x-z-ea) < \frac{2^n}{a^n}.
\]

Using Lemma~\ref{lemma:curvaturebound},
\begin{align*}
 |f_a^*f_a v_a(z)-v_a(z)|\leq&
a 2^n \sum_{e \in B, e\neq 0}|\nabla_ev_a(z)|+ a^2 2^n\sum_{e \in B, e\neq 0}F|v_a(z+ea)|.
\end{align*}
Then Proposition \ref{property:f*f} follows.

\subsection{Proof of Proposition\ref{property:weakfafa*}}
For any $x=(x_1,x_2,\cdots, x_n) \in T^n$ there exists a 
lattice site $z=(z_1,z_2,\cdots,z_n)\in T^n_a$ which satisfies
$z_i \leq x_i < z_i+a$ and we denote the associated 
unit hypercube by $c'(x):=\prod_i [z_i, z_i+a)$.
Namely, $c'(x)$ contains $x$.
We denote a set of lattice sites $B'(x)$ by the vertices of $c'(x)$.
For a continuum field $\psi \in C^\infty(E) \subset \HH$, 
\begin{align*}
  f_af_a^*\psi(x) &= a^n \sum_{z\in T^n_a} \int d^n x' \rho_a(x-z)\rho_a(z-x')U(x,z)U(x',z)^{-1}\psi(x')\\
&= a^n \sum_{z\in B'(x)} \int d^n x' \rho_a(x-z)\rho_a(z-x')U(x,z)U(x',z)^{-1}\psi(x')
\end{align*}
Setting $x'=x+y$, and noting that the integrand is nonzero only when $|y_i|<2a$ for $i=1,\cdots n$,
we have 
\begin{align*}
=& a^n \sum_{z\in B'(x)} \int d^n y \rho_a(x-z)\rho_a(x+y-z)U(x,z)U(x+y,z)^{-1}\psi(x+y)\\
=& a^n \sum_{z\in B'(x)} \int d^n y \rho_a(x-z)\rho_a(x+y-z)U(x,z)U(x,z)^{-1}\psi(x)\\
&+ a^n \sum_{z\in B'(x)} \int d^n y \rho_a(x-z)\rho_a(x+y-z)U(x,z)[U(x+y,z)^{-1}\psi(x+y)-U(x,z)^{-1}\psi(x)]\\
=& \psi(x)
+ a^n \sum_{z\in B'(x)} \int d^n y \rho_a(x-z)\rho_a(x+y-z)y^i\nabla^{\rm cont.}_i\psi(x)+\cdots,
\end{align*}
where the residual denoted by $\cdots$ is a finite $O(a^2)$ quantity.
Therefore, the following inequality holds.
\begin{align*}
  |f_af_a^*\psi(x)-\psi(x)| \leq a c_4 |\nabla^{\rm cont.}_i\psi(x)| + c_5 a^2 |R_a(x)|.
\end{align*}
with $a$-independent positive constants $c_4$ and $c_5$
and a $C^\infty$ function $R_a(x)$, which is finite in the  $a\to 0$ limit.

\subsection{Proof of Proposition~\ref{property:DWweak}}
Let us first compute
\begin{align*}
 \nabla_i f_a^*\psi(z) &= U(z,z+e_ia)\int d^n x \rho_a(z+e_ia-x)U(x, z+e_ia)^{-1}\psi(x)/a
- \int d^n x \rho_a(z-x)U(x,z)^{-1}\psi(x)/a
\\&=
U(z,z+e_ia)\int d^n x \rho_a(z-x)[U(x+e_ia, z+e_ia)^{-1}\psi(x+e_ia)-U(z,z+e_ia)^{-1}U(x,z)^{-1}\psi(x)]/a
\\&=
\int d^n x \rho_a(z-x)U(x,z)^{-1}[U(x,z)U(z,z+e_ia)U(x+e_ia, z+e_ia)^{-1}\psi(x+e_ia)-\psi(x)]/a
\\&=
\int d^n x \rho_a(z-x)U(x,z)^{-1}\nabla_i'\psi(x)
+ a\int d^n x \rho_a(z-x)G_i(x,z)\psi(x+e_ia),
\end{align*}
where we have defined 
$$
G_i(x,z)=U(x+e_ia,z)^{-1}[U_W(x+ea,z,z+e_ia)-U_W(x+e_ia,z,x)]/a^2,
$$
which is an $O(1)$ quantity.

Then in
\begin{align*}
 f_a\nabla_i f_a^*\psi(x)&=
a^n\sum_{z\in T^n_a}\int d^n x'\rho(x-z)\rho_a(z-x')U(x,z)U(x',z)^{-1}[\nabla_i'\psi(x')+aG_i(x',z)\psi(x'+e_ia)]
\\&=
a^n\sum_{z\in T^n_a}\int d^n y\rho(x-z)\rho_a(x+y-z)U(x,z)U(x+y,z)^{-1}[\nabla_i'\psi(x+y)+aG_i(x+y,z)\psi(x+y+e_ia)]
\end{align*}
$\rho_a(x+y-z)$ has a support from a hypercube with sides $2a$ only, 
which allows us to expand the other integrand in $y$, including
\[
 \nabla_i'\psi(x+y) = \frac{U(x+y,x+y+e_ia)\psi(x+y+e_ia)-\psi(x+y)}{a}=\nabla_i^{\rm cont.}\psi(x)+\cdots,
\]
where the residual denoted by $\cdots$ is a finite $O(a)$ quantity.
A similar analysis applies to the conjugate of the backward difference, and
the operations are approximated as
\begin{align*}
 f_a\nabla_i f_a^*\psi(x)&= \nabla_i^{\rm cont.}\psi(x) + \cdots,\\
 f_a\nabla_i^* f_a^*\psi(x)&= \nabla_i^{\rm cont.*}\psi(x) + \cdots.
\end{align*}

Since the conjugate of the Wilson Dirac operator is expressed by
\[
 (D_{W,a})^* =-\sum_{i=1}^n(P_+^i \nabla_i +P_-^i \nabla_i^*),\;\;\; P_\pm^i = \frac{{\rm id}\pm c_i}{2},
\]
we have
\[
 f_aD_{W,a}^* f_a^*\psi(x)= D_{T^n}^*\psi(x) + aH\psi(x),
\]
where $H$ is the residual finite error operator with mass dimension squared,
whose leading order amplitude is determined by the curvature around $x$.


\bibliographystyle{alpha}
\bibliography{LatticeASindex.bib}

\end{document}